\documentclass[11pt]{article}

\usepackage{amsmath,amsthm,amscd,latexsym}
\usepackage{amsfonts,pdfsync,color,graphicx}
\usepackage[psamsfonts]{amssymb}
\usepackage{epsfig}
\usepackage{color}
\usepackage[noadjust]{cite}
\usepackage{accents}

\usepackage{listings}
\usepackage{multicol}
\usepackage{float}
\usepackage{hyperref}
\usepackage{tikz-cd}
\usepackage{subfigure}

\usepackage{listings}

\usepackage[top=1.5in,bottom=1.5in,left=1.3in,right=1.0in]{geometry}

\newtheoremstyle{mystyle}               
{}                
{}                
{}        
{}                
{\bfseries \itshape}       
{.}      
{ }      
{}       

\newtheorem{theorem}{Theorem}[section]
\newtheorem{proposition}[theorem]{Proposition}
\newtheorem{lemma}[theorem]{Lemma} 
\newtheorem{corollary}[theorem]{Corollary}
\theoremstyle{definition}
\newtheorem{definition}[theorem]{Definition}
\newtheorem{example}[theorem]{Example}	
\newtheorem{conjecture}[theorem]{Conjecture}	
\theoremstyle{mystyle}
\newtheorem{remark}[theorem]{Remark}
\numberwithin{equation}{section}

\title{Reflection equation with piecewise constant arguments$^{1,2}$}
\date{ }
\author{Alberto Cabada and Paula Cambeses-Franco\\
	$^1$ CITMAga, 15782, Santiago de Compostela, Galicia, Spain\\
	$^2$ Departamento de Estatística, Análise Matemática e Optimización\\
	Facultade de Matem\'aticas, Universidade de Santiago de Com\-pos\-te\-la, Spain.\\
	alberto.cabada@usc.es; paula.cambeses.franco@usc.es}
\begin{document}
	\maketitle
	\begin{abstract}
In this work, we study nonlocal differential equations with  particular focus on those with reflection in their argument and piecewise constant dependence. The approach entails deriving the explicit expression of the solution to the linear problem by constructing the corresponding Green's function, as well as developing a novel formula to delineate the set of parameters involved in the analyzed equations for which the Green's function exhibits a constant sign. Furthermore, we demonstrate the existence of solutions for nonlinear problems through the utilisation of the monotone method.

The aforementioned methodology is specifically applied to the linear problem with periodic conditions $v'(t) + mv(-t) + Mv([t]) = h(t)$ for $t \in [-T,T]$, proving several existence results for the associated nonlinear problem and precisely delimiting the region where the Green's function $H_{m,M}$ has a constant sign on its domain of definition. 

The equations studied have the potential to be applied in fields such as biomedicine or quantum mechanics. Furthermore, this work represents a significant advance, as it is, as far the authors know, the first time that equations with involution and piecewise constant arguments have been studied together.
	\end{abstract}

		\noindent{\bf AMS Subject Classifications:}  34B05, 34B08, 34B10, 34B15, 34B18, 34B27.

	\noindent{\bf Keywords:} Green’s function, equation with reflection, piecewise constant arguments, constant sign solutions, explicit expression.

	\section{Introduction}
In the expansive and continuously evolving field of mathematics, differential equations occupy a pivotal position in the modelling and comprehension of phenomena across a range of disciplines. However, the difficulties associated with modelling complex systems, where interactions are not solely local or where changes occur in discrete intervals, have prompted the development of novel classes of differential equations. Notable among these are nonlocal differential equations and differential equations with piecewise constant arguments, which have applications in diverse fields including quantum mechanics and biomedicine.

The concept of non-locality has been one of the most fascinating and controversial topics in the field of modern physics. Albert Einstein was profoundly sceptical about the notion that physical effects could be transmitted instantaneously over a distance. This scepticism is reflected in his phrase ``spooky action at a distance'' \cite{chakrabarti2021there}, which he used to describe quantum entanglement, a phenomenon in which two entangled particles can influence each other instantaneously, regardless of the distance between them. Einstein insisted that a complete theory of the universe should respect the principle of locality, where ``nothing can travel faster than light'' \cite{magueijo2011faster}.

Despite his reservations, the predictions of quantum mechanics challenged this intuition. The EPR (Einstein-Podolsky-Rosen) experiment, proposed by Einstein himself in collaboration with Boris Podolsky and Nathan Rosen in 1935 \cite{einstein1935can}, sought to demonstrate that quantum mechanics was incomplete precisely because it permitted the existence of nonlocal correlations. In his own words, ``I cannot believe that God plays dice with the universe,'' reflecting his disagreement with the probabilistic and nonlocal nature of quantum predictions. Despite Einstein's opposition, subsequent experiments, including those conducted by John Bell \cite{bell1964einstein} and Alain Aspect \cite{aspect1982experimental},  have validated the reality of quantum entanglement and established nonlocal effects as a fundamental aspect of the quantum realm.

Moving away from physics, such types of equations also appear in biomedicine. Non-locality manifests itself in the way we model complex systems that cannot be fully described by local interactions. For example, the dynamics of infectious disease transmission may involve nonlocal effects due to global human mobility and long-range interactions between individuals.

A key approach to modelling complex systems is the use of differential equations with piecewise constant arguments. These equations model situations where the dependent variables or their derivatives are evaluated at discrete points in time and/or space, usually following periodic or constant subdivisions of the domain. This approach is particularly useful for systems where events occur at regular intervals or where conditions change abruptly.

These equations are used extensively in control engineering, where systems are monitored and adjusted at regular intervals. A typical example is digital controllers that sample state variables and make corrections at discrete intervals. In biomedicine, they are used to model physiological rhythms, such as the electrical activity of the heart, which is measured at regular intervals, or in drug dosing, where doses are administered at fixed times \cite{murray2002Mathematical}. In economics, these equations are fundamental in describing phenomena where adjustments are made at specific times, such as price adjustments at regular intervals or investment decisions made at discrete points in time \cite{goodwin1951nonlinear}.

In light of the aforementioned considerations, this paper will concentrate on the study of functional equations with involution and piecewise constant dependence. Involution equations are a particular type of nonlocal equation characterized by the fact that the function composed with itself is the identity. Two of the most prominent examples of involution are reflection and inversion. The study of these differential equations commenced with Silberstein in 1940, who analyzed the resolution of the equation $v'(t) = v(\frac{1}{t})$ \cite{silberstein1940xvii}. Since then, many authors have worked on this type of problem \cite{aftabizadeh1988bounded, hendersonnontrivial, o1994existence, cabada2015differential}. On the other hand, the study of differential equations with piecewise constant arguments began in the early 1980s and has since been widely treated in the literature \cite{myshkis1977certain, torres2024oscillations,abildayeva2022multi,cabada2004green, buedo2024boundary, MR2916144}.  In their study, techniques from both differential equations and difference equations have been combined. This work constitutes, as far as we know, the first time these types of equations are analyzed together. Let's approach this problem using the theory of Green's functions and analyzing their properties. We will focus, in particular, on studying solutions with a constant sign. This phenomenon is of interest because many quantities that arise in the context of the problems modelled by the equations considered can only take non-negative values. These include, for example, pressure, power, temperature in Kelvin and the number of people affected by a disease. 

We will structure the article as follows. In Section \ref{sec2}, we will introduce the necessary preliminaries for the rest of the work. Subsequently, in Section \ref{sec3}, we will derive the expression for the Green's function for a problem that involves both involution and reflection. In Section \ref{sec4}, we will study the properties of these functions and present a new method that enables us to delineate the region where the function maintains a constant sign. This method will be applied to a previously known case on piecewise equalities. Next, in Section \ref{sec5}, we will analyze a first-order differential equation with reflection and piecewise constant arguments, focusing on the region where its Green's function exhibits a constant sign. Finally, in Section \ref{sec6}, we will apply the concepts introduced and study the existence of a solution for a nonlinear problem using the monotone method related to the existence of lower and upper solutions. We will also provide a numerical approximation of this method.

\section{Preliminaries, hypotheses and main assumptions}
\label{sec2}
The utilisation of Green's functions is an invaluable tool in the resolution of ordinary differential equations. Accordingly, we will initially present the concept of Green's functions and subsequently examine their properties.

The following general problem of order $n$ is considered
\begin{equation}\label{a1}
L_n\,v(t)=\sigma(t),\;  \textup{ a.e } t\in J,\quad V_i(v)=h_i,\; i=1,\ldots,n,
\end{equation}
along with the two-point boundary conditions
\begin{equation}
\label{a2}
V_i(v)\equiv \sum_{j=0}^{n-1}\left(\alpha_j^i\, v^{(j)}(a) + \beta_j^i\,v^{(j)}(b)\right),\qquad i=1,\ldots,n,
\end{equation}
where
\begin{equation}
\label{a3}
L_n\,v(t) \equiv v^{(n)}(t)+a_1(t)\,v^{(n-1)}(t)+\cdots+a_{n-1}(t)\,v^\prime(t)+a_n(t)\,v(t),\quad  t\in J:=[a,b],
\end{equation}
being \(\alpha_j^i\), \(\beta_j^i\), and \(h_i\) real constants for all \(i = 1, \ldots, n\) and \(j = 0, \ldots, n-1\), and \(\sigma, \text{ } a_k \in \mathcal{L}^1(J)\) for all \(k = 1, \ldots, n\), being \(\mathcal{L}^1(J)\) the set of 1-integrable functions, i.e.:
\begin{equation*}
\mathcal{L}^{1}(J)=\{f \textup{ is a Lebesgue measurable function on $J$ and } \int_{J}{|f|}< \infty\}.
\end{equation*}

In this situation, we seek solutions that belong to the space 
$$W^{n,1}(J)=\{v \in {\cal C}^{n-1}(J), \textup{ }  v^{(n-1)} \in \mathcal{AC}(J)\},$$
where $\mathcal{AC}(J)$ is the set of absolutely continuous functions on $J$. 

%

In this case, when the uniqueness of solutions of Problem \eqref{a1}--\eqref{a2} can be guaranteed, such solution can be written in the form \(v(t) = L_{n}^{-1}\sigma(t)\). It is in this context that we can refer to the Green's function, \(G: J \times J \rightarrow \mathbb{R}\), associated with the linear problem of order \(n\) \eqref{a1}--\eqref{a2}. This function, in case it exists, is unique and corresponds to the integral kernel of the inverse operator \(L_{n}^{-1}\), meaning that it satisfies
\begin{equation*}
v(t)=\int_{a}^{b}{G(t,s)\sigma(s)\mathrm{d}s}, \textup{ for all }t \in J.
\end{equation*}

In reference \cite[Section 1.4]{cabada2014green}, it is proved a number of properties that the Green's function must satisfy, which allows us to define it axiomatically as follows.

\begin{definition}
We say that $G \in \mathcal{C}^{n-2}(J \times J) \cap \mathcal{C}^{n}\left((J \times J) \setminus \{(t,t), t \in J\}\right)$ is the Green's function related to Problem \eqref{a1}--\eqref{a2} if and only if it is a solution of problem
\begin{equation*}
L_{n}(G(t,s))=0, \quad t \in J \setminus \{s\}, \quad V_i(G(\cdot,s)) = 0, \, i = 1,\ldots,n,
\end{equation*}
for any $s \in (a,b)$ fixed.
Moreover, it satisfies the following jump condition at the diagonal of the square of definition:
For each $t \in (a,b)$, there exist and are finite, the lateral limits
    $$
    \frac{\partial^{n-1}}{\partial t^{n-1}} G(t^-,t) = \frac{\partial^{n-1}}{\partial t^{n-1}} G(t,t^+) \quad \text{and} \quad \frac{\partial^{n-1}}{\partial t^{n-1}} G(t,t^-) = \frac{\partial^{n-1}}{\partial t^{n-1}} G(t^+,t),
    $$
and additionally,
    $$
    \frac{\partial^{n-1}}{\partial t^{n-1}} G(t^+,t) - \frac{\partial^{n-1}}{\partial t^{n-1}} G(t^-,t) = \frac{\partial^{n-1}}{\partial t^{n-1}} G(t,t^-) - \frac{\partial^{n-1}}{\partial t^{n-1}} G(t,t^+) = 1.
    $$
  
\label{d1}

Moreover, if the homogeneous problem \eqref{a1}--\eqref{a2} ($\sigma=0$ on $J$ and $h_{i}=0$, $i=1, \ldots, n$) has as a unique solution the trivial one, the Green's function exists and is unique. 
\end{definition}

This work will address both ordinary differential equations and nonlocal equations involving involution and piecewise constant arguments. Accordingly, we will undertake a review of the concept of involution \cite{wiener2002glimpse} and the methodology employed in the study of differential equations with involution.

\begin{definition}
Let $A \subset \mathbb{R}$ be a set containing more than one point and $f:A \rightarrow A$ a function such that $f$ is not the identity $Id$. Then, $f$ is an involution if and only if
\begin{equation*}
f^{2} \equiv f \circ f = Id \textup{ on } A
\end{equation*}
or, equivalently, if
\begin{equation*}
f = f^{-1} \textup{ on } A.
\end{equation*}
If $A = \mathbb{R}$, we say that $f$ is a strong involution.
\end{definition}

Following the theoretical framework presented in \cite[Section $3$]{cabada2013comparison} and \cite[Section $1.3.2$]{cabada2015differential}, we can transform differential equations with involution into expressions that have the same form as Problem \eqref{a1}--\eqref{a2}, which we already know how to solve.

We will attempt to study equations with reflection, specifically focusing on analyzing problems of the form
\begin{equation}
L_n\,v(t) + m\,v(-t) = \sigma(t), \; t \in \hat{J}=[-T,T], \quad V_i(v) = h_i, \; i=1,\ldots,n,
\label{cambiado}
\end{equation}
and, moreover, we will also discuss problems with piecewise constant arguments as the following ones \cite{nieto2005green, cabada2004green}:
\begin{equation}
L_n\,v(t) + M\,v([t]) = \sigma(t), \; t \in \hat{J}, \quad V_i(v) = h_i, \; i=1,\ldots,n,
\label{cambiado2}
\end{equation}
with $m$, $M \in \mathbb{R}$, $V_i$ and $L_n$ defined in \eqref{a2} and \eqref{a3}, respectively, $T>0$, and $\sigma \in \mathcal{L}^{1}(\hat{J})$. The function $[t]$ is given by
$$[t]=
\left\{
\begin{array}{rll}
n, & \mbox{\rm if}  & t \in [n,n+1), \\
-n, & \mbox{\rm if}  & t \in (-n-1,-n],
\end{array}
\right.$$
where $n \in \{0,1,2, \ldots\}$. Notice that $[t]=0$ for all $t \in (-1,1)$.

We will denote by $\Lambda$ the set of all functions $v:\hat{J} \rightarrow \mathbb{R}$ that are continuous on $\tilde{J}_{n}=[-T,[-T]) \cup [[-T],[-T]+1) \cup \ldots \cup [-2,1) \cup [-1,1) \cup [1,2) \cup \ldots \cup [[T]-1,[T]) \cup [[T],T]$, and such that, for every
\begin{equation*}
	t \in \{[-T], \ldots, -1,1, \ldots, [T]\}
\end{equation*}
for which $t^- \in \hat{J}$, $v(t^-) \in \mathbb{R}$ exists. Additionally, if $v \in \Lambda$, we understand that $v(t)=v(t^{+})$ for all $t \in \{[-T], \ldots, -1,1, \ldots, [T]\}$, $t^+ \in \hat{J}$.

\begin{remark}
	If $T \in \mathbb{N}$, then the points $[-T]$ and $[T]$ lie at the boundaries of $\hat{J}$, and therefore the limits $v([-T]^-)$ and $v([T]^+)$ have no sense. In such a case, we look for solutions $v \in \mathcal{C}([-T,-T+1]) \cap \mathcal{C} ([T-1,T])$.
\end{remark}
For all $r \in \{1,2, \ldots\}$, let $\Omega^{r}$ denote the set of all functions $v:\hat{J}  \rightarrow \mathbb{R}$ such that $v \in W^{r,1}(\hat{J})$ and $v^{(r)} \in \Lambda$. 

With all this at hand, we are able to study equations that combine involutions and piecewise constant arguments.

\section{Green's functions of equations with piecewise constant arguments and involution.}
\label{sec3}
Once the theoretical formalism of Green's functions is understood and equations with involution and with piecewise constant arguments were introduced, we are now in a position to combine problems \eqref{cambiado} and \eqref{cambiado2} and study the following case:
\begin{equation}\label{b1}
L_n\,v(t)+m\,v(-t)+M\,v([t])=\sigma(t),\;  \textup{ a.e }t\in \hat{J} ,\quad V_i(v)=h_i,\; i=1,\ldots,n,
\end{equation}
where $m$ and $M \in \mathbb{R}$, and $V_{i}$ and $L_{n}$ are defined in \eqref{a2} and \eqref{a3}.

We will say that a function $v: \hat{J} \rightarrow \mathbb{R}$ is a solution of Problem \eqref{b1} if $v \in \Omega^{n}$ and satisfies equation \eqref{b1}.

We see that, now, the expression depends on two real parameters $m$ and $M$. Throughout the rest of the work, we aim to analyze the properties of the Green's function for Problem \eqref{b1} and determine for which values of $m$ and $M$ it has constant sign on $\hat{J} \times \hat{J}$.
Now, we assume that both problems, \eqref{cambiado} and \eqref{b1}, has a unique solution for any $\sigma \in L^{1}(\hat{J})$.


%

Next, we will see how we can approach this problem by using Green's functions.

Let us assume that $G_{m}(t,s)$ is the Green's function corresponding to the involution Problem \eqref{cambiado}, and we attempt to find a new Green's function associated with the new Problem \eqref{b1}, which we will denote by $H_{m,M}(t,s)$. We are interested in expressing $H_{m,M}(t,s)$ in terms of $G_{m}(t,s)$.

By the definition of Green's function, we have that $v(t)$ is a solution of Problem \eqref{b1} if and only if
\begin{equation*}
v(t)=\int_{-T}^{T}{G_{m}(t,s)\left(\sigma(s)-Mv([s]) \right) \mathrm{d}s}.
\end{equation*}

Let $l=[t] \in \{[-T], \ldots, 0, \ldots, [T]\}$, then we can write
\begin{equation}
\begin{split}
v(l) &= \int_{-T}^{T} G_{m}(l,s) \sigma(s) \, \mathrm{d}s - M \Bigg( \int_{-T}^{[-T]} G_{m}(l,s) v([-T]) \, \mathrm{d}s 
+ \int_{[-T]}^{[-T]+1} G_{m}(l,s) v([-T+1]) \, \mathrm{d}s \\
&\quad + \cdots 
+ \int_{-1}^{0} G_{m}(l,s) v([0]) \, \mathrm{d}s+ \int_{0}^{1} G_{m}(l,s) v([0]) \, \mathrm{d}s + \int_{1}^{2} G_{m}(l,s) v([1]) \, \mathrm{d}s \\
& \quad+ \cdots 
+ \int_{[T]}^{T} G_{m}(l,s) v([T]) \, \mathrm{d}s \Bigg).
\end{split}
\label{ex1}
\end{equation}

For simplicity, we denote:
\begin{align*}
 h(l) &\equiv \int_{-T}^{T} G_{m}(l,s) \sigma(s) \, \mathrm{d}s, \\
 a_{l,[-T]} &\equiv \int_{-T}^{[-T]} G_{m}(l,s) \, \mathrm{d}s,\\ 
 a_{l,[T]} &\equiv \int_{[T]}^{T} G_{m}(l,s) \, \mathrm{d}s, \\
 a_{l,0} &\equiv \int_{\max\{-T,-1\}}^{\min\{T,1\}} G_{m}(l,s) \, \mathrm{d}s \\
 \end{align*}
and  \begin{align*}
a_{l,k} &\equiv \int_{k-1}^{k} G_{m}(l,s) \mathrm{d}s \textup{ for }k \in \{[-T]+1, \ldots , [T]-1\} \setminus \{0\}.
 \end{align*}
 
Note that, in the particular case where $T \leq 1$, we will consider:
\begin{equation*}
	a_{0,[-T]}=a_{0,[T]}=a_{0,0}=\int_{-T}^{T}{G_m(0,s) \mathrm{d}s}.
\end{equation*}
 
In this way, we can rewrite the previous equation for all $l \in \{[-T], \ldots, [T]\}$ as follows:
\begin{align*}
v(l)=h(l)&-Ma_{l,[-T]}v([-T])-Ma_{l,[-T+1]}v([-T+1])- \cdots  \\
&-Ma_{l,0}v([0])-Ma_{l,1}v([1])- \cdots -Ma_{l,[T]}v([T]).
\end{align*}
From the above, we deduce that the following matrix equation holds:
\begin{equation*}
Ac=b,
\end{equation*}
where $A$, $b$ and $c$ are given by the following expressions:
\begin{equation}
A \equiv
\begin{pmatrix}
    Ma_{[-T][-T]}+1 & Ma_{[-T][-T+1]} & \cdots & Ma_{[-T]0} & \cdots & Ma_{[-T][T]} \\
    Ma_{[-T+1][-T]} & Ma_{[-T+1][-T+1]}+1 & \cdots & Ma_{[-T+1]0} &  \cdots & Ma_{[-T+1][T]} \\
    \vdots  & \vdots & \cdots & \vdots & \ldots & \vdots \\
    \vdots  & \vdots & \cdots & \vdots & \ldots & \vdots \\
    Ma_{[T][-T]} & Ma_{[T][-T+1]} & \cdots & Ma_{[T]0} & \cdots & Ma_{[T][T]}+1
\end{pmatrix},
\label{matriz}
\end{equation}
\begin{equation*}
c \equiv 
\begin{pmatrix}
v([-T]) \\ 
\vdots \\ 
v([0]) \\
\vdots \\
v([T])
\end{pmatrix}
\textup{ and }
b \equiv
\begin{pmatrix} 
h([-T]) \\ 
\vdots \\ 
h([0]) \\
\vdots \\
h([T])
\end{pmatrix}.
\end{equation*}

When the matrix $A$ is invertible, we have that $c=A^{-1}b$. 

From now on, we will denote the elements of $A^{-1}$ as $\tilde{a}_{i,j}$. Specifically:
\begin{equation*}
A^{-1}=
\begin{pmatrix}
    \tilde{a}_{[-T][-T]} & \tilde{a}_{[-T][-T+1]} & \cdots & \tilde{a}_{[-T]0} & \cdots & \tilde{a}_{[-T][T]} \\
    \tilde{a}_{[-T+1][-T]} & \tilde{a}_{[-T+1][-T+1]} & \cdots & \tilde{a}_{[-T+1]0} &  \cdots & \tilde{a}_{[-T+1][T]} \\
    \vdots  & \vdots & \cdots & \vdots & \ldots & \vdots \\
    \vdots  & \vdots & \cdots & \vdots & \ldots & \vdots \\
    \tilde{a}_{[T][-T]} & \tilde{a}_{[T][-T+1]} & \cdots & \tilde{a}_{[T]0} & \cdots & \tilde{a}_{[T][T]}
\end{pmatrix}.
\end{equation*}
We now consider and arbitrary $T>0$ and calculate the explicit expression of $v(t)$. We start again from expression \eqref{ex1}. If we denote by $\chi_{I}(s)$ the indicator function of an interval $I$, that is:
$$\chi_{I}(s)=
\left\{
\begin{array}{lll}
1, & \mbox{\rm if}  & s \in I \\
0, & \mbox{\rm if}  & s \notin I,
\end{array}
\right.$$
we can rewrite $v(t)$ for all $t \in \hat{J}$ as follows:
\begin{align*}
v(t) &= \int_{-T}^{T} G_{m}(t,s) \sigma(s) \, \mathrm{d}s - M \Bigg( \int_{-T}^{T} G_{m}(t,s) v([-T]) \chi_{[-T,-[T]]}(s) \\
& \quad +  G_{m}(t,s) v([-T+1]) \,\chi_{(-[T],[-T+1]]}(s) + \cdots 
+ G_{m}(t,s) v([0]) \chi_{(-1,1)}(s) \\
& \quad +  G_{m}(t,s) v([1]) \chi_{[1,2)}(s)\, 
+ \cdots 
+  G_{m}(t,s) v([T]) \chi_{[[T],T]}(s) \, \mathrm{d}s \Bigg).
\end{align*}

Next, we take into account that $c=A^{-1}b$, so
\begin{align*}
v(t) &= \int_{-T}^{T} G_{m}(t,s) \sigma(s) \, \mathrm{d}s- M \Bigg( \int_{-T}^{T} G_{m}(t,s) \bigg(\tilde{a}_{[-T][-T]} \int_{-T}^{T} G_{m}([-T],s) \sigma(s) \, \mathrm{d}s + \cdots \\
&\quad + \tilde{a}_{[-T]0} \int_{-T}^{T} G_{m}(0,s) \sigma(s) \, \mathrm{d}s + \cdots + \tilde{a}_{[-T][T]} \int_{-T}^{T} G_{m}([T],s) \sigma(s) \, \mathrm{d}s \bigg) \chi_{[-T,-[T]]}(s) \\
&\quad  + \cdots + G_{m}(t,s) \bigg(\tilde{a}_{0[-T]} \int_{-T}^{T} G_{m}([-T],s) \sigma(s) \, \mathrm{d}s + \cdots \\
&\quad + \tilde{a}_{00} \int_{-T}^{T} G_{m}(0,s) \sigma(s) \, \mathrm{d}s + \cdots + \tilde{a}_{0[T]} \int_{-T}^{T} G_{m}([T],s) \sigma(s) \, \mathrm{d}s \bigg) \chi_{(-1,1)}(s)+ \cdots \\
&\quad + G_{m}(t,s) \bigg(\tilde{a}_{[T][-T]} \int_{-T}^{T} G_{m}([-T],s) \sigma(s) \, \mathrm{d}s + \cdots \\
&\quad + \tilde{a}_{[T]0} \int_{-T}^{T} G_{m}(0,s) \sigma(s) \, \mathrm{d}s + \cdots + \tilde{a}_{[T][T]} \int_{-T}^{T} G_{m}([T],s) \sigma(s) \, \mathrm{d}s \bigg) \chi_{[[T],T]}(s) \, \mathrm{d}s \Bigg).
\end{align*}

Rearranging the previous equation, we would have that
{\fontsize{10}{11}{
\begin{align*}
v(t) &= \int_{-T}^{T} G_{m}(t,s) \sigma(s) \, \mathrm{d}s - M \Bigg( \int_{-T}^{T} G_{m}(t,s) \bigg( \int_{-T}^{T} G_{m}([-T],r) \sigma(r) \Big[ \tilde{a}_{[-T][-T]}\chi_{[-T,[-T]]}(s) + \cdots \\
 & \quad + \tilde{a}_{0[-T]}\chi_{(-1,1)}(s) + \cdots + \tilde{a}_{[T][-T]}\chi_{[[T],T]}(s)  \, \mathrm{d}r \bigg) + \cdots \\
& \quad +  \int_{-T}^{T} G_{m}(0,r) \sigma(r) \Big[ \tilde{a}_{[-T]0}\chi_{[-T,[-T]]}(s) + \cdots + \tilde{a}_{00}\chi_{(-1,1)}(s) + \cdots + \tilde{a}_{[T]0}\chi_{[[T],T]}(s) \Big] \, \mathrm{d}r  + \cdots \\
& \quad +  \int_{-T}^{T} G_{m}([T],r) \sigma(r) \Big[ \tilde{a}_{[-T][T]}\chi_{[-T,[-T]]}(s) + \cdots + \tilde{a}_{0[T]}\chi_{(-1,1)}(s) + \cdots + \tilde{a}_{[T][T]}\chi_{[[T],T]}(s) \Big] \, \mathrm{d}r \bigg) \mathrm{d}s \Bigg).
\end{align*}
}}

Now, since previous equation is fulfilled for all $\sigma \in \mathcal{L}^{1}(\hat{J} )$, we conclude that for any $k \in \{[-T], \ldots, [T]\}$, the following equality is satisfied:
{\fontsize{10}{11}{
\begin{align*}
& \int_{-T}^{T} G_{m}(t,s) \bigg( \int_{-T}^{T} G_{m}(k,r) \sigma(r) \Big[ \tilde{a}_{[-T]k}\chi_{[-T,[-T]]}(s) + \cdots + \tilde{a}_{0k}\chi_{(-1,1)}(s) + \cdots + \tilde{a}_{[T]k}\chi_{[[T],T]}(s) \Big] \, \mathrm{d}r \bigg)\mathrm{d}s \\
& = \int_{-T}^{T} \bigg(\int_{-T}^{T} G_{m}(t,s) G_{m}(k,r) \sigma(r) \Big[ \tilde{a}_{[-T]k}\chi_{[-T,[-T]]}(s) + \cdots + \tilde{a}_{0k}\chi_{(-1,1)}(s) + \cdots + \tilde{a}_{[T]k}\chi_{[[T],T]}(s) \Big] \, \mathrm{d}r \bigg) \mathrm{d}s \\
& = \int_{-T}^{T} \bigg(\int_{-T}^{T} G_{m}(t,r) G_{m}(k,s) \sigma(s) \Big[ \tilde{a}_{[-T]k}\chi_{[-T,[-T]]}(r) + \cdots + \tilde{a}_{0k}\chi_{(-1,1)}(r) + \cdots + \tilde{a}_{[T]k}\chi_{[[T],T]}(r) \Big] \, \mathrm{d}s \bigg)  \mathrm{d}r.
\end{align*}
}}

Observing the previous expression, it is easy to deduce that the Green's function corresponding to Problem \eqref{b1}, namely $H_{m,M}(t,s)$, is given by the expression
\begin{equation}
\begin{aligned}
H_{m,M}(t,s)&=G_{m}(t,s)\\
& \quad -M \int_{-T}^{T}{G_{m}(t,r) \mathrm{d}r}\bigg(G_{m}([-T],s)\Big[\tilde{a}_{[-T][-T]} \chi_{[-T,-[T]]}(r)\\
& \quad+ \cdots+ \tilde{a}_{0[-T]} \chi_{(1,1)}(s)+ \cdots  +\tilde{a}_{[T][-T]} \chi_{[[T],T]}(r) \Big] \\
& \quad + \cdots + G_{m}(0,s)\Big[\tilde{a}_{[-T]0} \chi_{[-T,-[T]]}(r)+ \cdots+ \tilde{a}_{00} \chi_{(1,1)}(s)+ \cdots +\tilde{a}_{[T]0} \chi_{[[T],T]}(r) \Big] \\
& \quad + \cdots +G_{m}([T],s)\Big[\tilde{a}_{[-T][T]} \chi_{[-T,-[T]]}(r)+ \cdots+ \tilde{a}_{0[T]} \chi_{(1,1)}(r)+\cdots +\tilde{a}_{[T][T]} \chi_{[[T],T]}(r) \Big] 
\bigg).
\end{aligned}
\label{g1}
\end{equation}
Finally, we define some new quantities with the main goal of simplifying the notation. Thus, we introduce the following functions defined on $\hat{J} $ and taking values on $\mathbb{R}$:
\begin{equation}
\begin{aligned}
\alpha_{[-T]j}(r) &= \tilde{a}_{[-T]j}\chi_{[-T,-[T]]}(r), \\
\alpha_{[-T+1]j}(r) &= \tilde{a}_{[-T+1]j}\chi_{(-[T],[-T+1]]}(r), \\
&\vdots \\
\alpha_{0j}(r) &= \tilde{a}_{0j}\chi_{\left(\max\{-T,-1\},\min\{T,1\}\right)}(r), \\
&\vdots \\
\alpha_{[T]j}(r) &= \tilde{a}_{[T]j}\chi_{[[T],T]}(r).
\end{aligned}
\label{g2}
\end{equation}
for all $j \in \{[-T],[-T+1], \ldots, 0, \ldots, [T]\}$.

Taking all of the above into account and using expressions \eqref{g1} and \eqref{g2}, we obtain the explicit expression of the function $H_{m,M}$ in terms of $G_{m}$ as follows:
\begin{equation}
H_{m,M}(t,s)=G_{m}(t,s)-M \Big[\sum_{j=[-T]}^{[T]} \sum_{i=[-T]}^{[T]}{G_{m}(j,s)\int_{-T}^{T}{G_{m}(t,r)} \alpha_{ij}(r)} \mathrm{d}r\Big].
\label{final1}
\end{equation}

Next, we will write the particular case where $T \in (0,1]$. For these values of $T$, the calculation is simpler and yields a more convenient expression to work with.

Provided that $1+M \int_{-T}^{T}{G_{m}(0,r) \mathrm{d}r} \neq 0$, we have:
\begin{equation*}
	\alpha_{0,0}(r)=\frac{1}{1+M \int_{-T}^{T}{G_m(0,r) \mathrm{d}r}}
\end{equation*}
and, following equation \eqref{final1}, we arrive at
\begin{equation*}
v(t) = \int_{-T}^{T} G_{m}(t, s) \, \sigma(s) \, \mathrm{d}s 
- M  \frac{
	\displaystyle \int_{-T}^{T} G_{m}(0, s) \, \sigma(s) \, \mathrm{d}s
}{
	1 + M \displaystyle \int_{-T}^{T} G_{m}(0, r) \, \mathrm{d}r
}  \displaystyle \int_{-T}^{T} G_{m}(t, r) \, \mathrm{d}r.
\end{equation*}
From the above, we can conclude that
\begin{equation}
	H_{m,M}(t,s) = G_{m}(t,s) - M 
	\frac{
		\displaystyle \int_{-T}^{T} G_{m}(t,r) \, \mathrm{d}r
	}{
		1 + M \displaystyle \int_{-T}^{T} G_{m}(0,r) \, \mathrm{d}r
	} \, G_{m}(0,s).
	\label{final2}
\end{equation}

As we will see next, expressions \eqref{final1} and \eqref{final2} will be of great importance, as they will allow us to deduce properties of the function  $H_{m,M}(t,s)$ from the already known Green's function $G_{m}(t,s)$.

\section{Constant sign characterization and properties of Green's functions}
\label{sec4}

Throughout this section we will develop techniques and comparison principles to obtain the explicit expression of the Green's functions related to different problems and to analyze the set of parameters involved, with a view to establishing whether they have a constant sign on the square $\hat{J}  \times \hat{J}$. 

We will start by deducing a formula that, in a particular case, allows us to establish a relationship between the Green's functions $H_{m_{0},M_{0}}$ and $H_{m_{1},M_{1}}$ associated with the problem \eqref{b1} for the parameters $m=m_{0}$ and $M=M_{0}$, and $m=m_{1}$ and $M=M_{1}$, respectively. Subsequently, a novel methodology will be devised with the objective of establishing a relationship between the parameters $m$ and $M$ associated with problem \eqref{b1} that delineates the regions where the Green's function exhibits a constant sign. Finally, we will apply this new methodology to a previously studied problem and verify that it yields the already known results.

\subsection{Relationship between Green's functions $H_{m_{0},M_{0}}$ and $H_{m_{1},M_{1}}$}
In this part, we will deduce a relationship between the Green's functions associated with Problem \eqref{b1} as a function of the parameters $m$ and $M \in \mathbb{R}$. We will follow the ideas developed in \cite{cabada2024explicit} for a general nth-order linear ordinary differential equation.

Let $M_{0}^{i}$, $M_{1}^{i}$, $m_{0}^{i}$ and $m_{1}^{i}$, $i=1, \ldots n-1$, be real constants. We consider the following two distinct problems, for which we assume that both have a unique solution for any $\sigma \in L^{1}(\hat{J})$:
\begin{equation}
L_{n}v_{0}(t)+\sum_{i=0}^{n-1}m_{0}^{i}\,v_{0}^{(i)}(-t)+\sum_{i=0}^{n-1}M_{0}^{i}\,v_{0}^{(i)}([t])=\sigma(t), \quad t \in \hat{J} , \quad V_{i}(v_{0})=0. \label{pp5}, \\
\end{equation}
and
\begin{equation}
L_{n}v_{1}(t)+\sum_{i=0}^{n-1} m_{1}^{i}\,v_{1}^{(i)}(-t)+\sum_{i=0}^{n-1} M_{1}^{i}\,v_{1}^{(i)}([t])=\sigma(t), \quad t \in \hat{J} , \quad V_{i}(v_{1})=0. \label{pp6}
\end{equation}

From the two previous expressions, we arrive at the following equality for a.e. $t \in \hat{J}$:
\begin{align*}
&L_{n}v_{0}(t)+\sum_{i=0}^{n-1} m_{1}^{i}v_{0}^{(i)}(-t)+\sum_{i=0}^{n-1} M_{1}^{i}v_{0}^{(i)}([t]) \\
& \quad =\sum_{i=0}^{n-1}(m_{1}^{i}-m_{0}^{i})v_{0}^{(i)}(-t)+\sum_{i=0}^{n-1}(M_{1}^{i}-M_{0}^{i})v_{0}^{(i)}([t])+ \sigma(t), \textup{ a.e } t \in \hat{J}, \quad V_{i}(v_{0})=0.
\end{align*}
Let $H_{m_{0}^{i},M_{0}^{i}}(t,s)$ denote the Green's function associated with Problem \eqref{pp5} and $H_{m_{1}^{i},M_{1}^{i}}(t,s)$ denote the Green's function associated with Problem \eqref{pp6}. It follows that
\begin{align*}
v_{0}(t)= \int_{-T}^{T}{H_{m_{0}^{i},M_{0}^{i}}(t,s)  \sigma(s) \mathrm{d}s}, t \in \hat{}J\\
v_{1}(t)= \int_{-T}^{T}{H_{m_{1}^{i},M_{1}^{i}}(t,s)  \sigma(s) \mathrm{d}s}, t \in \hat{J}.
\end{align*}

Therefore, we have that
\begin{align*}
v_{0}(t) &= \int_{-T}^{T} H_{m_{1}^{i},M_{1}^{i}}(t,s) \left( \sum_{i=0}^{n-1} (m_{1}^{i}-m_{0}^{i}) v_{0}^{(i)}(-s) + \sum_{i=0}^{n-1} (M_{1}^{i}-M_{0}^{i}) v_{0}^{(i)}([s]) \right) \, \mathrm{d}s  \\
& \quad+ \int_{-T}^{T} H_{m_{1}^{i},M_{1}^{i}}(t,s) \sigma(s) \, \mathrm{d}s \\
&= \sum_{i=0}^{n-1} (m_{1}^{i} - m_{0}^{i}) \int_{-T}^{T} H_{m_{1}^{i},M_{1}^{i}}(t,s) \left( \int_{-T}^{T} \frac{\partial^{i}}{\partial s^{i}} H_{m_{0}^{i},M_{0}^{i}}(-s,r) \sigma(r) \, \mathrm{d}r \right) \, \mathrm{d}s \\
&\quad + \sum_{i=0}^{n-1} (M_{1}^{i} - M_{0}^{i}) \int_{-T}^{T} H_{m_{1}^{i},M_{1}^{i}}(t,s) \left( \int_{-T}^{T} \frac{\partial^{i}}{\partial s^{i}} H_{m_{0}^{i},M_{0}^{i}}([s],r) \sigma(r) \, \mathrm{d}r \right) \, \mathrm{d}s \\
&\quad + \int_{-T}^{T} H_{m_{1}^{i},M_{1}^{i}}(t,s) \sigma(s) \, \mathrm{d}s \\
&= \sum_{i=0}^{n-1} (m_{1}^{i} - m_{0}^{i}) \int_{-T}^{T} \left( \int_{-T}^{T} H_{m_{1}^{i},M_{1}^{i}}(t,r) \frac{\partial^{i}}{\partial r^{i}} H_{m_{0}^{i},M_{0}^{i}}(-r,s) \, \mathrm{d}r \right) \sigma(s) \, \mathrm{d}s \\
&\quad + \sum_{i=0}^{n-1} (M_{1}^{i} - M_{0}^{i}) \int_{-T}^{T} \left( \int_{-T}^{T} H_{m_{1}^{i},M_{1}^{i}}(t,r) \frac{\partial^{i}}{\partial r^{i}} H_{m_{0}^{i},M_{0}^{i}}([r],s) \, \mathrm{d}r \right) \sigma(s) \, \mathrm{d}s \\
&\quad + \int_{-T}^{T} H_{m_{1}^{i},M_{1}^{i}}(t,s) \sigma(s) \, \mathrm{d}s.
\end{align*}

Thus, since previous equalities hold for any $\sigma \in \mathcal{L}^{1}(\hat{J})$, we finally arrive at
\begin{equation}
\begin{split}
H_{m_{0}^{i},M_{0}^{i}}(t,s)&=\sum_{i=0}^{n-1}(m_{1}^{i}-m_{0}^{i})\int_{-T}^{T}{H_{m_{1}^{i},M_{1}^{i}}(t,r)\frac{\partial^{i}}{\partial{r}^{i}}H_{m_{0}^{i},M_{0}^{i}}(-r,s) \mathrm{d}r} \\
& \quad +\sum_{i=0}^{n-1}(M_{1}^{i}-M_{0}^{i})\int_{-T}^{T}{H_{m_{1}^{i},M_{1}^{i}}(t,r)\frac{\partial^{i}}{\partial{r}^{i}}H_{m_{0}^{i},M_{0}^{i}}([r],s) \mathrm{d}r}+H_{m_{1}^{i},M_{1}^{i}}(t,s).
\end{split}
\label{llorocom}
\end{equation}

In the particular case when $m_{0}^{i} = m_{1}^{i}$ and $M_{0}^{i} = M_{1}^{i}$ for $i = 1, \ldots, n-1$, we will denote $m_{0}^{0} = m_{0}$, $m_{1}^{0} = m_{1}$, $M_{0}^{0} = M_{0}$, and $M_{1}^{0} = M_{1}$. Thus, we arrive at the following expression:

\begin{equation}
\begin{split}
H_{m_{0},M_{0}}(t,s)&=(m_{1}-m_{0})\int_{-T}^{T}{H_{m_{1},M_{1}}(t,r)H_{m_{0},M_{0}}(-r,s) \mathrm{d}r} \\
& \quad +(M_{1}-M_{0})\int_{-T}^{T}{H_{m_{1},M_{1}}(t,r)H_{m_{0},M_{0}}([r],s) \mathrm{d}r}+H_{m_{1},M_{1}}(t,s).
\end{split}
\label{lloro}
\end{equation}

From expression \eqref{lloro}, we can obtain a series of properties for these Green's functions.

\begin{proposition}
Let $H_{m,M}$ be the Green's function related to Problem \eqref{b1}, and consider different problems of the form \eqref{b1} by varying the value of the parameter $M \in \mathbb{R}$, with $m \in \mathbb{R}$ fixed that are uniquely solvable. Then
\begin{enumerate}
\item For all $M \in \mathbb{R}$ for which the Green's function $H_{m,M}$ is positive on $\mathring{\hat{J} } \times \mathring{\hat{J}}$, we have that $H_{m,M}$ decreases with $M$. That is, if we consider two values $M_{0}$ and $M_{1}$ such that $M_{1}>M_{0}$, $H_{m,M_{0}}>0$ and $H_{m,M_{1}}>0$ on $ \mathring{\hat{J}} \times \mathring{\hat{J}}$, then $H_{m,M_{0}}>H_{m,M_{1}}$ on $\mathring{\hat{J}} \times \mathring{\hat{J}}$.

\item For all $M \in \mathbb{R}$ for which the Green's function $H_{m,M}$ is negative on $\mathring{\hat{J}} \times \mathring{\hat{J}}$, it is fulfilled that $H_{m,M}$ decreases with $M$. That is, if we consider two values $M_{0}$ and $M_{1}$ such that $M_{1}>M_{0}$, $H_{m,M_{0}}<0$ and $H_{m,M_{1}}<0$ on $\mathring{\hat{J}} \times \mathring{\hat{J}}$, then $H_{m,M_{0}}>H_{m,M_{1}}$ on $\mathring{\hat{J}} \times \mathring{\hat{J}}$.

\item If we have two parameter values $M_{0}$ and $M_{1}$ such that $H_{m,M_{0}}<0$  and $H_{m,M_{1}}>0$ on $\mathring{\hat{J}} \times \mathring{\hat{J}}$, then necessarily $M_{1}>M_{0}$.
\end{enumerate}
\label{decrece1}
\end{proposition}

\begin{proof}
The proof of the previous proposition is immediate by observing expression \eqref{lloro}, with $m_{0}=m_{1}=m$.

For parts $1$ and $2$, we have that
\begin{equation}
H_{m,M_{0}}(t,s)=H_{m,M_{1}}(t,s)+(M_{1}-M_{0})\int_{-T}^{T}{H_{m,M_{1}}(t,r) H_{m,M_{0}}([r],s) \mathrm{d}r},
\label{emfixed}
\end{equation}
where $(M_{1}-M_{0})\int_{-T}^{T}{H_{m,M_{1}}(t,r) H_{m,M_{0}}([r],s) \mathrm{d}r}>0$ when $M_{1}>M_{0}$. Therefore, $H_{m,M_{0}}>H_{m,M_{1}}$ on $\hat{J}  \times \hat{J}$.

Part $3$ can be proved by contradiction. Suppose that $H_{m,M_{1}}<0$ for all $(t,s) \in \mathring{\hat{J} } \times \mathring{\hat{J}}$ and $H_{m,M_{0}}>0$ for all $(t,s) \in \mathring{\hat{J}} \times \mathring{\hat{J}}$ with $M_{1}>M_{0}$. Then, from \eqref{emfixed}, we have that $H_{m,M_{1}}(t,s)<0$ and $(M_{1}-M_{0})\int_{-T}^{T}{H_{m,M_{1}}(t,r) H_{m,M_{0}}([r],s) \mathrm{d}r}<0$, which implies $H_{m,M_{0}}<H_{m,M_{1}}<0$ on $\mathring{\hat{J}} \times \mathring{\hat{J}}$. This leads to a contradiction, and the proof is concluded.
\end{proof}

Analogously to previous result, we arrive at the following one:

\begin{proposition}
Let $H_{m,M}$ be the Green's function associated with Problem \eqref{b1}, and consider different problems of the form \eqref{b1} by varying the value of the parameter $m \in \mathbb{R}$, for any $M \in \mathbb{R}$ fixed that are uniquely solvable. Then:
\begin{enumerate}
\item For all $m \in \mathbb{R}$ for which the Green's function $H_{m,M}$ is positive on $\mathring{\hat{J}} \times \mathring{\hat{J}}$, we have that $H_{m,M}$ decreases with $m$. That is, if we consider two values $m_{0}$ and $m_{1}$ such that $m_{1}>m_{0}$, $H_{m_{0},M}>0$ and $H_{m_{1},M}>0$ on $\mathring{\hat{J}} \times \mathring{\hat{J}}$, then $H_{m_{0},M}>H_{m_{1},M}$ on $\mathring{\hat{J}} \times \mathring{\hat{J}}$.

\item For all $m \in \mathbb{R}$ for which the Green's function $H_{m,M}$ is negative on $\mathring{\hat{J}} \times \mathring{\hat{J}}$, it is fulfilled that $H_{m,M}$ decreases with $m$. That is, if we consider two values $m_{0}$ and $m_{1}$ such that $m_{1}>m_{0}$, $H_{m_{0},M}<0$ and $H_{m_{1},M}<0$ on $\mathring{\hat{J}} \times \mathring{\hat{J}}$, then $H_{m_{0},M}>H_{m_{1},M}$ on $\mathring{\hat{J}} \times \mathring{\hat{J}}$.

\item If we have two parameter values $m_{0}$ and $m_{1}$ such that $H_{m_{0},M}<0$ on $\mathring{\hat{J}} \times \mathring{\hat{J}}$ and $H_{m_{1},M}>0$, then necessarily $m_{1}>m_{0}$.
\end{enumerate}
\end{proposition}

\subsection{Constant sign region of $H_{m,M}$}
\label{metodofixo}
The following section will focus on an analysis of solutions with constant signs. Our objective is to delineate, as a function of the parameters $m$ and $M$, the region where the Green's function $H_{m,M}(t,s)$ assumes a positive or negative value for all $(t,s) \in \mathring{\hat{J}}  \times \mathring{\hat{J}}$. The following section outlines the steps to determine when the function is positive. The procedure is similar in the case where the objective is to ensure that the function is negative.

This procedure is valid provided that the parameters $(m, M)$ are not eigenvalues of the problem under consideration. It is therefore necessary to calculate these values in advance and, based on them, to study the values of such parameters where the Green's function is positive on $\hat{J} \times \hat{J}$.

Our main objective will be to find for each fixed $m \in \mathbb{R}$, if it exists, the biggest $M_{0}(m) \in \mathbb{R}$ such that
 $$\min_{(t,s) \in \hat{J} \times \hat{J}} H_{m,M_{0}}(t,s) = 0.$$
If for some $m \in \mathbb{R}$, the region where the Green's function is positive is non-empty, and if $\overline{M}$ satisfies $H_{m,\overline{M}} > 0$ on $\mathring{\hat{J}} \times \mathring{\hat{J}}$, then, by the decreasing property of $H_{m,M}$ with respect to $M$, as stated in Proposition \ref{decrece1}, $H_{m,M} \geq 0$ on $\mathring{\hat{J}} \times \mathring{\hat{J}}$ if $M$ is greater than the largest eigenvalue (if it exists), less than $\overline{M}$ and $M \leq M_{0}$ for some $M_{0} > \overline{M}$ (or $M$ unbounded).

From equation \eqref{emfixed}, it can be verified that the following equality holds (provided the integral is non-zero):
\begin{equation}
M_{0}=\frac{-H_{m,M_{0}}(t,s)+H_{m,M_{1}}(t,s)}{\int_{-T}^{T}{H_{m,M_{1}}(t,r)H_{m,M_{0}}([r],s)}\mathrm{d}r}+M_{1}, \quad \forall (t,s) \in \hat{J} \times \hat{J}.
\label{mcerta}
\end{equation}

Since we are looking for \( M_{0} \) such that \(\min_{(t,s) \in \hat{J} \times \hat{J}} H_{m,M_{0}}(t,s) = H_{m,M_{0}}(\hat{t},\hat{s}) = 0\) for some \((\hat{t}, \hat{s}) \in \hat{J} \times \hat{J}\), then \( M_{0} \) must satisfy:
\begin{equation}
M_{0}=\frac{H_{m,M_{1}}(\hat{t},\hat{s})}{\int_{-T}^{T}{H_{m,M_{1}}(\hat{t},r)H_{m,M_{0}}([r],\hat{s})} \mathrm{d}r}+M_{1},
\label{exprem}
\end{equation}
for all $M_{1} \in \mathbb{R}$ which is not an eigenvalue of the considered problem. In particular, we deduce that the right side of equation \eqref{exprem} is independent of $M_{1}$.
Consequently, for each fixed \( m \) and \( M_{1} \), we define the operator \( \overline{T}_{m}: \mathbb{R} \times \hat{J}  \times \hat{J} \rightarrow \mathbb{R} \) as follows:
\begin{equation}
\overline{T}_{m}(M,t,s)=\frac{H_{m,M_{1}}(t,s)}{\int_{-T}^{T}{H_{m,M_{1}}(t,r)H_{m,M}([r],s)} \mathrm{d}r}+M_{1}.
\label{opfixo}
\end{equation}
The objective would be to find the biggest \( M_{0} \) such that
\begin{equation*}
M_{0}=\overline{T}_{m}(M_{0},\hat{t},\hat{s})
\end{equation*}
where the point $(\hat{t},\hat{s}) \in \hat{J} \times \hat{J}$ satisfies that \(\min_{(t,s) \in \hat{J} \times \hat{J}} H_{m,M_{0}}(t,s) = H_{m,M_{0}}(\hat{t},\hat{s}) = 0\).

Note that the point \((\hat{t}, \hat{s})\) depends on \(M_{0}\) and \(m\), but not on \(M_{1}\).

Since the expression \eqref{exprem} is independent of the parameter \( M_{1} \), if \( M_{1} = 0 \) is not an eigenvalue of the problem under consideration, we can set \( M_{1} = 0 \) to simplify the calculations. Then, we have to look for the fixed points with respect to $M$ of operator:
\begin{equation}
\overline{T}_m^{0}(M,t,s)=\frac{G_{m}(\hat{t},\hat{s})}{\int_{-T}^{T}{G_{m}(\hat{t},r)H_{m,M}([r],\hat{s})} \mathrm{d}r}.
\label{exprem22}
\end{equation}

In most cases, it will not be easy to find the point \((\hat{t}, \hat{s}) \in \hat{J} \times \hat{J}\) where the function \(H_{m,M_{0}}\) attains its minimum on \(\hat{J} \times \hat{J}\) and such minimum is equals to \(0\). In such cases, it will be necessary to approach the problem differently.

If, for some \(m \in \mathbb{R}\), there exists a \(\overline{M} \in \mathbb{R}\) such that the Green's function \(H_{m, \overline{M}}\) is well-defined and positive on \(\hat{J} \times \hat{J}\), then \(M_{0}\) will be the value that satisfies the following equality:

\begin{equation}
	M_0 = \inf \left\{ M \;\Bigg|\; 
	M = \overline{T}_{m}^{0}(M,t,s) \text{ for some } (t,s) \in \operatorname{int}(\hat{J} \times \hat{J}) 
	\right\}.
	\label{igualdadM}
\end{equation}

It is easy to verify that, if $G_{m} \neq 0$ and $M_1=0$, expression \eqref{mcerta} can be rewritten as
\begin{equation}
M_{0}=\overline{T}_{m}^{0}(M_{0},t,s)-\frac{H_{m,M_{0}}(t,s)}{G_{m}(t,s)}\overline{T}_{m}^{0}(M_{0},t,s), \quad \forall (t,s) \in \hat{J} \times \hat{J}.
\label{maxmin}
\end{equation}
Furthermore, if the points $(t,s) \in \hat{J} \times \hat{J}$ where $\frac{\partial} {\partial{t}}H_{m,M}$ and $\frac{\partial} {\partial{s}}H_{m,M}$ with $M=M_{0}$ and $M=0$ are well defined, we can differentiate the previous expression and obtain the following equalities.

\begin{equation*}
\begin{aligned}
\frac{\partial}{\partial t} \overline{T}_{m}^{0}(M_{0}, t, s) &= \frac{\frac{\partial}{\partial t} H_{m,M_{0}}(t, s) G_{m}(t, s) - H_{m,M_{0}}(t, s) \frac{\partial}{\partial t} G_{m}(t, s)}{G_{m}(t, s)^2} \overline{T}_{m}^{0}(M_{0}, t, s) \\
& \qquad + \frac{H_{m,M_{0}}(t, s)}{G_{m}(t, s)} \frac{\partial} {\partial t}\overline{T}_{m}^{0}(M_{0}, t, s)
\end{aligned}
\end{equation*}
and
\begin{equation*}
\begin{aligned}
\frac{\partial}{\partial s} \overline{T}_{m}^{0}(M_{0}, t, s) &= \frac{\frac{\partial}{\partial s} H_{m,M_{0}}(t, s) G_{m}(t, s) - H_{m,M_{0}}(t, s) \frac{\partial}{\partial s} G_{m}(t, s)}{G_{m}(t, s)^2} \overline{T}_{m}^{0}(M_{0}, t, s) \\
& \qquad + \frac{H_{m,M_{0}}(t, s)}{G_{m}(t, s)} \frac{\partial}{\partial s} \overline{T}_{m}^{0}(M_{0}, t, s)
\end{aligned}
\end{equation*}

Then, if $(\hat{t},\hat{s}) \in \hat{J} \times \hat{J}$ satisfies $H_{m,M_{0}}(\hat{t}, \hat{s})=0$, $\frac{\partial} {\partial{t}}H_{m,M_{0}}(\hat{t},\hat{s})=0$ and $\frac{\partial}{\partial{s}}H_{m,M_{0}}(\hat{t},\hat{s})=0$, it is a critical point of $\overline{T}_{m}^{0}(M_{0},t,s)$ with respect to the variables $(t,s)$.
We can also reason in another way. As long as the denominator does not vanish and $M_0$ is not an eigenvalue, the operator $\overline{T}_m^{0}(M_0,t,s)$ is continuous with respect to $M_0$. Therefore, it follows that the first time $M_0=\overline{T}_{m}^{0}(M_0,t,s)$ occurs, it must happen at a local minimum or maximum of $\overline{T}_{m}^{0}(M_0,t,s)$ with respect to $(t,s) \in \hat{J} \times \hat{J}$.

Moreover, provided that the denominator does not vanish, and either the operators $\overline{T}_{m}^{0}(M_0,\cdot,s): \hat{J} \rightarrow \mathbb{R}$ and  $\overline{T}_{m}^{0}(M_0,t,\cdot): \hat{J} \rightarrow \mathbb{R}$ are continuous, or we are working with periodic boundary conditions (where we can ensure that the operator attains all intermediate values between the minimum and the maximum), it is easy to see that we must look for a global maximum or minimum of $\overline{T}_{m}^{0}(M_0,t,s)$ with respect to $(t,s) \in \hat{J} \times \hat{J}$. For example, in this case, if we identify a point $(t_{m}, s_{m})$ such that $M > \overline{T}_{m}^{0}(M,t_{m},s_{m})$ or, equivalently, considering equation \eqref{mcerta} with $M_1=0$,
\begin{equation*}
	\frac{H_{m,M}(t_m,s_m)}{\int_{-T}^{T}{G_{m}(t_m,r)H_{m,M}([r],s_m)} \mathrm{dr}}<0,
\end{equation*}
 $(\hat{t},\hat{s}) \in \hat{J} \times \hat{J}$ will be precisely the point where the operator $\overline{T}_{m}^{0}(M_0,t,s)$ attains its global maximum (see Figure \ref{globalmax} (a)) with respect to $(t,s) \in \hat{J} \times \hat{J}$.
 
In the case where the denominator vanishes at some point (see Figure \ref{globalmax} (b)), the first value for which $M_0=\overline{T}_m^0(M_0,t,s)$ may occur either when the denominator vanishes for some $(t,s) \in \hat{J} \times \hat{J}$, or at a point $(\hat{t},\hat{s}) \in \hat{J} \times \hat{J}$ where $\overline{T}_m^0(M_0,t,s)$ attains a local minimum, local maximum, or is non regular.

\begin{figure}[H] 
    \centering
    \subfigure[Justification for seeking the global maximum of $\overline{T}_{m}^{0}(M_{0},t,s)$, with $(t,s) \in \hat{J} \times \hat{J}$, when it is attained.]{
        \includegraphics[width=0.45\textwidth]{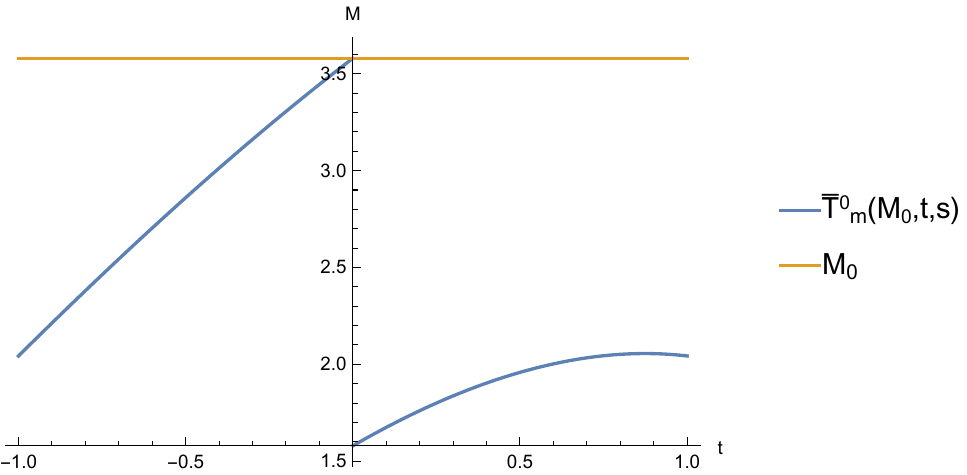}
    }
    \subfigure[Global maximum or minimum of $\overline{T}_{m}^{0}(M_{0},t,s)$, with $(t,s) \in \hat{J} \times \hat{J}$ when it is not attained.]{
        \includegraphics[width=0.45\textwidth]{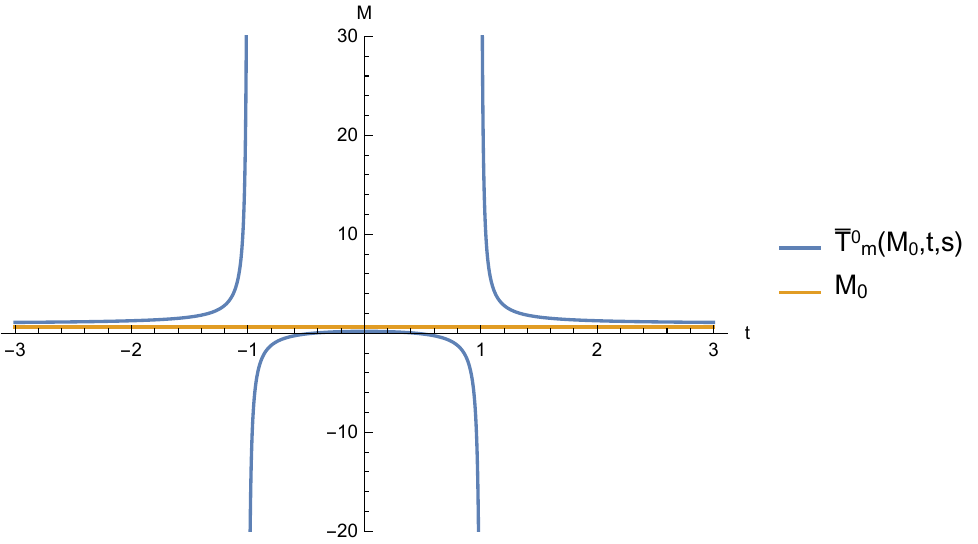}
    }
\caption{Relative maximum or minimum of $\overline{T}_{m}^{0}(M_{0},t,s)$, with $(t,s) \in \hat{J} \times \hat{J}$}
\label{globalmax}
\end{figure}

Looking for a minimum or a maximum of $\overline{T}_{m}^{0}(M_{0},t,s)$ with $(t,s) \in \hat{J} \times \hat{J}$ can be computationally much more efficient than searching for a minimum of $H_{m,M_{0}}$. This is because the integral can be divided into different intervals, requiring the evaluation of $H_{m,M_{0}}$ only at integer values of $t \in \hat{J}$, rather than for all $t \in \hat{J}$. Moreover, evaluating $H_{m,M_{0}}$ for large values of $T$ can be very costly.

Another way to solve the problem is to find $M_{0}$ that satisfies equality \eqref{igualdadM}, restricting the search to points $(t,s) \in \hat{J} \times \hat{J}$ that are critical, non-regular, or on the boundary.

Following an analogous reasoning, but starting now from the particular case of expression \eqref{lloro}:
\begin{equation*}
H_{m_{0},M}(t,s)=H_{m_{1},M}(t,s)+(m_{1}-m_{0})\int_{-T}^{T}{H_{m_{1},M}(t,r)H_{m_{0},M}(-r,s)} \mathrm{d}r,
\end{equation*}
we could find for each fixed $M \in \mathbb{R}$, if it exists, the biggest $m_{0}(M) \in \mathbb{R}$ such that
\begin{equation}
	\min_{(t,s) \in \hat{J} \times \hat{J}}{H_{m_{0},M}}(t,s)=0.
\end{equation}

%

\subsection{Application of the method to a known problem}
	
In order to verify the theoretical procedure discussed in the previous subsection, we will use the stated results to study the following problem already analyzed in the literature \cite{cabada2011first,cabada2004green}.
\begin{equation}
v'(t)+m\,v(t)+M\,v([t])=\sigma(t), \textup{ a.e. } t \in I:=[0,T], \quad v(0)=v(T),
\label{comprobar}
\end{equation}
with $m$, $M \in \mathbb{R}$, $T \in (0,1]$ and $\sigma \in \mathcal{L}^{1}(I)$.

We will say that a function $v: I \rightarrow \mathbb{R}$ is a solution of Problem \eqref{comprobar} if $v \in \Omega^1$ and satisfies equation \eqref{comprobar}.

In \cite{cabada2011first}, this problem is studied, and the following lemma is proved.
\begin{lemma}
\cite[Theorem 4.1]{cabada2011first}
Assume that \( m + M \neq 0 \). Let \(\sigma \in \mathcal{L}^{1}(I)\) be a function that is non negative on \(I\). Then, the unique solution of Problem \eqref{comprobar} with \(T = 1\) is non negative on \(I\) if and only if one of the two following conditions is satisfied:
\begin{enumerate}
\item $0 \neq -m <M \leq \frac{m}{e^{m}-1}$.
\item $0=-m<M \leq 1$.
\end{enumerate}
\label{lemcomproba}
\end{lemma}

We will attempt to reach an analogous result considering the new method previously outlined.

\begin{remark}
When considering equations with reflection, it is necessary for the solution to be defined on the interval $\hat{J}=[-T,T]$. Therefore, in the results from the previous sections, we work on the interval $\hat{J}$. However, it is easy to see that if the equation does not involve reflection, the previously mentioned results still hold valid on any interval $J=[a,b]$, and, in particular, on $I=[0,T]$. It is also possible to work with reflection on a general interval, but in this case, we need to consider a function $f:[a,b] \rightarrow [a,b]$ defined by $f(t)=a + b - t$.
\end{remark}

To begin, we consider the Green's function for Problem \eqref{comprobar} with \(M = 0\), that is,
\begin{equation}
v'(t)+m\,v(t)=\sigma(t), \textup{ a.e. } t \in I, \quad v(0)=v(T),
\label{comprobar2}
\end{equation}
with $m \in \mathbb{R} \setminus \{0\}$ and $T \in (0,1]$.

\begin{remark}
In this section, we denote by $G_{m}$ the Green's function related to Problem \eqref{comprobar2}.
\end{remark}

As it is proved in \cite{cabada2011first}, we have that the Green's function related to \eqref{comprobar2} exists and is unique if and only if $m\neq0$. In such a case, it is given by the expression
\begin{equation}
G_{m}(t,s)=\frac{1}{e^{mT}-1}\left\{ \begin{array}{lll}
e^{m(s-t+T)}, & \mbox{\rm if}  & 0\leq s< t\leq T,\\ \\
e^{m(s-t)}, & \mbox{\rm if}  &  0\leq t< s \leq T.
\end{array}\right.
\label{defg}
\end{equation}

On the other hand, following equation \eqref{final2} and using that
\begin{equation}
\int_{0}^{T}{G_{m}(t,r) \mathrm{d}r}=\frac{1}{m} \, \textup{for all} \, t \in I,
\label{now}
\end{equation}
(see \cite{cabada2014green}), we obtain that the Green's function \(H_{m,M}\) related to Problem \eqref{comprobar} follows the expression:
\begin{equation}
H_{m,M}(t,s)=G_{m}(t,s)-\frac{M}{m+M}G_{m}(0,s), \quad (t,s) \in I \times I, t \neq s.
\label{ehm}
\end{equation}
 When $m=0$ we have that these is not $G_{m}$. Despite this, by direct integration, we obtain, for $M \neq 0$ that
\begin{equation}
H_{0,M}(t,s)=\frac{1}{M\,T}\left\{ \begin{array}{lll}
1-M\,t+M\,T, & \mbox{\rm if}  & 0\leq s< t\leq T,\\ \\
1-M\,t, & \mbox{\rm if}  &  0\leq t< s \leq T.
\end{array}\right.
\label{eqM0}
\end{equation}

Following Definition \ref{d1} of the Green's function, it is not difficult to verify that if $m+M \neq 0$ then there exists a unique Green's function $H_{m,M}$ and it is characterized by the following properties:

\begin{proposition}
The Green's function \(H_{m,M}\) related to Problem \eqref{comprobar} satisfies the following properties:
\begin{enumerate}

\item \(H_{m,M}\) is defined on the square \(J \times J\) (except at \(t = s\)).

\item $\displaystyle
H_{m,M}$ and  $\displaystyle
\frac{\partial}{\partial t}H_{m,M}$ exist and are continuous on the triangles \(0 \leq s < t \leq T\) and \(0 \leq t < s \leq T\).
 
\item For each \(s \in (0,T)\), the function \(t \mapsto H_{m,M}(t,s)\) is the solution of the following differential equation on \([0,s) \cup (s,T]\). That is,
\begin{equation*}
\frac{\partial }{\partial t} H_{m,M}(t,s) + m\, H_{m,M}(t,s) + M\, H_{m,M}([t],s) = 0, \textup{ for all }t \in I \setminus \{s\}.
\end{equation*}

\item For each \( t \in (0,T) \), the one-sided limits exist and are finite:
\begin{equation*}
H_{m,M}(t^-,t) = H_{m,M}(t,t^+)\quad \text{and}\quad H_{m,M}(t,t^-) = H_{m,M}(t^+,t),
\end{equation*}
and furthermore,
\begin{equation*}
H_{m,M}(t^+,t) - H_{m,M}(t^-,t) = H_{m,M}(t,t^-) - H_{m,M}(t,t^+) = 1.
\end{equation*}

\item For each \( s \in (0,T) \), the function \( t \rightarrow H_{m,M}(t,s) \) satisfies the boundary condition
\begin{equation*}
H_{m,M}(0,s) = H_{m,M}(T,s).
\end{equation*}

\end{enumerate}
\label{carasimple}
\end{proposition}

\begin{remark}
	Notice that the previous result is valid for every $T>0$. Not only for $T \in (0,1]$.
	\end{remark}
In the sequel, we prove that $H_{m,M}$ satisfies a symmetry property.
\begin{proposition}
If $m+M \neq 0$ and $T \leq 1$ then, the following symmetry property holds:
\begin{equation*}
H_{m,M}(t,s)=-H_{-m,-M}(T-t,T-s), \textup{ for all }(t,s) \in \hat{J} \times \hat{J}, t \neq s.
\end{equation*}
\label{simetriafacil}
\end{proposition}
\begin{proof}
Let $v$ be the unique solution of Problem \eqref{comprobar} for a given function $\sigma \in \mathcal{L}^{1}(I)$. We define $r(t)=T-t$ and $w(t)=v(r(t))$. Then, from the definition of function $w$, we deduce that
\begin{equation*}
\frac{\mathrm{d}}{\mathrm{d}t}w(t)=v'(r(t)) \,r'(t)=-v'(T-t)=m\,v(T-t)+M\,v([T-t])-\sigma(T-t) \textup{ a.e } t \in I, \quad w(0)=w(T).
\end{equation*}
Assuming $T \leq 1$, it follows that $v([T-t])=v(0)=v(T)=v(T-[t])$.
Consequently, we obtain that
\begin{equation*}
w'(t)-m\,w(t)-M\,w([t])=-\sigma(T-t), \textup{ a.e. } t \in I, \quad w(0)=w(T).
\end{equation*}
Therefore, on the one hand, we would have that
\begin{equation*}
w(t)=-{\int_{0}^{T}}{H_{-m,-M}(t,s) \sigma(T-s) \mathrm{d}s}=-\int_{0}^{T}{H_{-m,-M}(t,T-r) \sigma(r) \mathrm{d}r}.
\end{equation*}
On the other hand,
\begin{equation*}
v(T-t)=\int_{0}^{T}{H_{m,M}(T-t,s) \sigma(s) \mathrm{d}s}.
\end{equation*}
From the two previous expressions and by the uniqueness of the Green's function, we deduce that
\begin{equation*}
H_{m,M}(t,s)=-H_{-m,-M}(T-t,T-s) \, \textup{ for all }t,s \in I,
\end{equation*}
and this concludes the proof.
\end{proof}

Next, we will prove a necessary condition for the Green's function to be positive.


\begin{lemma}
A necessary condition for the Green's function $H_{m,M}$ associated with Problem \eqref{comprobar} to be positive is that $m+M>0$. Similarly, a necessary condition for $H_{m,M}$ to be negative is that $m+M<0$.
\label{necesaria}
\end{lemma}

\begin{proof}
By taking $\sigma(t)=1$, the unique solution of Problem \eqref{comprobar} is given by $\frac{1}{m+M}$. It follows directly that for $H_{m,M}>0$, it is necessary that $m+M>0$, and for $H_{m,M}<0$, it is necessary that $m+M<0$.

Obviously, we also obtain that $m+M=0$ is a straight line of eigenvalues for Problem \eqref{comprobar}. In fact, from \eqref{ehm} and \eqref{eqM0}, we have that the eigenvalues are exactly such straight line.
\end{proof}

Next, let us see that we can determine the values of $(t,s) \in I \times I$ where the function $H_{m,M}$ attains its minimum when it is positive. We point out that a necessary condition is that $m+M>0$.

Considering Proposition \ref{carasimple} and $m \neq 0$, we have that
\begin{align*}
\frac{\partial }{\partial t}H_{m,M}(t,s)&=-m\,H_{m,M}(t,s)-M\,H_{m,M}(0,s) =-m\left(G_{m}(t,s)-\frac{M}{m+M}G_{m}(0,s) \right) \\
& \quad -M \left(G_{m}(0,s)-\frac{M}{m+M}G_{m}(0,s)\right)=-m\,G_{m}(t,s).
\end{align*}
Observing equation \eqref{defg}, we see that $m \, G_{m}(t,s)>0$ for all \( m \in \mathbb{R} \), $m \neq 0$ and for all \( (t,s) \in I \times I \). Therefore, for each fixed \( s \in I \), the function $t \rightarrow H_{m,M}(t,s) $ decreases with \( t \) and, since
\begin{itemize}
\item $H_{m,M}(t,s)$ is continuous except at $t = s$,
\item $H_{m,M}(0,s)=H_{m,M}(T,s)$ for all $s \in (0,T)$,
\item $H_{m,M}(t^+,t)-H_{m,M}(t^-,t)=1$ for all $t \in (0,T)$,
\end{itemize}
we deduce that the minimum can only be attained at $t = s^{-}$.

Moreover, it is easy to deduce, from expression \eqref{defg}, that
$$\frac{\partial}{\partial{s}}H_{m,M}(s^-,s)=m\,H_{m,M}(s^-,s) \textup{ }s \in (0,T).$$

Now, we define the function $q(s):= H_{m,M}(s^-,s)$, $s \in (0,T)$. As a consequence, we have that
\begin{align*}
q'(s)&=\frac{\mathrm{d}}{\mathrm{d}s}H_{m,M}(s^-,s)=\frac{\partial}{\partial{t}}H_{m,M}(s^-,s)+\frac{\partial}{\partial{s}}H_{m,M}(s^-,s)\\
&=-m\,G_{m}(s^-,s)+m\,H_{m,M}(s^-,s)=-\frac{mM}{m+M}G_{m}(0,s), \textup{ } s \in (0,T).
\end{align*}

Since $m+M>0$ and $m\,G_{m}(t,s)>0$, it is immediate to verify that $M \, q'(s)<0$ for all $M \neq 0$ and $s \in (0,T)$.

Therefore, since function $q$ is continuous on $(0,T)$, we finally obtain that
\begin{equation*}
	\min_{(t,s) \in I \times I} \displaystyle H_{m,M}(t,s) = \left\{ 
	\begin{array}{lll}
		\displaystyle \lim_{s \rightarrow 0^+} H_{m,M}(s^-,s) = \frac{m}{(-1+e^{mT})(m+M)}, & \text{if} & M<0,\\[12pt]
		\displaystyle \lim_{s \rightarrow T^-} H_{m,M}(s^-,s) = \frac{m+M-e^{mT} M}{(-1+e^{mT})(m+M)}, & \text{if} & M>0.
	\end{array}
	\right.
\end{equation*}

With all the above, we can now find, for each fixed $m$, the maximum value of $M_{0}>-m$ for which the Green's function $H_{m,M}$ is positive on $I \times I$. To do this, following equations \eqref{exprem} and \eqref{eqM0} and taking for simplicity $M_{1}=0$, we need to solve the following fixed-point problem:

\begin{equation*}
	M_{0} = \left\{
	\begin{array}{lll}
		\displaystyle \lim_{s \rightarrow 0^+}
		\frac{G_{m}(s^-,s)}{
			\displaystyle \int_{0}^{T} G_{m}(s,r) H_{m,M_{0}}([r],s) \, \mathrm{d}r}
		= m + M_{0}, & \text{if} & M_0 < 0, \\[12pt]
		\displaystyle \lim_{s \rightarrow T^-}
		\frac{G_{m}(s^-,s)}{
			\displaystyle \int_{0}^{T} G_{m}(s,r) H_{m,M_{0}}([r],s) \, \mathrm{d}r}
		= e^{-mT}(m + M_{0}), & \text{if} & M_0 > 0.
	\end{array}
	\right.
\end{equation*}

Therefore, we have that $M_{0}=\overline{T}_{m}^{0}(M_{0})$ if and only if $M_{0}=\frac{m}{e^{mT}-1}(>0)$.

We deduce that the Green's function is positive on $I \times I$ if and only if $M$ satisfies $-m<M < \frac{m}{e^{mT}-1}$ when $m \neq 0$. In the case $m=0$, by continuity, we have $0<M < 1/T$.

%
%

By using the symmetry property proved in Proposition \ref{simetriafacil}, we conclude that $H_{m,M}<0$ on $I \times I$ if and only if $\frac{m}{e^{-mT}-1}<M<-m$ if $m \neq 0$ and $M \in \left(-\frac{1}{T},0 \right)$ whenever $m=0$.

Noticed that we obtain, as a particular case $(T=1)$, Lemma \ref{lemcomproba}. See Figure \ref{todofacil}. It is important to point out that the arguments used here are completely different to the ones used in \cite{cabada2011first}.  

If we set the minimum of the function \( H_{m,M}(t,s) \) over the domain \( I \times I \) equal to zero, that is,
$$
\min_{(t,s) \in I \times I} H_{m,M}(t,s) = 0,
$$
and solve the resulting equation for \( M \), we obtain exactly the same expression as the one derived from taking the limit.

\begin{figure}[H]
	\begin{center}
	\includegraphics[scale=0.6]{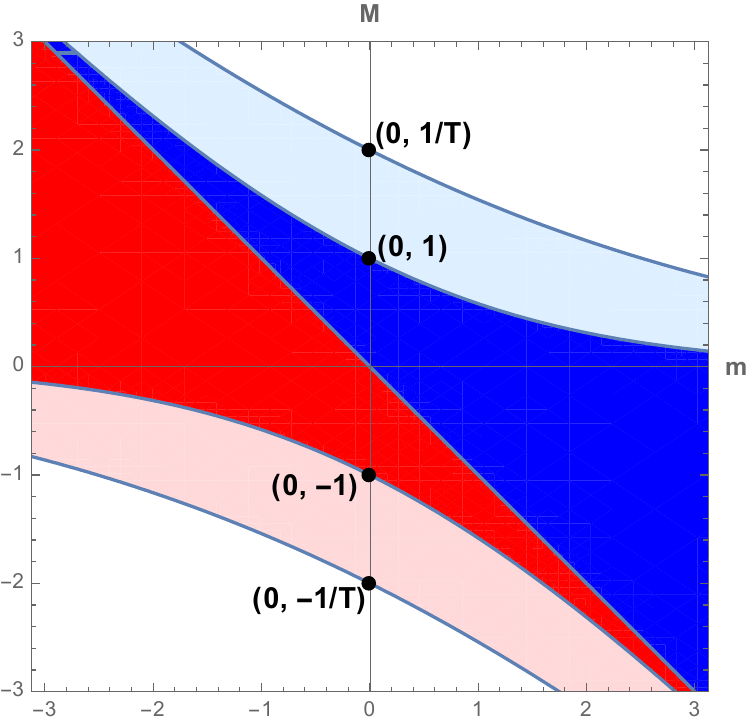}
	\end{center}
	\vspace{-20pt}
    \caption{The regions where the Green's function $H_{m,M}$ is positive are shown in light blue for $T=1/2$ and in dark blue for $T=1$. Similarly, the regions where it is negative are shown in light red for $T=1/2$ and in dark red for $T=1$.}
	\label{todofacil}
\end{figure}

\section{Periodic first order problems with involution and piecewise constant arguments}
\label{sec5}
In this section, we will focus on solving a first order differential equation with a reflection part and a piecewise constant dependence function. We will work with periodic boundary conditions. Specifically, we will consider the following functional differential equation:
\begin{equation}
v'(t)+mv(-t)+Mv([t])=h(t), \textup{ a.e. } t \in \hat{J}=[-T,T], \quad v(T)=v(-T).
\label{im1}
\end{equation}
Here \( m \) and \( M \) are real constants with the condition that both are not simultaneously zero, \( T >0 \) and \( h \in \mathcal{L}^1(\hat{J}) \).

We will say that a function $v: \hat{J} \rightarrow \mathbb{R}$ is a solution of Problem \eqref{im1} if $v \in \Omega^1$ and satisfies equation \eqref{im1}.

First, we introduce the properties of the Green's function for the particular case of $M=0$ (without piecewise constant argument). The results can be found in \cite{cabada2013comparison, cabada2015differential}.
\subsection{Solution of the equation $v'(t)+m\textup{ }v(-t)=h(t)$ with periodic boundary conditions}

We work with the problem
\begin{equation}
v'(t)+m\,v(-t)=h(t), \textup{ a.e } t \in \hat{J}, \quad v(T)=v(-T),
\label{mim1}
\end{equation}
where $m$ is a nonzero real constant, $T>0$, and $h \in \mathcal{L}^{1}(\hat{J})$.

Following the steps mentioned in \cite{cabada2013comparison,cabada2015differential}, we can proceed to work with the second order ordinary differential equation coupled to periodic boundary conditions:
\begin{equation}
\begin{aligned}
v''(t)+m^{2}v(t)&=f(t), \textup{ a.e } t\, \in \, \hat{J}, \\
v(T)-v(-T)&=0, \\
v'(T)-v'(-T)&=0,
\end{aligned}
\label{red1}
\end{equation}
with $f \in \mathcal{L}^{1}(\hat{J})$.

\begin{remark}
In this section $G_{m}$ denotes the Green's function of Problem \eqref{red1}.
\end{remark}

We present a proposition that gives us properties of the Green's function \( G_{m} \) related to Problem \eqref{red1}, the proof of which can be found in \cite[Proposition 3.1]{cabada2013comparison}.
\begin{proposition}
For all \( t, s \in \hat{J} \), the Green's function \( G_{m} \), related to Problem \eqref{red1}, satisfies the following properties:
\begin{enumerate}
\item $G_{m}(t,s)=G_{m}(s,t)$,
\item $G_{m}(t,s)=G_{m}(-t,-s)$,
\item $\frac{\partial{G_{m}}}{\partial{t}}(t,s)=\frac{\partial{G_{m}}}{\partial{s}}(s,t)$,
\item $\frac{\partial{G_{m}}}{\partial{t}}(t,s)=-\frac{\partial{G_{m}}}{\partial{t}}(-t,-s)$,
\item $\frac{\partial{G_{m}}}{\partial{t}}(t,s)=-\frac{\partial{G_{m}}}{\partial{s}}(t,s)$.
\end{enumerate}
\label{propg}
\end{proposition}

With the above, let us now state a proposition that indicates how to obtain the Green's function for Problem \eqref{mim1}.
\begin{proposition}
Assume that \( m \neq k\pi/T \), \( k \in \mathbb{Z} \). Then, the Problem \eqref{mim1} has a unique solution given by the expression
\begin{equation*}
v(t):=\int_{-T}^{T}{\overline{G}_{m}(t,s)h(s) \mathrm{d}s},
\end{equation*}
where
\begin{equation*}
\overline{G}_{m}(t,s):=m \, G_{m}(t,-s)-\frac{\partial{G_{m}}}{\partial{s}}(t,s), (t,s) \in \hat{J} \times \hat{J},
\end{equation*}
and $G_{m}$ is the Green's function related to the second order periodic boundary value Problem \eqref{red1}.
\label{formamim}
\end{proposition}

Again, it is possible to see that the Green's function $\overline{G}_{m}$ satisfies the following properties.
\begin{proposition}
\cite[Proposition $3.2.3$]{cabada2015differential}
$\overline{G}_{m}$ satisfies the following properties:
\begin{enumerate}
\item $\overline{G}_{m}$ and $\frac{\partial{\overline{G}_{m}}}{\partial{t}}$ exist and they are continuous on $(\hat{J} \times \hat{J}) \setminus \{(t,t), (t,-t)\}$,
\item $\overline{G}_{m}(t,t^{-})$ and $\overline{G}_{m}(t,t^{+})$ exist and are finite for all $t \in \hat{J}$ and they satisfy $\overline{G}_{m}(t,t^{-})-\overline{G}_{m}(t,t^{+})=1$, for all $t \in \hat{J}$,
\item $\frac{\partial{\overline{G}_{m}}}{\partial{t}}(t,s)+m\,\overline{G}_{m}(-t,s)=0$ for all $t,s \in \hat{J}$, $s \neq \pm t$,
\item $\overline{G}_{m}(T,s)=\overline{G}_{m}(-T,s)$ for all $s \in (-T,T)$,
\item $\overline{G}_{m}(t,s)=\overline{G}_{m}(-s,-t)$ for all $t,s \in \hat{J}$,
\item $\overline{G}_{m}(t,s)=-\overline{G}_{-m}(-t,-s)$ for all $t,s \in \hat{J}$,
\item $\overline{G}_{m}(t,T)=\overline{G}_{m}(t,-T)$ for all $t \in (-T,T)$.
\end{enumerate}
\label{caragbarra}
\end{proposition}
%
%

It is not difficult to verify that the Green's function $G_{m}$ is given by the following expression:
\begin{equation*}
G_{m}(t,s)=\frac{1}{2m\sin(mT)}
\left\{
\begin{array}{lll}
\cos{m(T+s-t)}, & \mbox{\rm if}  & s \leq t, \\
\cos{m(T-s+t)}, & \mbox{\rm if}  & s \geq t.
\end{array}
\right.
\end{equation*}
Therefore, as shown in \cite[Section $3.2.1$]{cabada2015differential}
\begin{equation}
\overline{G}_{m}(t,s)=\frac{1}{2m\sin(mT)}
\left\{
\begin{array}{lll}
\cos{m(T-s-t)}+\sin{m(T+s-t)}, & \mbox{\rm if}  & -t \leq s<t, \\
\cos{m(T-s-t)}-\sin{m(T-s+t)}, & \mbox{\rm if}  & -s \leq t<s, \\
\cos{m(T+s+t)}+\sin{m(T+s-t)}, & \mbox{\rm if}  & |-t|>s, \\
\cos{m(T+s+t)}-\sin{m(T-s+t)}, & \mbox{\rm if}  & t<-|s|.
\end{array}
\right.
\label{formulag}
\end{equation}

From the previous formula, the sign of the Green's function $\overline{G}_{m}$ can be directly studied, leading to the following result \cite[Theorem $3.2.8$]{cabada2015differential}: 

\begin{theorem} 
	The Green's function, $\overline{G}_m$, related to Problem \eqref{mim1}, satisfies the following properties:
\begin{enumerate}
    \item If $mT \in \left(0, \frac{\pi}{4}\right)$, then $\overline{G}_{m}$ is strictly positive on $\hat{J} \times \hat{J}$.
    \item If $mT \in \left(-\frac{\pi}{4}, 0\right)$, then $\overline{G}_{m}$ is strictly negative on $\hat{J} \times \hat{J}$.
    \item If $mT = \frac{\pi}{4}$, then $\overline{G}_{m}$ vanishes at $P := \{(-T, -T), (0, 0), (T, T), (T, -T)\}$ and is strictly positive on $(\hat{J} \times \hat{J}) \setminus P$.
    \item If $mT = -\frac{\pi}{4}$, then $\overline{G}_{m}$ vanishes at $P$ and is strictly negative on $(\hat{J} \times \hat{J}) \setminus P$.
    \item If $mT \in \mathbb{R} \setminus \left[-\frac{\pi}{4}, \frac{\pi}{4}\right]$, then $\overline{G}_{m}$ changes sign on $\hat{J} \times \hat{J}$.
\end{enumerate}
\label{positividadg}
\end{theorem}

Now, by using the formula proved in Section $3$, we are in conditions to solve Problem \eqref{im1} in the next subsection.
\subsection{Solution of Problem \eqref{im1}}

Once the Green's function for Problem \eqref{mim1} has been obtained and analyzed, we can proceed to work with Problem \eqref{im1}. To calculate the Green's function associated with this new problem, which we denote by $\overline{H}_{m,M}$, we follow the steps outlined in Section \ref{sec3}. The Green's function $\overline{H}_{m,M}$ must satisfy the following proposition:
\begin{proposition}
	If $m \neq k\pi/T$ then the solution to Problem \eqref{im1} and the Green's function $\overline{H}_{m,M}$ is unique if and only if the matrix $A$, defined in \eqref{matriz}, associated with this problem is invertible.
	\label{prop59}
\end{proposition}

 Thus, assuming that $m \neq \frac{k \pi}{T}$, $k \in \mathbb{Z}$ and matrix $A$, defined in \eqref{matriz}, is invertible, from equation \eqref{final1}, we arrive at
\begin{equation}
\overline{H}_{m,M}(t,s)=\overline{G}_{m}(t,s)-M \Big(\sum_{j=[-T]}^{[T]} \sum_{i=[-T]}^{[T]}{\overline{G}_{m}(j,s)\int_{-T}^{T}{\overline{G}_{m}(t,r)} \alpha_{ij}(r)} \mathrm{d}r\Big),
\label{exparti}
\end{equation}
where $\overline{G}_{m}$ is the Green's function of Problem \eqref{mim1} given by expression \eqref{formulag}, and $\alpha_{ij}$ are the functions defined in \eqref{g2} (Notice that $\alpha_{ij}(r)$ depends also on $m$ and $M$).

%
%

We proceed, then, to deduce a series of properties of the function $\overline{H}_{m,M}$ that will be useful to us. For simplicity, we will denote by $D$ and $D_{t}$ the following sets:
\begin{equation*}
D \equiv \{-[T], \ldots, 0, \ldots,[T]\}
\end{equation*}
and
\begin{equation*}
D_{t} \equiv D \cup \{t,-t\}, \, \textup{ for }t \in \hat{J} \textup{ given}. 
\end{equation*}

Using Definition \ref{d1} and the properties of $\overline{G}_{m}$ shown in Proposition \ref{caragbarra}, it is not difficult to verify the following result.

\begin{proposition}
Under the hypotheses of Proposition \ref{prop59}, the Green's function $\overline{H}_{m,M}$ related to Problem \eqref{im1} exists, is unique and satisfies the following properties:
\begin{enumerate}
    \item Given $s \in \mathring{\hat{J}}$, function $t \rightarrow\overline{H}_{m,M}(t,s)$ is well-defined and continuous for all $t \in \hat{J}$, $t \neq s$, and of class $\mathcal{C}^{1}$ for all $t \in \hat{J}$, $t \neq s$, $t \neq -s$, and $t \notin D \setminus \{0\}$ (see Figure \ref{rexicontinua}).

    \item Given $t \in \mathring{\hat{J}}$, function $s \rightarrow \overline{H}_{m,M}(t, \cdot)$ is well-defined and continuous for all $s \in \hat{J}$, $s \neq t$, and $s \notin D$. Moreover, $\overline{H}_{m,M}(t, \cdot)$ is of class $\mathcal{C}^{1}$ for all $s \in \hat{J}$, $s \notin D_{t}$ (see Figure \ref{rexicontinua}).

    \item For each $s \in (-T,T)$, the function $t \rightarrow \overline{H}_{m,M}(t,s)$ is the solution of the following differential equation:
    \begin{equation}
    \frac{\partial }{\partial t}\overline{H}_{m,M}(t,s)+m\,\overline{H}_{m,M}(-t,s)+M\,\overline{H}_{m,M}([t],s)=0,
    \label{enova}
    \end{equation}
    for all $t \in \hat{J}$, $t \neq s$, $t \neq -s$, $t \notin D \setminus \{0\}$.

    \item For each $t \in (-T,T)$, $t \notin D$, the lateral limits exist and are finite:
    $$\overline{H}_{m,M}(t^-,t)=\overline{H}_{m,M}(t,t^+)\quad \text{and} \quad \overline{H}_{m,M}(t,t^-)=\overline{H}_{m,M}(t^+,t),$$
    and furthermore,
    $$\overline{H}_{m,M}(t^+,t)-\overline{H}_{m,M}(t^-,t)=\overline{H}_{m,M}(t,t^-)-\overline{H}_{m,M}(t,t^+)=1.$$

    \item For each $s \in (-T,T)$, the function $t \rightarrow \overline{H}_{m,M}(t,s)$ satisfies the boundary condition:
    $$\overline{H}_{m,M}(T,s)=\overline{H}_{m,M}(-T,s).$$
        \item For each $t \in (-T,T)$, the function $s \rightarrow \overline{H}_{m,M}(t,s)$ satisfies the boundary condition:
    $$\overline{H}_{m,M}(t,T)=\overline{H}_{m,M}(t,-T).$$
\end{enumerate}
\label{cara}
\end{proposition}
\vspace{-0.1cm}
\begin{figure}[H]
	\begin{center}
	\includegraphics[scale=0.2]{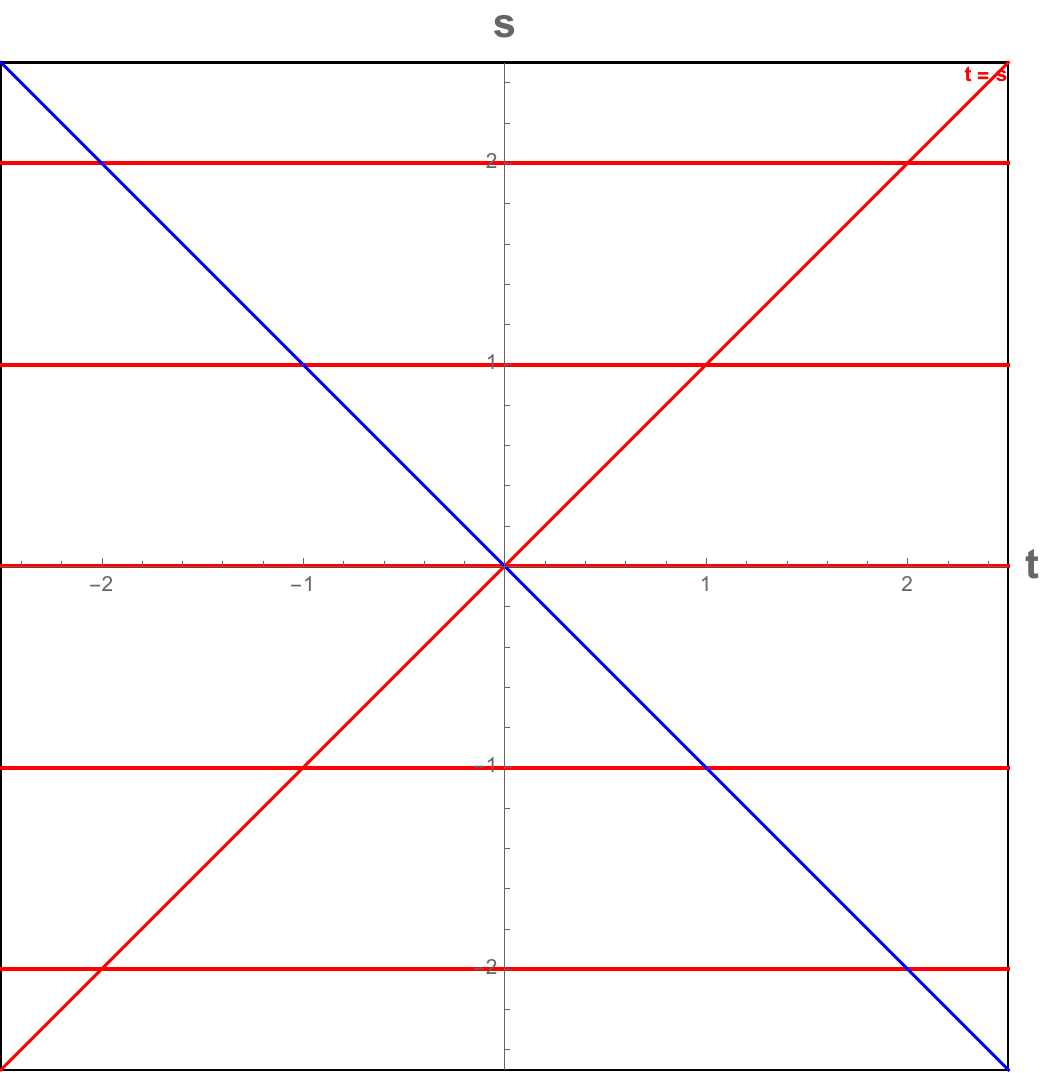}
	\end{center}
	\vspace{-20pt}
    \caption{The region where the Green's function $\overline{H}_{m,M}$ is continuous and of class $\mathcal{C}^{1}$ for $T = 5/2$. The Green's function is continuous throughout the square except along the red-marked lines, and is of class $\mathcal{C}^{1}$ throughout the square, except along the lines marked in red and blue.}
	\label{rexicontinua}
\end{figure}

\begin{remark}
In general, we have that
\begin{align*}
\overline{H}_{m,M}(t,t^-)-\overline{H}_{m,M}(t,t^+)&=\overline{G}_{m}(t,t^-)-\overline{G}_{m}(t,t^+) \\
& \quad-M \left( \sum_{j=[-T]}^{[T]} \sum_{i=[-T]}^{[T]}\left(\overline{G}_{m}(j,t^-)-\overline{G}_{m}(j,t^+)\right)\int_{-T}^{T}{\overline{G}_{m}(t,r) \alpha_{ij}(r)} \mathrm{d}r\right).
\end{align*}
If $t \in \hat{J}$, $t \neq j$ for all $j \in \{-[T], \ldots, 0, \ldots,[T]\}$, then $\overline{H}_{m,M}(t,t^-)-\overline{H}_{m,M}(t,t^+)=1$.

If $t = j_{0}$, $t \notin D$, then
\begin{equation*}
\overline{H}_{m,M}(j_{0},j_{0}^-)-\overline{H}_{m,M}(j_{0},j_{0}^+)=1-M \sum_{i=[-T]}^{[T]}{\int_{-T}^{T}{\overline{G}_{m}(j_{0},r) \alpha_{ij_{0}}(r) \mathrm{d}r}}.
\end{equation*}

%
%

\end{remark}




Furthermore, when we can ensure the uniqueness of the solution for Problem \eqref{im1}, we can also guarantee the existence of a symmetry of the function $\overline{H}_{m,M}$. Thus, we present the following Lemma.

\begin{lemma}
Under the assumptions of Proposition \ref{prop59}, the following symmetry property is satisfied
\begin{equation*}
\overline{H}_{m,M}(t,s)=-\overline{H}_{-m,-M}(-t,-s) \textup{ for all }t,s \in \hat{J} \times \hat{J} \,(s \notin D_{t}).
\end{equation*}
\label{simetriah}
\end{lemma}
\begin{proof}
First, we define
\begin{equation*}
v(t,s):=\overline{H}_{-m,-M}(-t,-s).
\end{equation*}
Therefore, using equality \eqref{enova}, since $[-t] = -[t]$, we have, for all $t \in \hat{J},\, t \neq s, \,t \neq -s \textup{ and }t \notin D \setminus \{0\}$, the following property:
\begin{align*}
\frac{\partial{v}}{\partial{t}}(t,s)&=-\frac{\partial{\overline{H}_{-m,-M}}}{\partial{t}}(-t,-s) =-m\,\overline{H}_{-m,-M}(t,-s)-M\,\overline{H}_{-m,-M}([-t],-s)\\
&=-m\,v(-t,s)-M\,v([t],s),
\end{align*}
that is:
\begin{equation*}
\frac{\partial{v}}{\partial{t}}(t,s)+m\,v(-t,s)+M\,v([t],s)=0 \textup{ for all }(t,s) \in \hat{J} \times \hat{J}, t \neq s, t \neq -s \textup{ and }t \notin D \setminus \{0\}.
\end{equation*}

On the other hand, from the periodicity of $\overline{H}_{m,M}$ shared in Proposition \ref{cara}, we would have that
\begin{equation*}
v(-T,s)=\overline{H}_{-m,-M}(T,-s)=\overline{H}_{-m,-M}(-T,-s)=v(T,s), \quad s \in (-T,T)
\end{equation*}
and, using Proposition \ref{cara}, Part $4$, we have that for all $s \notin D_{t}$,
\begin{align*}
\overline{H}_{-m,-M}(s^{+},s)-\overline{H}_{-m,-M}(s^{-},s)=1.
\end{align*}
Therefore
\begin{equation*}
 v(s^{+},s)-v(s^{-},s)=\overline{H}_{-m,-M}((-s)^{-},-s)-\overline{H}_{-m,-M}((-s)^{+},-s)=-1.
\end{equation*}
Thus, the uniqueness of the solution, implies that
\begin{equation*}
v(t,s)=-\overline{H}_{m,M}(t,s)
\end{equation*}
or, which is the same,
\begin{equation*}
\overline{H}_{m,M}(t,s)=-\overline{H}_{-m,-M}(-t,-s) \quad \forall (t,s) \in \hat{J} \times \hat{J}, s \notin D_{t}.
\end{equation*}
\end{proof}

From previous Lemma, it is immediate to verify that

\begin{corollary}
$\overline{H}_{m,M}$ is positive on $\hat{J} \times \hat{J}$ if and only if $\overline{H}_{-m,-M}$ is negative on $\hat{J} \times \hat{J}$.
\end{corollary}

In addition to the symmetry in the Green's function $\overline{H}_{m,M}$, we also have a symmetry in the matrix $A$ associated with the problem, given by equation \eqref{matriz}. Indeed, let $E \equiv A(m,M)$ and $F \equiv A(-m,-M)$ denote the matrix $A$ associated with Problem \eqref{im1} with parameters $m$ and $M \in \mathbb{R}$, and $-m$ and $-M \in \mathbb{R}$, respectively.  Then, the following Proposition holds:
\begin{proposition}
Let $e_{ij}$ be the elements of the matrix $E$ and $f_{ij}$ the elements of the matrix $F$. Then, it is verified that
\begin{equation*}
e_{ij}=f_{-i,-j},
\end{equation*}
with $i,j \in \{ [-T], \ldots,0, \ldots, [T]\}$. 

Moreover, it holds that $\det{E} = \det{F}$, i.e.: 
\begin{equation*}
\det{A(m,M)}=\det{A(-m,-M)} \textup{ for all } m,M \in \mathbb{R}.
\end{equation*}
\end{proposition}
\begin{proof}
Observing as the matrices $E$ and $F$ are, we see that all the equalities we want to prove are of the following type:
\begin{equation*}
M\int_{a}^{b}{\overline{G}_{m}(i,s) \mathrm{d}s}=-M \int_{-b}^{-a}{\overline{G}_{-m}(-i,s) \mathrm{d}s},
\end{equation*}
with $a,b \in \mathbb{R}$.

Therefore, to show that $e_{ij} = f_{-i,-j}$, it is enough to demonstrate that
\begin{equation*}
\int_{-a}^{-b}{\overline{G}_{m}(-c,s)\mathrm{d}s}=-\int_{a}^{b}{\overline{G}_{-m}(c,s) \mathrm{d}s} \textup{ for all }a,b,c \in \mathbb{R}.
\end{equation*}

By making the variable change $u = -s$, we have that
\begin{equation*}
\int_{a}^{b}{\overline{G}_{m}(c,s) \mathrm{d}s}=\int_{-b}^{-a}{\overline{G}_{m}(c,-u) \mathrm{d}u}.
\end{equation*}

Next, from Part $6$ of Proposition \ref{caragbarra}, we obtain that
\begin{equation*}
\int_{a}^{b}{\overline{G}_{m}(c,s) \mathrm{d}s}=-\int_{-b}^{-a}{\overline{G}_{-m}(-c,u) \mathrm{d}u}
\end{equation*}
which is precisely what we wanted to prove.

Finally, let us see that $\det{E} = \det{F}$. From the equality $e_{ij} = f_{-i,-j}$, we can deduce that to obtain the matrix $F$ from $E$, we need to swap $[T]$ rows and the same number of columns. Therefore, by the properties of matrices, we conclude that $\det{E} = \det{F}$.
\end{proof}

\begin{figure}[H]
	\begin{center}
	\includegraphics[scale=0.5]{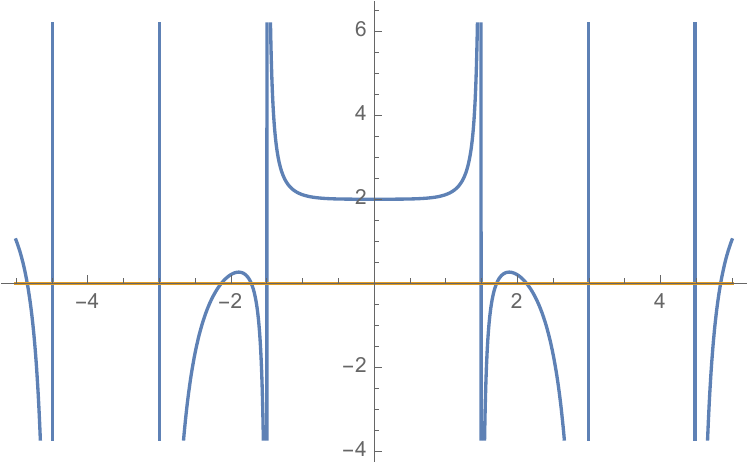}
	\end{center}
	\vspace{-20pt}
    \caption{Representation of the determinant of $A$ when $M = m$ and $T = 2.1$ as a function of the parameter $m$.}
	\label{figura10}
\end{figure}

\begin{remark}
Observing Figure \ref{figura10}, it is interesting to note that the values where the determinant has vertical asymptotes coincide with the spectrum of Problem \eqref{mim1}.
\end{remark}

On the other hand, we can also state the following Lemma regarding the determinant of $A$.
\begin{lemma}
Given a matrix $A$ related to Problem \eqref{im1} with parameters $T >0$, $m$, $M \in \mathbb{R}$ such that $m = -M$, then $\det{A(m,M)} = 0$.
\label{detcero}
\end{lemma}
\begin{proof}
The previous analysis shows that if Problem \eqref{mim1} has a unique solution then $det{A(m,M)}=0$ if and only if Problem  \eqref{im1} with $h=0$ on $\hat{J}$, has nontrivial solutions. Thus, if we consider a constant solution \( v(t) = C \) for all \( t \in \hat{J} \), then
\begin{equation*}
v'(t)+mv(-t)+Mv([t])=C(m+M).
\end{equation*}
Thus, \( m + M = 0 \) is a straight line of eigenvalues, and consequently, \( \det{A(m,M)} = 0 \).
\end{proof}


On the other hand, from Lemma \ref{detcero}, it is easy to deduce the following corollary.
\begin{corollary}
A necessary condition for the Green's function \( \overline{H}_{m,M} \) associated with Problem \eqref{im1} to be positive on \( \hat{J} \times \hat{J} \) is that \( m + M > 0 \). Similarly, a necessary condition for \( \overline{H}_{m,M} \) to be negative on \( \hat{J} \times \hat{J} \) is that \( m + M < 0 \).
\end{corollary}
\begin{proof}
Assuming that Problem \eqref{im1} has a unique solution for any $h \in \mathcal{L}(\hat{J})$. The proof is a direct consequence of the fact that if $h(t)=1$ for all $t \in \hat{J}$ then $v(t)=\frac{1}{m+M}$ is its unique solution.
\end{proof}

Next, we will study the behaviour of the partial derivatives of \( \overline{H}_{m,M} \). In Proposition \ref{cara}, expression \eqref{enova}, we have already obtained the equation that \( \frac{\partial}{\partial t} \overline{H}_{m,M} \) satisfies.


On the other hand, for the partial derivative of \( \overline{H}_{m,M} \) with respect to \( s \), we arrive at the following proposition.
\begin{proposition}
\begin{equation*}
\frac{\partial}{\partial{s}}\overline{H}_{m,M}(t,s)=m\,\overline{H}_{m,M}(t,-s) \textup{ for all }s \in \hat{J}, \, s \notin D_{t}.
\end{equation*}
\label{derivadas}
\end{proposition}

\begin{proof}
Throughout the proof, we will consider \( s \in \hat{J} \), \( s \notin D_{t} \).

We will begin by proving that
\begin{equation*}
\frac{\partial}{\partial{s}}\overline{G}_{m}(t,s)=m\,\overline{G}_{m}(t,-s),
\end{equation*}
where \( \overline{G}_{m} \) is the Green's function related to Problem \eqref{mim1} and given by expression \eqref{formulag}. By Proposition \ref{formamim}, \( \overline{G}_{m} \) satisfies that
\begin{equation*}
\overline{G}_{m}(t,s)=m\,G_{m}(t,-s)-\frac{\partial}{\partial{s}}G_{m}(t,s), \forall (t,s) \in \hat{J} \times \hat{J},
\end{equation*}
where \( G_{m} \) is the Green's function related to Problem \eqref{red1}.

Next, using Property 5 of Proposition \ref{propg}, we deduce that
\begin{equation*}
\frac{\partial}{\partial{s}}\overline{G}_{m}(t,s)=-m \frac{\partial}{\partial{s}}G_{m}(t,-s)-\frac{\partial^{2}}{\partial{t}^{2}}G_{m}(t,s).
\end{equation*}
Thus, since \( \frac{\partial^{2}}{\partial t^{2}} G_{m}(t,s) + m^{2} G_{m}(t,s) = 0 \), we have that
\begin{align*}
\frac{\partial}{\partial{s}}\overline{G}_{m}(t,s)&=-m\frac{\partial}{\partial{s}}G_{m}(t,-s)+m^{2}\,G_{m}(t,s) \\
&=m(m\,G_{m}(t,s)-\frac{\partial}{\partial{s}}G_{m}(t,-s))=m\,\overline{G}_{m}(t,-s).
\end{align*}

Now it only remains to use equation \eqref{exparti}:
\begin{align*}
\frac{\partial}{\partial{s}}\overline{H}_{m,M}(t,s)&=\frac{\partial}{\partial{s}}\overline{G}_{m}(t,s)-M \Big[\sum_{j=[-T]}^{j=[T]} \sum_{i=[-T]}^{i=[T]}{\frac{\partial}{\partial{s}}\overline{G}_{m}(j,s)\int_{-T}^{T}{\overline{G}_{m}(t,r)} \alpha_{ij}(r)} \mathrm{d}r\Big] \\
&=m\,\overline{G}_{m}(t,-s)-M \Big[\sum_{j=[-T]}^{j=[T]} \sum_{i=[-T]}^{i=[T]}{m\,\overline{G}_{m}(j,-s)\int_{-T}^{T}{\overline{G}_{m}(t,r)} \alpha_{ij}(r)} \mathrm{d}r\Big] \\
&=m\,\overline{H}_{m,M}(t,-s),
\end{align*}
and the proof is concluded.

\end{proof}

\subsection{Study of the function $\overline{H}_{m,M}$ when $T \in (0,1]$}
First, we will study the case in which $T \in (0,1]$. We will search for the point $(\hat{t},\hat{s})$ where the function $\overline{H}_{m,M}$ reaches its minimum when it is positive.

To begin with, we present the following observation.

\begin{remark}
	For all $T > 0$ and $m \neq \frac{k\pi}{T}$, $k \in \mathbb{Z}$, the unique solution to Problem \eqref{mim1} with $\sigma(t)=1$ for all $t \in \hat{J}$ is the constant solution $v_{1}(t)=\frac{1}{m}$.
	
	In particular, it holds that
	\begin{equation*}
		\frac{1}{m}=v_{1}(t)=\int_{-T}^{T}{\overline{G}_{m}(t,s) \mathrm{d}s}, \quad \forall t \in \hat{J}.
	\end{equation*}

Therefore, taking into account the previous observation and the expression \eqref{final2}, we deduce that the Green's function for Problem \eqref{im1} in the particular case of $T \in (0,1]$ takes the form
\begin{equation}
	\overline{H}_{m,M}(t,s)=\overline{G}_{m}(t,s)-\frac{M}{m+M}\overline{G}_{m}(0,s) \quad (t,s) \in \hat{J} \times \hat{J}.
	\label{exparti2}
\end{equation}

Furthermore, if $T \leq 1$, $t \in \hat{J}$ and $t \neq 0$, we have:
\begin{equation*}
	\overline{H}_{m,M}(t,t^-)-\overline{H}_{m,M}(t,t^+)=\overline{G}_{m}(t,t^-)-\overline{G}_{m}(t,t^+)-\frac{M}{m+M}(\overline{G}_{m}(0,t)-\overline{G}_{m}(0,t))=1.
\end{equation*}
Moreover, if $t=0$ then (see Figure \ref{salto})
\begin{equation*}
	\overline{H}_{m,M}(0,0^+)-\overline{H}_{m,M}(0,0^-)=\overline{G}_{m}(0,0^+)-\overline{G}_{m}(0,0^-)-\frac{M}{m+M}(\overline{G}_{m}(0,0^+)-\overline{G}_{m}(0,0^-))=\frac{m}{m+M}.
\end{equation*}

\begin{figure}[H]
	\begin{center}
		\includegraphics[scale=0.52]{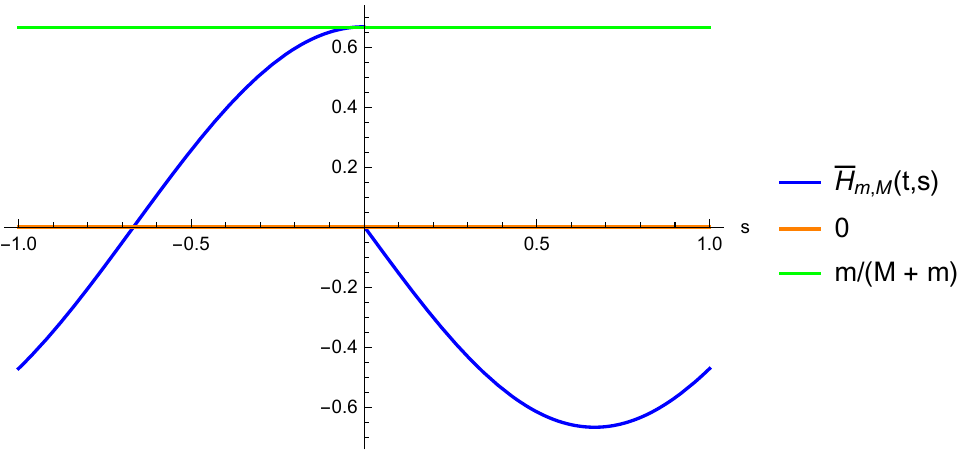}
	\end{center}
	\vspace{-20pt}
	\caption{Representation of the jump of $\overline{H}_{m,M}$ when $t = 0$ for $T = 1$, $m = 2.36$, and $M = 1.19$.}
	\label{salto}
\end{figure}

\end{remark}

\subsubsection*{Minimum of the function $\overline{H}_{m,M}$ when it is positive and $T \in (0,1]$}

Let us begin by proving that the minimum of \( \overline{H}_{m,M} \) occurs at \( t = s^- \) when we assume that \( \overline{H}_{m,M} \) is positive and \( m \in (-\pi/4T, \pi/4T) \). To do this, we will use the derivative of \( \overline{H}_{m,M}(t,s) \) with respect to \( t \). We have that
\begin{align*}
\frac{\partial}{\partial{t}}\overline{H}_{m,M}(t,s)&=-m\,\overline{H}_{m,M}(-t,s)-M\,\overline{H}_{m,M}(0,s)\\
&=-m\,\left(\overline{G}_{m}(-t,s)-\frac{M}{m+M}\overline{G}_{m}(0,s) \right)-M \left(\overline{G}_{m}(0,s)-\frac{M}{m+M}\overline{G}_{m}(0,s) \right)\\
&=-m\, \overline{G}_{m}(-t,s) \textup{ for all }t \in \hat{J}, \, t \neq s \textup{ and } t \neq -s.
\end{align*}

Applying Theorem \ref{positividadg}, we know that \( \overline{G}_{m} \) is positive on $\hat{J} \times \hat{J}$ if and only if \( m \in (0, \pi/4T) \) and negative when \( m \in (-\pi/4T, 0) \). From this, we deduce that \( \frac{\partial}{\partial t} \overline{H}_{m,M}(t,s) < 0 \) where it is defined (for such values of $m$ and $M$).

Moreover, \( \overline{H}_{m,M}(t,s) \) satisfies the following conditions for all $M \in (0, \pi/4T)$:
\begin{enumerate}
\item \( \overline{H}_{m,M}(\cdot, s) \) is continuous on \( \hat{J} \setminus \{s\} \),
\item \( \frac{\partial}{\partial t} \overline{H}_{m,M}(t, s) < 0 \) for \( t \neq s \), \( t \neq -s \), and \( t \neq 0 \),
\item \( \overline{H}_{m,M}(-T, s) = \overline{H}_{m,M}(T, s) \) for all \( s \in (-T, T) \),
\item \( \overline{H}_{m,M}(s^{+}, s) = \overline{H}_{m,M}(s^{-}, s) + 1 \) for all \( s \in (-T, T) \setminus \{0\} \).
\end{enumerate}

Therefore, for any $s \in (-T,T)$ fixed, the minimum on \( \hat{J} \) is attained at \( t = s^{-} \).

Since we are interested in finding the minimum of the function \( \overline{H}_{m,M} \) on \( \hat{J} \), we define the following function \( \overline{q}: \hat{J} \rightarrow \mathbb{R} \) (see Figure \ref{figuq}).
\begin{equation}
\overline{q}(s)=
\left\{
\begin{array}{lll}
\overline{H}_{m,M}(s^-,s), & \mbox{\rm if}  & s \neq -T, \\
\overline{H}_{m,M}(-T,-T^+), & \mbox{\rm if}  & s=-T. \\
\end{array}
\right.
\label{defq}
\end{equation}

\begin{figure}[H]
	\begin{center}
	\includegraphics[scale=0.5]{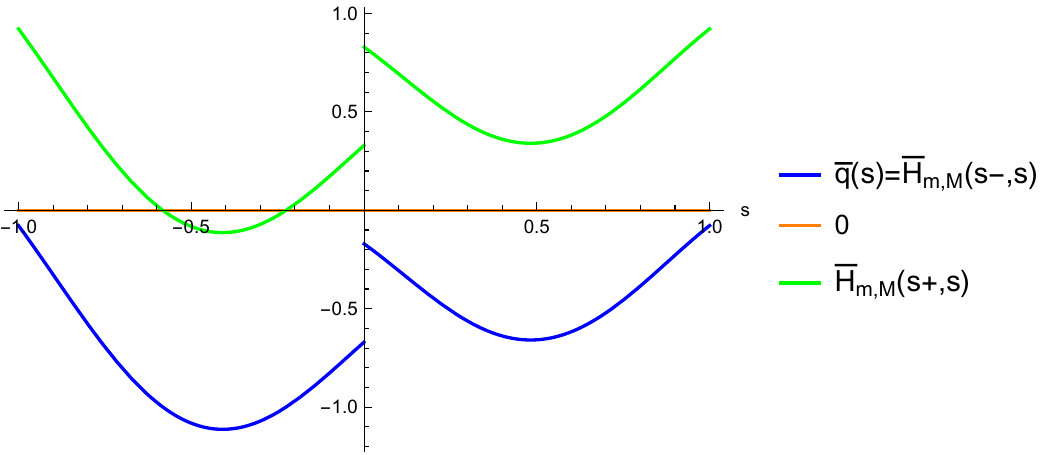}
	\end{center}
	\vspace{-20pt}
    \caption{Representation of \( \overline{H}_{m,M}(s^+, s) \) and \( \overline{H}_{m,M}(s^-, s) \) for \( T = 1 \) and \( m = M=-1.88 \).}
	\label{figuq}
\end{figure}

From now on, we will work only with the blue branch in Figure \ref{figuq}, that is, \( \overline{q}(s) \). We will start with the hypothesis that \( \overline{q}(s) \geq 0 \) for all \( s \in \hat{J} \) and will attempt to determine at which value of \( s \in \hat{J} \) the function \( \overline{q} \) attains its minimum.

We begin by calculating the first and second derivatives:
\begin{align*}
\overline{q}'(s)=\frac{\mathrm{d}}{\mathrm{d}s}\overline{H}_{m,M}(s^-,s)&=\frac{\partial}{\partial{t}}\overline{H}_{m,M}(s^-,s)+\frac{\partial}{\partial{s}}\overline{H}_{m,M}(s^-,s) \\
&=-m\,\overline{H}_{m,M}(-s,s)-M\,\overline{H}_{m,M}(0,s)+m\,\overline{H}_{m,M}(s,-s)
\end{align*}
and
\begin{align*}
\overline{q}''(s)&=\frac{\mathrm{d}^{2}}{\mathrm{d}s^{2}}\overline{H}_{m,M}(s^-,s)=-m\left(-\frac{\partial}{\partial{t}}\overline{H}_{m,M}((-s)^+,s)+\frac{\partial}{\partial{s}}\overline{H}_{m,M}((-s)^+,s) \right) \\
& \quad -M \left(\frac{\partial}{\partial{s}}\overline{H}_{m,M}(0,s)\right) +m \left(\frac{\partial}{\partial{t}}\overline{H}_{m,M}(s^-,-s)-\frac{\partial}{\partial{s}}\overline{H}_{m,M}(s^-,-s)\right) \\
&=-m\left(m\,\overline{H}_{m,M}(s^-,s)+M\,\overline{H}_{m,M}(0,s) +m\,\overline{H}_{m,M}((-s)^+,-s)\right) -Mm\,\overline{H}_{m,M}(0,-s) \\ 
& \quad +m\left(-m\,\overline{H}_{m,M}((-s)^+,-s)-M\,\overline{H}_{m,M}(0,-s)-m\,\overline{H}_{m,M}(s^-,s)\right) \\
&=-2\,m^{2}\left(\overline{H}_{m,M}(s^-,s)+\overline{H}_{m,M}((-s)^+,-s)\right)-m\,M\left(\overline{H}_{m,M}(0,s)+2\overline{H}_{m,M}(0,-s)\right).
\end{align*}

Recall that, for these values of \( T \), the function \( \overline{H}_{m,M} \) is given by the expression \eqref{exparti2}.

Therefore, we can rewrite the first and second derivatives of \( \overline{q}(s) \) as follows:
\begin{align*}
\overline{q}'(s)&=m \left(\overline{G}_{m}(s^-,s)-\frac{M}{m+M}\overline{G}_{m}(0,-s)-\overline{G}_{m}(-s,s)+\frac{M}{m+M}\overline{G}_{m}(0,s) \right) \\
& \quad -M \left(\overline{G}_{m}(0,s)-\frac{M}{m+M}\overline{G}_{m}(0,s) \right)
=m \left(\overline{G}_{m}(s,-s)-\overline{G}_{m}(-s,s)-\frac{M}{m+M}\overline{G}_{m}(0,-s) \right) \\
&=m\left(\overline{H}_{m,M}(s,-s)-\overline{G}_{m}(-s,s)\right)
\end{align*}
and, using Part 5 of Proposition \ref{caragbarra}, we deduce that
\begin{align*}
\overline{q}''(s) &=  -2\,m^2 \left(\overline{H}_{m,M}(s^-,s) + \overline{H}_{m,M}((-s)^+,-s)\right) - m\,M \left(\overline{H}_{m,M}(0,s) + 2\,\overline{H}_{m,M}(0,-s)\right) \\
       &= -2\,m^2 \left( \overline{G}_{m}(s^-,s) - \frac{M}{m+M} \overline{G}_{m}(0,s) + \overline{G}_{m}(s,-s) - \frac{M}{m+M} \overline{G}_{m}(0,-s) \right) \\
       & \quad - m\,M \left( \overline{G}_{m}(0,s) - \frac{M}{m+M} \overline{G}_{m}(0,s) + 2\,\overline{G}_{m}(0,-s) - \frac{2M}{m+M} \overline{G}_{m}(0,-s) \right) \\
       &= -4\,m^2 \overline{G}_{m}(s^-,s) + \frac{m^2 M}{m+M} \overline{G}_{m}(0,s) = -m^2 \left( 4\, \overline{G}_{m}(s^-,s) - \frac{M}{m+M} \overline{G}_{m}(0,s) \right) \\
       &= -m^2 (3\, \overline{G}_{m}(s^-,s) + \overline{H}_{m,M}(s^-,s)).
\end{align*}

Using previous equalities, we are in position to prove the following proposition:
\begin{proposition}
The function \( \overline{q}: \hat{J} \rightarrow \mathbb{R} \) satisfies the following properties:
\begin{enumerate}
\item \( \overline{q} \) is well-defined and \( \mathcal{C}^{1} \) on \( \hat{J} \setminus \{0\} \),
\item \( \overline{q}(0^+) - \overline{q}(0^-) = \frac{M}{m+M} \),
\item \( \overline{q}(T) = \overline{q}(-T) \),
\item \( \overline{q}'(T) = \overline{q}'(-T) \).
\end{enumerate}
\label{carq}
\end{proposition}

\begin{proof}
Let us see the proof of the different parts of the result.
\begin{enumerate}
\item The first part is directly deduced from the expression of the second derivative \( \overline{q}''(s) \) for all \( s \in \hat{J} \setminus \{0\} \).

\item Part 2 is proved as follows (see Figure \ref{saltoq}):
\begin{equation*}
\begin{aligned}
\overline{q}(0^+)-\overline{q}(0^-)&=\overline{H}_{m,M}((0^-)^+,0^+)-\overline{H}_{m,M}((0^-)^-,0^-)\\
&=\overline{G}_{m}(0,0^+)-\frac{M}{m+M}\overline{G}_{m}(0,0^+)-\overline{G}_{m}((0^-)^-,0^-)+\frac{M}{m+M}\overline{G}_{m}(0,0^-)\\
&=\overline{G}_{m}(0,0^+)-\overline{G}_{m}(0,0^+)+\frac{M}{m+M}\left(\overline{G}_{m}(0,0^-)-\overline{G}_{m}(0,0^+)\right)=\frac{M}{m+M}.
\end{aligned}
\end{equation*}

\begin{figure}[H]
	\begin{center}
	\includegraphics[scale=0.55]{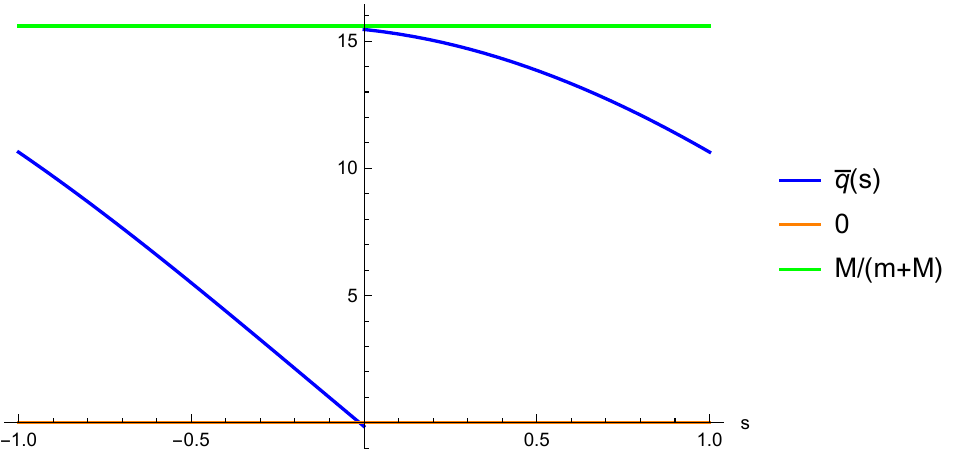}
	\end{center}
	\vspace{-20pt}
    \caption{Representation of the jump in \( \overline{q}(s) \) at \( s = 0 \) for \( m = -0.73 \), \( M = 0.78 \), and \( T = 1 \).}
	\label{saltoq}
\end{figure}

\item Part $3$.

Using properties $4$ and $5$ of Proposition \ref{caragbarra}, we obtain that:
\begin{align*}
\overline{q}(-T)&=\overline{H}_{m,M}((-T)^-,-T)= \overline{G}_{m}((-T)^-,-T)-\frac{M}{m+M}\overline{G}_{m}(0,-T)\\
&=\overline{G}_{m}(T^-,T)-\frac{M}{m+M}\,\overline{G}_{m}(0,T)=\overline{H}_{m,M}(T^-,T)=\overline{q}(T).
\end{align*}

\item Part $4$.

On the one hand,
\begin{align*}
\overline{q}'(T)&=m\left(\overline{H}_{m,M}(T,-T)-\overline{G}_{m}(-T,T)\right) \\
&=m\left(\overline{G}_{m}(T,-T)-\frac{M}{m+M}\overline{G}(0,-T)-\overline{G}_{m}(T,-T)\right)=\frac{-m\,M}{m+M}\overline{G}_{m}(0,-T).
\end{align*}
Now, using Property $5$ of Proposition \ref{caragbarra},
\begin{equation*}
\overline{q}'(T)=\frac{-m\,M}{m+M}\overline{G}_{m}(T,0).
\end{equation*}
On the other hand,
\begin{equation*}
\overline{q}'(-T)=m\left(\overline{G}_{m}(-T,T)-\frac{M}{m+M}\overline{G}(0,T)-\overline{G}_{m}(-T,T)\right)=\frac{-m\,M}{m+M}\overline{G}_{m}(0,T).
\end{equation*}
Again, using Property $5$ of Proposition \ref{caragbarra},
\begin{equation*}
\overline{q}'(-T)=\frac{-m\,M}{m+M}\overline{G}_{m}(-T,0).
\end{equation*}
Finally, taking into account Property $4$ of Proposition \ref{caragbarra}, we conclude that
\begin{equation*}
\overline{q}'(T)=\overline{q}'(-T).
\end{equation*}

\end{enumerate}
\end{proof}

Based on the previous expressions for \( \overline{q}'(s) \) and \( \overline{q}''(s) \) and Proposition \ref{carq}, we are in a position to prove the following result.

\begin{proposition}
Assume that \( \overline{H}_{m,M}(s^-,s) \geq 0 \) for all $s \in \hat{J}$. If \( mT \in (-\frac{\pi}{4}, \frac{\pi}{4}) \) and \( m + M > 0 \) (which must be satisfied if \( \overline{H}_{m,M} \) is positive), then the function \( \overline{q} \) attains its minimum at:
\begin{equation}
\left\{
\begin{array}{lll}
\overline{q}(0^-)=\lim_{s \rightarrow 0^{+}}{\overline{q}(s)} & \text{if} & M > 0, \\ 
\overline{q}(0^+)=\lim_{s \rightarrow 0^{-}}{\overline{q}(s)} & \text{if} & M < 0.
\end{array}
\right.
\label{defc}
\end{equation}
\label{minimo}
\end{proposition}
\begin{proof}
%
We divide the proof into two parts. First, we consider the case where \( mT \in (0, \pi/4) \), and subsequently, we will study what happens for \( mT \in (-\pi/4, 0) \).
\begin{enumerate}
\item Case $mT \in (0, \pi/4)$. 
We know that
\begin{equation*}
\overline{q}''(s) = -m^2 \left(3 \, \overline{G}_{m}(s^-,s) + \overline{H}_{m,M}(s^-,s)\right)
\end{equation*}
and, moreover \( \overline{G}_{m} > 0 \). Since we assume that we are dealing with pairs of values \( (m, M) \in \mathbb{R} \times \mathbb{R} \) such that \( \overline{H}_{m,M} > 0 \), then \( \overline{q}''(s) < 0 \). This means that the function \( \overline{q}(s)  \) is concave. 

As a concave function that is continuous for all \( s \in \hat{J} \) except for \( s = 0 \), the minimum can only be attained at \( s = 0 ^+\), $s=0^-$, \( s = -T \), or \( s = T \). Furthermore, in light of parts $3$ and $4$ of Proposition \ref{carq}, we directly infer that the infimum is taken at \( s = 0 \).


Now, function \( \overline{q} \) has a jump at \( s = 0 \). Taking into account Part 2 of Proposition \ref{carq}, we deduce that the infimum of the function \( \overline{q}(s) \) for \( s \in \hat{J} \) will be
\begin{equation*}
\left\{
\begin{array}{lll}
\overline{q}(0^-)=\lim_{s \rightarrow 0^{-}}{\overline{q}(s)} & \text{if} & M > 0, \\ 
\overline{q}(0^+)=\lim_{s \rightarrow 0^{+}}{\overline{q}(s)} & \text{if} & M < 0.
\end{array}
\right.
\label{saltoraro}
\end{equation*}

\item Case $mT \in (-\pi/4,0)$.
In this situation, we start from the fact that
\begin{equation*}
\overline{q}'(s)=m\left(\overline{H}_{m,M}(s^-,-s)-\overline{G}_{m}(-s,s)\right).
\end{equation*}
By Theorem \ref{positividadg}, we know that \( \overline{G}_{m} \) is strictly negative on \( \hat{J} \times \hat{J} \) if and only if $m \in (-\frac{\pi}{4T},0)$. Furthermore, by hypothesis, we have that \( \overline{H}_{m,M} \) is positive. Therefore, it is obvious that \( \overline{H}_{m,M}(s^-, -s) - \overline{G}_{m}(-s, s) \geq 0 \). Additionally, since \( m < 0 \), it follows that \( \overline{q}'(s) \leq 0 \) for all \( s \in \hat{J} \), $s \neq 0$.

Since the derivative is negative, the infimum can only be attained at \( s = T \), \( s = 0^+ \) or $s=0^-$ (the endpoint of the interval or the unique point where the function \( \overline{q} \) is discontinuous). However, by Part 3 of Proposition \ref{carq}, we discard the case where \( s = T \) and we conclude that \( s = 0 \).

Again, we must consider that there is a jump at \( s = 0 \). Since by hypothesis \( m + M > 0 \) and \( m < 0 \), it follows that \( M > 0 \). Taking this into account and considering Part 2 of Proposition \ref{carq}, we deduce that the minimum of \( \overline{q}(s) \) is given by
\begin{equation*}
\overline{q}(0^-)=\lim_{s \rightarrow 0^{-}}{\overline{q}(s)},
\end{equation*}
and we conclude the proof of the proposition.
\end{enumerate}
\end{proof}

\subsubsection{Constant sign region when $T \in (0,1]$ using the fixed-point method of Section \ref{metodofixo}}
As a consequence of previous results, we are able to delineate the region where the Green's function \( \overline{H}_{m,M}(t,s) \) is positive. 

We will begin by proving a series of lemmas that will allow us to reach the main result of this section.
\begin{lemma}
If \( mT > \frac{\pi}{4} \) or \( mT < \frac{-\pi}{4} \), then $\overline{H}_{m,M}$ changes its sign on \( \hat{J} \times \hat{J} \).
\label{lemapi}
\end{lemma}
\begin{proof}
By Theorem \ref{positividadg}, we know that if \( mT > \frac{\pi}{4} \) or \( mT < \frac{\pi}{4} \), then \( \overline{G}_{m} \) changes sign on \( \hat{J} \times \hat{J} \), and for \( mT = \frac{\pi}{4} \) or \( mT = -\frac{\pi}{4} \), it vanishes at \( (0,0) \).

Additionally, from equation \eqref{exparti2}, it is easy to deduce that
\begin{equation*}
\overline{H}_{m,M}(0,s)=\overline{G}_{m}(0,s)-\frac{M}{m+M}\overline{G}_{m}(0,s)=\frac{m}{m+M}\overline{G}_{m}(0,s) \textup{ for all }s \neq 0.
\end{equation*}
With this, it is clear that \( \overline{H}_{\pi/4T, M}(0, 0^+) = \overline{H}_{-\pi/4T, M}(0, 0^-) = 0 \) for all \( M \neq -m \) for which \( \overline{H}_{m,M} \) exists. 
Observing the expression for \( \overline{G}_{m}(t,s) \) given by \eqref{formulag}, we see that \( \overline{G}_{m}(0,s) \) changes sign on \( \hat{J} \) whenever either \( m > \pi/4T \) or \( m < \pi/4T \) and the same happens for $\overline{H}_{m,M}$ for any $M \in \mathbb{R}$.
\end{proof}

Let's now analyze the case where \( mT \in (-\pi/4, \pi/4) \). If for each \( m \in (-\pi/4T, \pi/4T) \), there exists a \( \overline{M} \) such that \( \overline{H}_{m, \overline{M}} > 0 \), then, \( \overline{H}_{m,M} > 0 \) for all \( M \) such that \( \overline{M} < M < M_{0} \), where \( M_{0} \) is the smallest positive solution to the fixed-point problem given by the operator \( \overline{T}_m(M_{0},t,s) \), defined in \eqref{opfixo}. As we have noticed previously, since this operator is independent of \( M_{1} \) and \( \overline{H}_{m,0} = \overline{G}_{m} \), we can take \( M_{1} = 0 \), which simplifies the calculations considerably. Therefore, we aim to find the fixed points of the operator $\overline{T}_m^0(M_0,t,s)$ defined as \eqref{exprem22}.

Thus, in our particular case, we would have \( \overline{H}_{m,M} > 0 \) for all \( M \) such that \( -m < M < M_{0} \), and \( \min_{(t,s) \in \hat{J} \times \hat{J}} \overline{H}_{m,M_{0}}(t,s) = 0 \). From previous results, we see that the minimum of \( \overline{H}_{m,M} \) whenever \( \overline{H}_{m,M} \geq 0 \) on $\hat{J} \times \hat{J}$, is attained at \( \overline{q}(0^-) \) if \( M > 0 \) and \( \overline{q}(0^+) \) if \( M < 0 \), where \( q \) is defined in \eqref{defq}.

Taking all simplifications into account, the expression of operator $\overline{T}_{m}^{0}$, defined in \eqref{opfixo}, can be written as:
\begin{enumerate}
\item If $M_{0}>0$:
\begin{equation*}
\overline{T}_{m}^{0}(M_{0})=\lim_{t \rightarrow s^{-}} \lim_{s \rightarrow 0^{-}}\frac{\overline{G}_{m}(t,s)}{\int_{-T}^{T}\overline{G}_{m}(t,r)\overline{H}_{m,M_{0}}(0,s) \mathrm{d}r}=\frac{(m+M_{0})(\cos{(mT)}-\sin{(mT)})}{\cos(mT)+\sin(mT)}.
\end{equation*}
\item If $M_{0}<0$:
\begin{equation*}
\overline{T}_{m}^{0}(M_{0})=\lim_{t \rightarrow s^{-}} \lim_{s \rightarrow 0^{+}}\frac{\overline{G}_{m}(t,s)}{\int_{-T}^{T}\overline{G}_{m}(t,r)\overline{H}_{m,M_{0}}(0,s) \mathrm{d}r}=m+M_{0}.
\end{equation*}
\end{enumerate}

%
%
Thus, from the previous equalities and Lemma \ref{necesaria}, the Green's function \( \overline{H}_{m,M} \) with \( m \in (-\pi/4T, \pi/4T) \) is positive when \( m \neq 0 \) if and only if \( M \in \mathbb{R} \) is such that \( -m < M < \frac{1}{2}m(-1 + \cot(m\,T)) \). For \( m = 0 \), by continuity, it would be positive on \( 0 < M < \frac{1}{2T} \) (see Figure \ref{figura2}).

\begin{figure}[H]
	\begin{center}
	\includegraphics[scale=0.63]{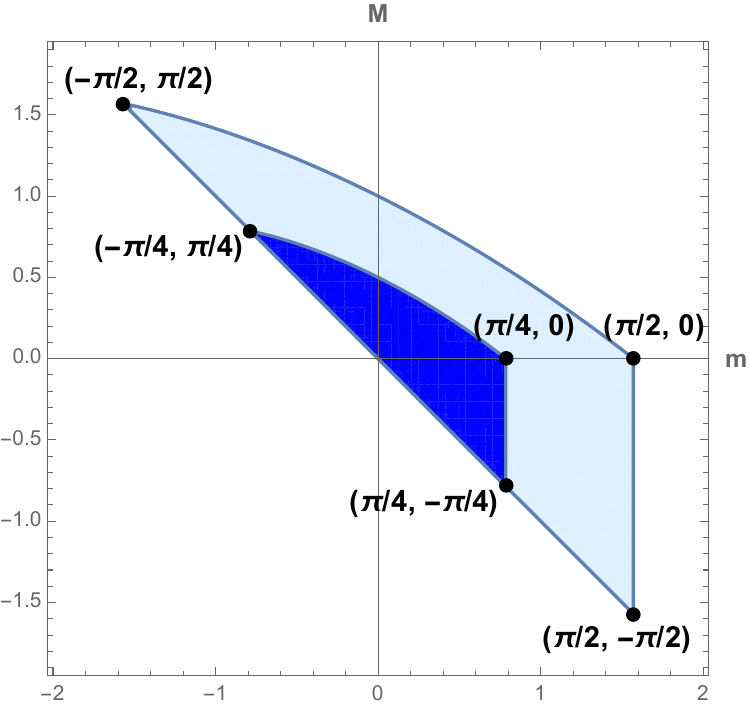}
	\end{center}
	\vspace{-20pt}
    \caption{Region where \( \overline{H}_{m,M} \) is positive for \( T = 1/2 \) in light blue and \( T = 1 \) in dark blue.}
	\label{figura2}
\end{figure}

On the other hand, by means of the symmetry property of \( \overline{H}_{m,M} \), which is proven in Lemma \ref{simetriah}, we can ensure that the function \( \overline{H}_{m,M} \) is negative if and only if \( M \) is such that \( -\frac{1}{2}m(1 + \cot(m \,T)) < M < -m \) and \( m \in (-\pi/4T, \pi/4T) \), \( m \neq 0 \). When \( m = 0 \), by continuity, it would be negative for \( -\frac{1}{2T} < M < 0 \) (see Figure \ref{figura3}).
\begin{figure}[H]
	\begin{center}
	\includegraphics[scale=0.63]{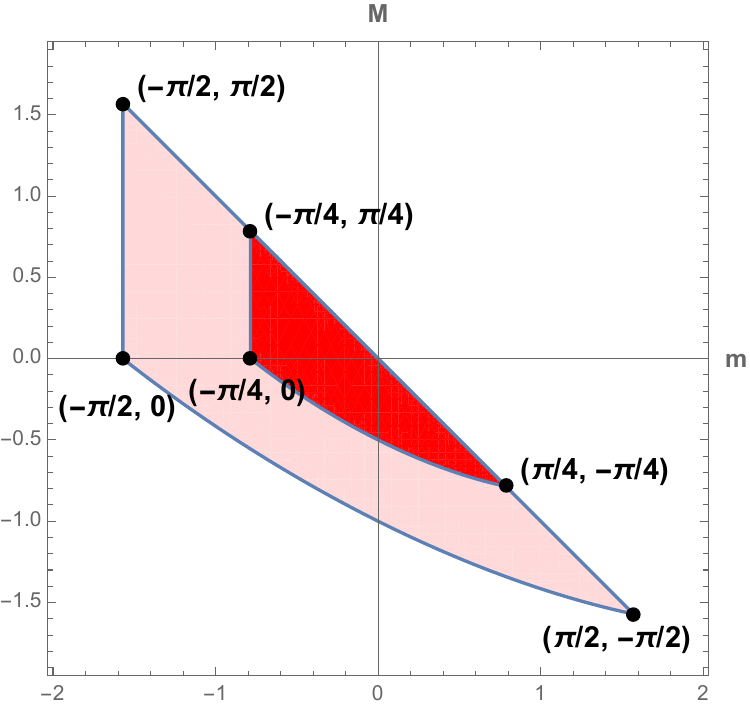}
	\end{center}
	\vspace{-20pt}
    \caption{Region where \( \overline{H}_{m,M} \) is negative for \( T = 1/2 \) in light red and \( T = 1 \) in dark red.}
	\label{figura3}
\end{figure}

\subsubsection{Constant sign region when $T \in (0,1]$ setting the minimum of the function equal to zero}

In this subsection, we will obtain the analytic expression of the constant sign region in another way. This involves setting the minimum to zero and determining an expression that relates \( m \) and \( M \).

We know that if \( \overline{H}_{m,M} \geq 0 \) in \( \hat{J} \times \hat{J} \), then
\begin{equation*} 
\min_{(t,s) \in J \times J}{\overline{H}_{m,M}(t,s)}=\left\{ \begin{array}{lll}
\lim_{t \rightarrow s^{-}}{\lim_{s \rightarrow 0^{-}}}{\overline{H}_{m,M}(t,s)}=\overline{H}_{m,M}((0^-)^-,0^-),& \mbox{\rm if}  & M>0,\\ \\
\lim_{t \rightarrow s^{-}}{\lim_{s \rightarrow 0^{+}}}{\overline{H}_{m,M}(t,s)}=\overline{H}_{m,M}((0^+)^-,0^+),& \mbox{\rm if}  & M<0.
\end{array}\right.
\end{equation*}

We want to obtain the expression for the pairs \((m, M)\) where the Green's function \( \overline{H}_{m,M} \) ceases to be positive. To do this, we need to set the previous expression to zero. This will give us a curve in terms of \( m \) and \( M \) that will help us delineate the region where the Green's function is positive. Setting it to zero yields the following values:

\begin{equation*}
\begin{array}{lll}
\overline{H}_{m,M}((0^-)^-,0^-)=0 \Leftrightarrow M= \frac{1}{2}m(-1+\cot{mT})& \mbox{\rm if}  & M>0,\\ \\
\quad \, \, \, \overline{H}_{m,M}((0^+)^-,0^+)=0 \Leftrightarrow m=\frac{\pi}{4T} & \mbox{\rm if}  & M<0.
\end{array}
\end{equation*}

Using the symmetry of \( \overline{H}_{m,M} \) given in Lemma \ref{simetriah}, we can obtain the curve in terms of \( m \) and \( M \) that will allow us to delineate the region where the Green's function is negative. We would have:

\begin{equation*}
\begin{array}{lll}
M= -\frac{1}{2}m(1+\cot{mT})& \mbox{\rm if}  & M>0,\\ \\
m=\frac{\pi}{4T} & \mbox{\rm if}  & M<0.
\end{array}
\end{equation*}

By representing the above when \( T = 1 \) together with the line \( m = -M \), since by Lemma \ref{necesaria} we know that a necessary condition for the Green's function \( \overline{H}_{m,M} \) to be positive is \( M > -m \) and for it to be negative is \( M < -m \), we obtain the graph shown in Figure \ref{figurana}.

By comparing figures \ref{figura2} and \ref{figurana}, we see that the two regions coincide.

\begin{figure}[H]
	\begin{center}
	\includegraphics[scale=0.6]{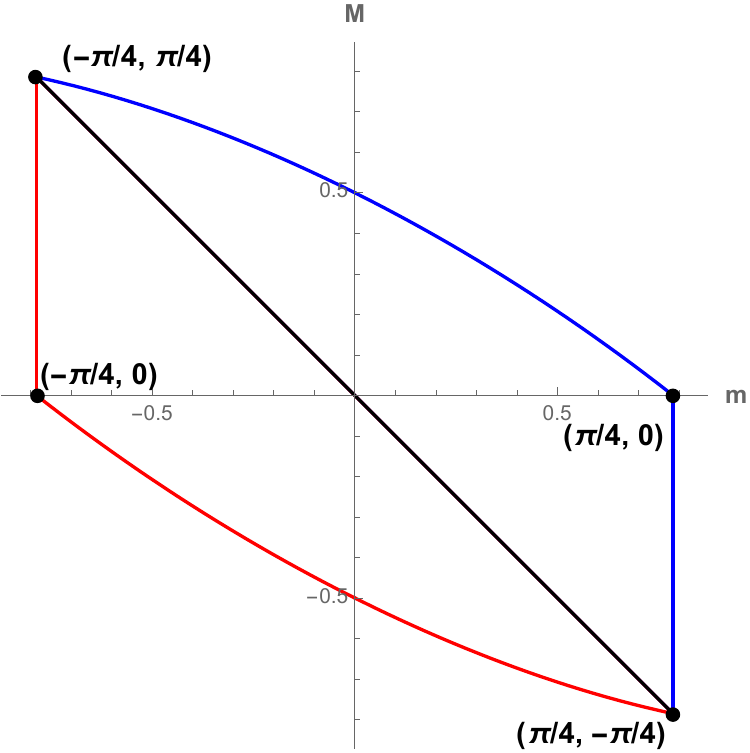}
	\end{center}
	\vspace{-20pt}
    \caption{Determination, setting the minimum of the function equal to zero, of the region where \( \overline{H}_{m,M} \) is positive (the region bounded between the blue and black curves) and where \( \overline{H}_{m,M} \) is negative (the region bounded between the red and black curves) for \( T = 1 \).}
	\label{figurana}
\end{figure}

\subsection{Study of the function $\overline{H}_{m,M}$ when $T>1$}

Now, we will study the case when $T>1$. Our first goal is to find the restrictions on the point $(\hat{t},\hat{s})$ where the function $\overline{H}_{m,M}$ attains its minimum, assuming it is positive.

\subsubsection*{Minimum of the function $\overline{H}_{m,M}$ when it is positive and $T>1$}

In the particular case where $m>0$ and $M>0$, we can easily deduce a result that restrict the points $(t,s) \in \hat{J} \times \hat{J}$ where the function $\overline{H}_{m,M}$ attains its minimum when it is positive. The steps will be similar to the case with $T<1$. First, we make the following remark:

\begin{remark}
The function $\overline{H}_{m,M}(t, \cdot)$ is discontinuous for values of $s \in \hat{J}$ that are integers. Therefore, if $s=n \in \mathbb{Z} \cap \hat{J}$, when searching for the minimum of the function $\overline{H}_{m,M}$, we must take into account that $$\overline{H}_{m,M}(t,n^-)=\lim_{s\rightarrow n^-}{\overline{H}_{m,M}(t,s)} \neq \lim_{s\rightarrow n^+}{\overline{H}_{m,M}(t,s)}=\overline{H}_{m,M}(t,n^+).$$
\end{remark}

Now, we present the following result.

\begin{proposition}
If the function $\overline{H}_{m,M}$ is positive, $m>0$ and $M>0$, then the minimum of $\overline{H}_{m,M}$ can occur only at $\overline{q}(n^-)$, $\overline{q}(n^+)$ with $n \in \mathbb{Z} \cap \hat{J}$ or at $\overline{q}(T)=\overline{q}(T^-)$, where $\overline{q}$ is given by \eqref{defq}.
\end{proposition}

\begin{proof}
According to parts $1$, $3$, $4$ and $5$ of Proposition \ref{cara} we have the following properties:
\begin{itemize}
\item $\overline{H}_{m,M}(\cdot,s)$ is well-defined and continuous for all $t \in \hat{J}$, $t \neq s$,
\item $\frac{\partial}{ \partial{t}}\overline{H}_{m,M}(t,s)<0$ for all $t \in \hat{J}$, $t \neq s$, $t \neq -s$, $t \notin D \setminus \{0\}$,
\item    $\overline{H}_{m,M}(s^+,s) - \overline{H}_{m,M}(s^-,s) = \overline{H}_{m,M}(s,s^-) - \overline{H}_{m,M}(s,s^+) = 1$ for all $s \in (-T,T)$, $s \notin D$,
\item  \( \overline{H}_{m,M}(T,s) = \overline{H}_{m,M}(-T,s) \) for all \( s \in (-T, T) \).
\end{itemize}
From these properties, we can deduce that the minimum of $\overline{H}_{m,M}$ occurs at $t=s^-$. 
Indeed, as in the case where $T<1$, we consider the function $\overline{q}: \hat{J} \rightarrow \mathbb{R}$, defined in \eqref{defq}.

Upon calculating the second derivative, we find:
\begin{equation*}
\begin{aligned}
\overline{q}''(s)&=-2m^{2}\left(\overline{H}_{m,M}(s^-,s)+\overline{H}_{m,M}((-s)^+,-s)\right)\\
& \quad -m\,M\left(\overline{H}_{m,M}(-[s],s)+\overline{H}_{m,M}([s],-s)+\overline{H}_{m,M}([s],-s)\right).
\end{aligned}
\end{equation*}
When $m>0$ and $M>0$, the function $\overline{q}$ is concave. Moreover $\overline{q}$ is well-defined and $\mathcal{C}^{1}$ on $\hat{J} \setminus D$. Additionally, it is easy to verify that $\overline{q}(T)=\overline{q}(-T)$.

Considering all the above, we conclude the proof of the result.
\end{proof}

In the general case, when $m$ and $M$ are not necessary positive, we can prove the following proposition that restricts the points $s \in \hat{J}$ where the function attains its minimum.

\begin{proposition}
	When the function $\overline{H}_{m,M}$ is positive, it can attain its minimum only at the following points:
	
	\begin{itemize}
		\item \underline{If $m < 0$}, at $s = n^+$ or $s=n^-$ with $n \in \mathbb{Z} \cap \hat{J}$ and $t \in \hat{J}$;
		
		\item \underline{If $m > 0$}: either
		\begin{itemize}
			\item at $s = n^+$ or $s = n^-$, with $n \in \mathbb{Z} \cap \hat{J}$ and $t \in \hat{J}$ or
			\item at $\overline{q}(s)$, with $s \in \hat{J}$ and $\overline{q}$ given by expression~\eqref{defq}.
		\end{itemize}
	\end{itemize}
\end{proposition}


\begin{proof}
Taking into account Proposition \ref{derivadas}, we deduce that if \( \overline{H}_{m,M} \) is positive then \( \frac{\partial}{\partial s} \overline{H}_{m,M}(t,s) \) has a constant sign (where it is well-defined). Specifically,  
\[
\frac{\partial}{\partial s} \overline{H}_{m,M}(t,s) = 
\begin{cases} 
>0, & \text{if } m > 0, \\[8pt]
<0, & \text{if } m < 0.
\end{cases}
\quad (t,s) \in \hat{J} \times \hat{J}, \, s \notin D_{t}
\]  
From this, we conclude that for any $t \in \hat{J}$ given, function \( \overline{H}_{m,M}(t,\cdot) \) is monotone with respect to $s$. Considering Proposition \ref{cara}, we have that:  
\begin{enumerate}
    \item The function \( \overline{H}_{m,M}(t,\cdot) \) is continuous for all \( s \in \hat{J} \), \( s \notin D_{t} \),
    \item \( \overline{H}_{m,M}(t,T) = \overline{H}_{m,M}(t,-T) \) for all \( t \in (-T, T) \),
    \item  $\overline{H}_{m,M}(s^+,s) - \overline{H}_{m,M}(s^-,s) = \overline{H}_{m,M}(s,s^-) - \overline{H}_{m,M}(s,s^+) = 1$ for all $s \in (-T,T)$, $s \notin D$.
\end{enumerate}  

As a direct consequence of the previous properties, we deduce that if function \( \overline{H}_{m,M}(t, \cdot) \) is increasing (i.e., \( m > 0 \)), the minimum can only occur at an integer value of \( s=n \) (which could be \( n^- \) or \( n^+ \)) with $s \in \hat{J}$ or at \( s = t^+ \) with $t \in \hat{J}$. On the other hand, when the function \( \overline{H}_{m,M}(t, \cdot) \) decreases (i.e., \( m < 0 \)), the minimum can only occur at an integer value of \( s=n \) (either \( n^- \) or \( n^+ \)) with $s \in \hat{J}$. It cannot occur at \( s = t^+ \) because, at this point, the derivative of \( \overline{H}_{m,M}(t, \cdot) \) is negative.

\end{proof}

\subsubsection{Analytical conjecture of the region of constant sign when \( T > 1 \)}

By numerically searching for the minimum of the function \( \overline{H}_{m,M}(t,s) \) when it is positive with \( (t,s) \in \hat{J} \times \hat{J} \), satisfying the restrictions already proven, we observe that, as in the case \( T < 1 \), the minimum seems to be reached at \( \overline{q}(0^-) \) if \( M > 0 \) or \( \overline{q}(0^+) \) if \( M < 0 \). Therefore, we conjecture the following:

\begin{conjecture}
If \( \overline{H}_{m,M}(t,s) > 0 \), the minimum of the function is attained at \( \overline{q}(0^-) \) if \( M > 0 \) or \( \overline{q}(0^+) \) if \( M < 0 \).
\end{conjecture}

If we accept this conjecture, we can obtain the explicit region and the analytical expressions in terms of \( m \) and \( M \) where the function \( \overline{H}_{m,M} \) is positive for \( T > 0 \) using either of the methods (the fixed-point method from Section \ref{metodofixo} or by setting the minimum of the function equal to zero). For example, we implemented the case with \( T = 1.6 \) in Mathematica and obtained the regions shown in Figure \ref{figura16}.

\begin{figure}[H]
	\begin{center}
	\includegraphics[scale=0.6]{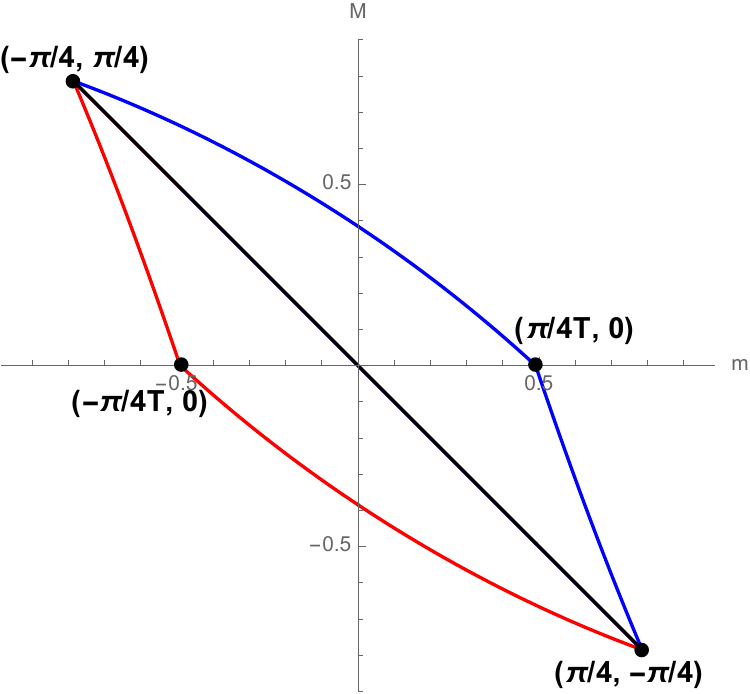}
	\end{center}
	\vspace{-20pt}
    \caption{Analytical determination of the region where \( \overline{H}_{m,M} \) is positive (the region bounded between the blue and black curves) and where \( \overline{H}_{m,M} \) is negative (the region bounded between the red and black curves) for \( T = 1.6 \).}
	\label{figura16}
\end{figure}

\subsubsection{Numerical approximation of the region of constant sign when \( T > 1 \)}
\label{secnume}

Even in the case where we do not know the point \((\hat{t},\hat{s})\) where the function \(\overline{H}_{m,M}\) attains its minimum when it is positive, we can numerically approximate the region where it has a constant sign. In general, this can be useful because it allows us to delineate this region knowing only the form of the Green's function.

We present a Python code that allows us to approximate this region for the function \(\overline{H}_{m,M}\) without making any conjectures. We divide the procedure into three cases:

\begin{itemize}
\item \underline{Case \(m>0\) and \(M>0\).} 

This is the simplest case because we already know that
\begin{equation*}
\min_{(t,s) \in \hat{J} \times \hat{J}} \overline{H}_{m,M} = \min\left\{{\{\overline{q}(s^+), \overline{q}(s^-), \overline{q}(T)\}}, \; s \in \{[-T], \ldots, 0, \ldots, [T]\}\right\}.
\end{equation*}
To approximate the region, we can use numerical variations of any of the methods previously studied, such as the fixed-point method or solving for the roots where the minimum of the function equals zero. For instance, using the second approach, we evaluate \(\overline{q}(T)\), \(\overline{q}(s^+)\) and \(\overline{q}(s^-)\) where \(s \in \hat{J}\) represents integer values, for different pairs of \((m,M)\). Subsequently, we compute the minimum value for each pair $(m,M)$. We then represent the region in terms of the values of \(m\) and \(M\) where this minimum is greater than zero. This region coincides with the region where the function \(\overline{H}_{m,M}\) is positive.

\item \underline{Case \(m<0\) and \(M>0\).}

In this case, we know that the minimum can only occur at \(s=n^+\) or $s=n^-$ with $n \in \mathbb{Z} \cap \hat{J}$. Here, we use a numerical variation of the fixed-point method, as it is computationally much more efficient. This efficiency arises because it requires far fewer evaluations of the function \(\overline{H}_{m,M}\), which depends on the inverse of a matrix whose size increases with \(T\). From equation \eqref{exprem22}, we know that
\begin{equation}
\begin{aligned}
\overline{T}_{m}^{0}(M,t,s) &= \frac{\overline{G}_{m}(t,s)}{\int_{-T}^{T} \overline{G}_{m}(t,r) \overline{H}_{m,M}([r],s) \, \mathrm{d}r} \\
&= \frac{\overline{G}_{m}(t,s)}{\overline{H}_{m,M}([-T],s) \int_{-T}^{[-T]} \overline{G}_{m}(t,r) \, \mathrm{d}r + \cdots + \overline{H}_{m,M}([T],s) \int_{[T]}^{T} \overline{G}_{m}(t,r) \, \mathrm{d}r}.
\end{aligned}
\label{maxt}
\end{equation}

For any \(m<0\), following Theorem  \ref{positividadg} we know that the function \(\overline{G}_{m}\) is not positive, and thus there exists at least one point \((t_{m},s_{m}) \in \hat{J} \times \hat{J}\) verifying that \(\overline{G}_{m}(t_{m},s_{m})<0\). From equations \eqref{emfixed} and \eqref{exprem} with \(M_{1}=0\), we deduce that \(\int_{-T}^{T} \overline{G}_{m}(t_{m},r) \overline{H}_{m,M_{0}}([r],s) \, \mathrm{d}r < 0\). Finally, taking into account equation \eqref{mcerta}, we have
\begin{equation*}
M_{0} > \overline{T}_{m}^{0}(M_{0},t_{m},s_{m}).
\end{equation*}
Following the reasoning in Section \ref{metodofixo}, we seek the global maximum of $\overline{T}_{m}^0(M,t,s)$ for $(t,s) \in \hat{J} \times \hat{J}$. If this maximum is attained, we will keep that value.


Thus, the Python code will compute the value of \(\overline{T}_{m}^{0}(M,t,s)\) for different pairs of values \((m,M)\) and search for the maximum for values \(t \in \hat{J}\) and integer $s=n^+$ or $s=n^-$ with $s \in \hat{J}$. Given this, we will represent the region, as a function of the parameters $m$ and $M$, where this maximum value ($\max_{(t,s) \in \hat{J} \times \hat{J}}{\overline{T}_{m}^{0}(M,t,s)}$) is less than $M$. This region must coincide with the region where $\overline{H}_{m,M}$ is positive.
 To optimize the process, the code will precompute values of \(\overline{H}_{m,M}([t],s)\) and the different integrals of \(\overline{G}_{m}\) to avoid repetitive calculations in the loops. Furthermore, parallelization techniques will also be employed to improve efficiency. The Python code is included in the appendix.

\item \underline{Case \(m>0\) and \(M<0\).}

In this case, the minimum of the function \(\overline{H}_{m,M}\) may occur when \(s=n^-\) or $s=n^+$ with $n \in \mathbb{Z} \cap \hat{J}$ for any \(t \in \hat{J}\), or $\overline{q}(s)$ with $s \in \hat{J}$ and $\overline{q}$ given by expression \eqref{defq}. For integer values of \(s \in \hat{J}\), we follow a similar procedure to the case \(m < 0\) and \(M > 0\). Using analogous steps, we determine that it is necessary to search for the minimum of \(\overline{T}_{m}^{0}(M_{0},t,s)\).

For the case \(t = s^+ \in \hat{J}\), the procedure is analogous to the one used in the case \(m > 0\) and \(M > 0\).

\end{itemize}

Following the previously described steps, we approximate the region where the function $\overline{H}_{m,M}$ is positive, as shown in Figure \ref{aproxregion}. A similar procedure can be employed to determine the regions where $\overline{H}_{m,M}$ is negative. Alternatively, the symmetry property outlined in Lemma \ref{simetriah} can be utilized to infer this region directly.

\begin{figure}[H]
	\begin{center}
	\includegraphics[scale=0.25]{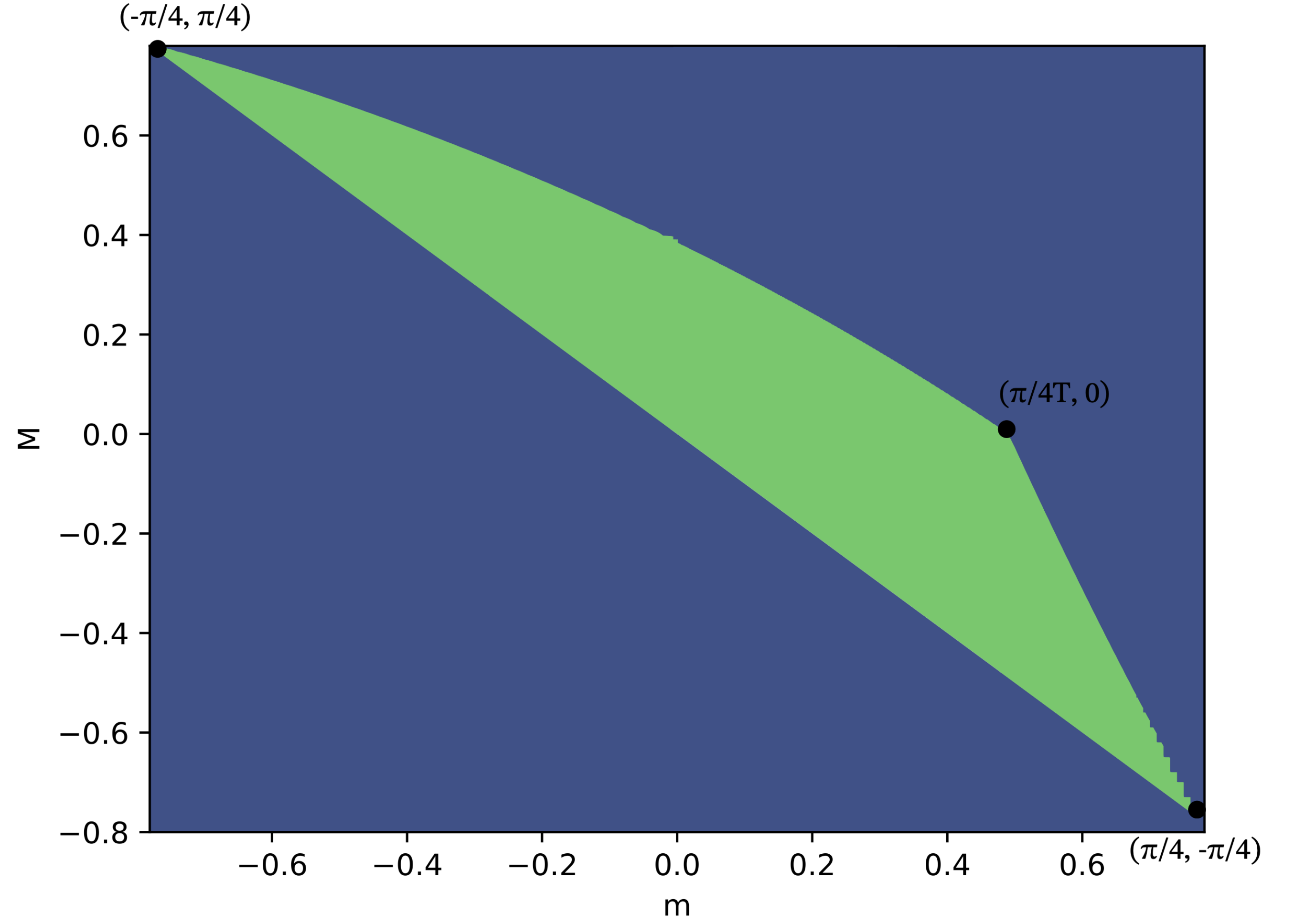}
	\end{center}
	\vspace{-20pt}
    \caption{Approximation of the region where the function $\overline{H}_{m,M}$ is positive for $T=1.6$.}
	\label{aproxregion}
\end{figure}

We observe that this region seems to be the same as the one shown in Figure \ref{figura16}.

\section{Existence of solution to nonlinear problems}
\label{sec6}
In this section, we will present some results that allow us to deduce the existence of solutions for nonlinear problems with periodic boundary conditions.

We will obtain existence results by using the monotone method, and employing the concepts of lower and upper solutions. 

\subsection{Monotone method}
One of the most common techniques for studying the existence and, in some cases, the construction of solutions to differential equations is the method of lower and upper solutions.

We will follow the classical approach for this type of problem (see \cite{cabada2014green} and references therein) to establish conditions under which the first order periodic problem with reflection and piecewise constant arguments,
\begin{equation}
v'(t)=f(t,v(-t),v([t])) \quad \textup{a.e. } t \in \hat{J}, \quad v(-T)=v(T),
\label{promono}
\end{equation}
with $T>0$ have solution.

We assume that $f:\hat{J} \times \mathbb{R} \times \mathbb{R} \rightarrow \mathbb{R}$ is a Carathéodory function, that is:
\begin{itemize}
    \item For almost every $t \in \hat{J}$, the function $(x,y) \in \mathbb{R} \times \mathbb{R} \rightarrow f(t,x,y) \in \mathbb{R}$ is continuous.
    \item For all $(x,y) \in \mathbb{R} \times \mathbb{R}$, the function $t \in \hat{J} \rightarrow f(t,x,y)$ is Lebesgue measurable.
    \item For all $R > 0$, there exists $h_{R} \in \mathcal{L}^{1}(\hat{J})$ such that, if $|x| < R$ and $|y| < R$, then
    \begin{equation*}
    |f(t,x,y)| \leq h_{R}(t) \textup{ a.e. } t \in \hat{J}.
    \end{equation*}
\end{itemize}

\begin{definition}
We say that $\alpha \in AC(\hat{J})$ is a lower solution of \eqref{promono} if $\alpha$ satisfies
\begin{equation*}
\alpha'(t) \geq f(t, \alpha(-t),\alpha([t])) \textup{ a.e. }t \in \hat{J}, \quad \alpha(-T)-\alpha(T) = 0.
\end{equation*}
\end{definition}
\begin{definition}
We say that $\beta \in AC(\hat{J})$ is an upper solution of \eqref{promono} if $\beta$ satisfies
\begin{equation*}
\beta'(t) \leq f(t, \beta(-t),\beta([t])) \textup{ a.e. }t \in \hat{J}, \quad \beta(-T)-\beta(T) = 0.
\end{equation*}
\end{definition}

We now present a theorem in line with the monotone method of lower and upper solutions that ensures the existence of a solution for Problem \eqref{promono} under certain additional conditions.
\begin{theorem}
Let us assume that the following hypotheses hold:
\begin{enumerate}
    \item There exist $\alpha$ and $\beta$, a pair of lower and upper solutions of Problem \eqref{promono}, such that $\beta \leq \alpha$ on $[-T,T]$.
    \item The function $f$ is a Carathéodory function satisfying
    \begin{equation*}
    f(t,x_{1},y_{1}) - f(t,x_{2},y_{2}) \geq -m(x_{1}-x_{2}) - M(y_{1}-y_{2})
    \end{equation*}
    for almost every $t \in \hat{J}$, with $\beta(-t) \leq x_{2} \leq x_{1} \leq \alpha(-t)$ and $\beta([t]) \leq y_{2} \leq y_{1} \leq \alpha([t])$.
    \item The pair $(m,M)$ satisfies that the Green's function $\overline{H}_{m,M}$ associated with Problem \eqref{im1} is positive. (If $T \in (0,1]$, this holds when $m \in (-\frac{\pi}{4T},\frac{\pi}{4T})$ and $-m < M < \frac{m}{2}(-1 + \cot(mT))$).
\end{enumerate}
Then, there exist two monotone sequences $(\alpha_{n})_{n \in \mathbb{N}}$ and $(\beta_{n})_{n \in \mathbb{N}}$, decreasing and increasing respectively, with $\alpha_{0} = \alpha$, $\beta_{0} = \beta$, which converge uniformly to the extremal solutions on $[\beta, \alpha]$ of Problem \eqref{promono}, respectively.
\label{teomono}
\end{theorem}

\begin{proof}
Consider the Problem
\begin{equation}
\begin{split}
v'(t)+mv(-t)+Mv([t])&=f(t, \gamma(-t), \gamma([t]))+m\gamma(-t)+M\gamma([t]) \textup{ a.e. }t \in \hat{J}, \\
\quad v(-T)&=v(T)
\label{promono2}
\end{split}
\end{equation}
where $\gamma \in \mathcal{L}^{1}(\hat{J})$ verifies that $\beta \leq \gamma \leq \alpha$. Then, by Condition $2$, we know that
\begin{equation*}
\begin{aligned}
(\alpha-v)'(t)+m(\alpha-v)(-t)+M(\alpha-v)([t])  \\
\geq f(t, \alpha(-t), \alpha([t]))-f(t, \gamma(-t),\gamma([t]))+m(\alpha-\gamma)(-t)+M(\alpha-\gamma)([t])
&=h_{\gamma(t)} \geq 0 \textup{ a.e }t \in \hat{J}, \\
(\alpha-v)(-T) &= (\alpha-v)(T).
\end{aligned}
\end{equation*}

From the previous inequalities and the regularity of function $\gamma$, we know that $h_{\gamma} \in \mathcal{L}^{1}(\hat{J})$ and we have that
\begin{equation*}
(\alpha-v)(t)=\int_{-T}^{T}{\overline{H}_{m,M}(t,s)h_{\gamma}(s) \mathrm{d}s} \geq 0 \quad \forall \, t \in \hat{J}.
\end{equation*}

Let us now consider $v_{i} = \mathcal{T} \gamma_{i}$, where the operator $\mathcal{T} \gamma_{i}$ is given by
\begin{equation}
(\mathcal{T} \gamma_{i})(t):= \int_{-T}^{T}{\overline{H}_{m,M}(t,s)[f(s,\gamma_{i}(s), \gamma_{i}([s]))+m\gamma_{i}(-s)+M \gamma_{i}([s])] \mathrm{d}s}.
\label{aproxt}
\end{equation}

It is easy to see that the operator $\mathcal{T}$ is continuous.
Moreover, $v_{i}$ is the unique solution of Problem \eqref{promono2} for $\gamma = \gamma_{i} \in \mathcal{L}^{1}(J)$, and suppose that $\beta \leq \gamma_{1} \leq \gamma_{2} \leq \alpha$. Then,
\begin{equation*}
\begin{split}
(v_{2}-v_{1})'(t)+m(v_{2}-v_{1})(-t)+M(v_{1}-v_{2})([t])&=f(t, \gamma_{2}(-t), \gamma_{2}([t]))-f(t, \gamma_{1}(-t), \gamma_{1}([t])) \\
& \quad +m(\gamma_{2}-\gamma_{1})(-t)+M(\gamma_{2}-\gamma_{1})([t]) \geq 0 \textup{ a.e. }\hat{J}.
\end{split}
\end{equation*}
\begin{equation*}
(v_{2}-v_{1})(-T)=(v_{2}-v_{1})(T).
\end{equation*}
Therefore, $v_{1} \leq v_{2}$ on $\hat{J}$.

Consequently, we can construct the mentioned sequences $(\alpha_{n})_{n \in \mathbb{N}}$ and $(\beta_{n})_{n \in \mathbb{N}}$ as follows. We take $\alpha_{0} = \alpha$, $\beta_{0} = \beta$, $\alpha_{n+1} = \mathcal{T}\alpha_{n}$ and $\beta_{n+1} = \mathcal{T}\beta_{n}$ for all $n \in \mathbb{N}$. Moreover, the sequences $(\alpha_{n})_{n \in \mathbb{N}}$ and $(\beta_{n})_{n \in \mathbb{N}}$ are monotone non increasing and non decreasing respectively, and are bounded on $[\beta, \alpha]$. By Dini's Theorem, we can ensure that both sequences converge uniformly on $\hat{J}$.

It is easily verified that the sequences $(\alpha_{n})_{n \in \mathbb{N}}$ and $(\beta_{n})_{n \in \mathbb{N}}$ converge to the extremal solutions $\phi = \mathcal{T}\phi$ and $\psi = \mathcal{T}\psi$ of Problem \eqref{promono}.
\end{proof}
In a similar manner, we arrive at the following Theorem.
\begin{theorem}
Assume that the following hypotheses are satisfied:
\begin{enumerate}
\item There exist $\alpha$ and $\beta$, a pair of lower solution and upper solution of Problem \eqref{promono}, such that $\alpha \leq \beta$ on $[-T,T]$.
\item The function $f$ is a Carathéodory function satisfying
\begin{equation*}
f(t,x_{1},y_{1}) - f(t,x_{2},y_{2}) \leq -m(x_{1} - x_{2}) - M(y_{1} - y_{2})
\end{equation*}
for almost every $t \in \hat{J}$ with $\alpha(-t) \leq x_{2} \leq x_{1} \leq \beta(-t)$ and $\alpha([t]) \leq y_{2} \leq y_{1} \leq \beta([t])$.
\item The pair $(m,M)$ satisfies that the Green's function $\overline{H}_{m,M}$ associated with Problem \eqref{im1} is negative. When $T \in (0,1]$, this holds for $m \in (-\frac{\pi}{4T}, \frac{\pi}{4T})$ and $-\frac{m}{2}(1 + \cot(mT)) < M < -m$.
\end{enumerate}
Then, there exist two monotone sequences $(\alpha_{n})_{n \in \mathbb{N}}$, $(\beta_{n})_{n \in \mathbb{N}}$, increasing and decreasing respectively, with $\alpha_{0} = \alpha$, $\beta_{0} = \beta$, which converge uniformly to the extremal solutions on $[\alpha, \beta]$ of \eqref{promono}, respectively.
\end{theorem}
Let us now look at a couple of practical examples where we can use the monotone method.
\begin{example}
Consider the Problem
\begin{equation}
v'(t)=\lambda \tanh{(t^2-2v(-t)+v([t]))}, \quad \textup{ a.e. } t \in \hat{J}, \quad v(-1)=v(1).
\label{ejemplomono2}
\end{equation}

It is easily verified that $\alpha \equiv 1$ and $\beta \equiv -1$ are lower solution and upper solution, respectively, for Problem \eqref{ejemplomono2} for all $\lambda \geq 0$. Moreover, $f(t,x,y) = \lambda \tanh{(t^2 - 2x + y)}$ satisfies $\left|\frac{\partial{f}}{\partial{x}}(t,x,y)\right| \leq 2 \lambda$ and $\left|\frac{\partial{f}}{\partial{y}}(t,x,y)\right| \leq \lambda$ for all $t \in \hat{J}$, $x$, $y \in \mathbb{R}$. Therefore, if we choose $m$ and $M \in \mathbb{R}$ such that the Green's function $\overline{H}_{m,M}$ with $T=1$ is positive and
\begin{equation*}
0 \leq \lambda \leq \min\{\frac{m}{2},M\},
\end{equation*}
by Theorem \ref{promono}, we can ensure that Problem \eqref{ejemplomono2} has extremal solutions on the sector $[-1,1]$.

From the previous section, we know that if $m = 1/2$ and $M = 1/5$, then $\overline{H}_{m,M} > 0$ for $T = 1$. Therefore, we will analyze the following problem.

\begin{equation}
v'(t)=\frac{\tanh{(t^2-2v(-t)+v([t]))}}{5}, \quad  t \in [-1,1], \quad v(-1)=v(1).
\label{ejemplomono2par}
\end{equation}

By using Python code, we calculate the sequences $(\alpha_{n})_{n \in \mathbb{N}}$ and $(\beta_{n})_{n \in \mathbb{N}}$ and plot them in the graphs of Figures \ref{ejemono22} and \ref{ejemono2}.

\begin{figure}[H] 
    \centering
    \subfigure[Sequence $(\alpha_{n})_{n \in \mathbb{N}}$ associated with Problem \eqref{ejemplomono2par}. We observe that the sequence is indeed non increasing.]{
        \includegraphics[width=0.45\textwidth]{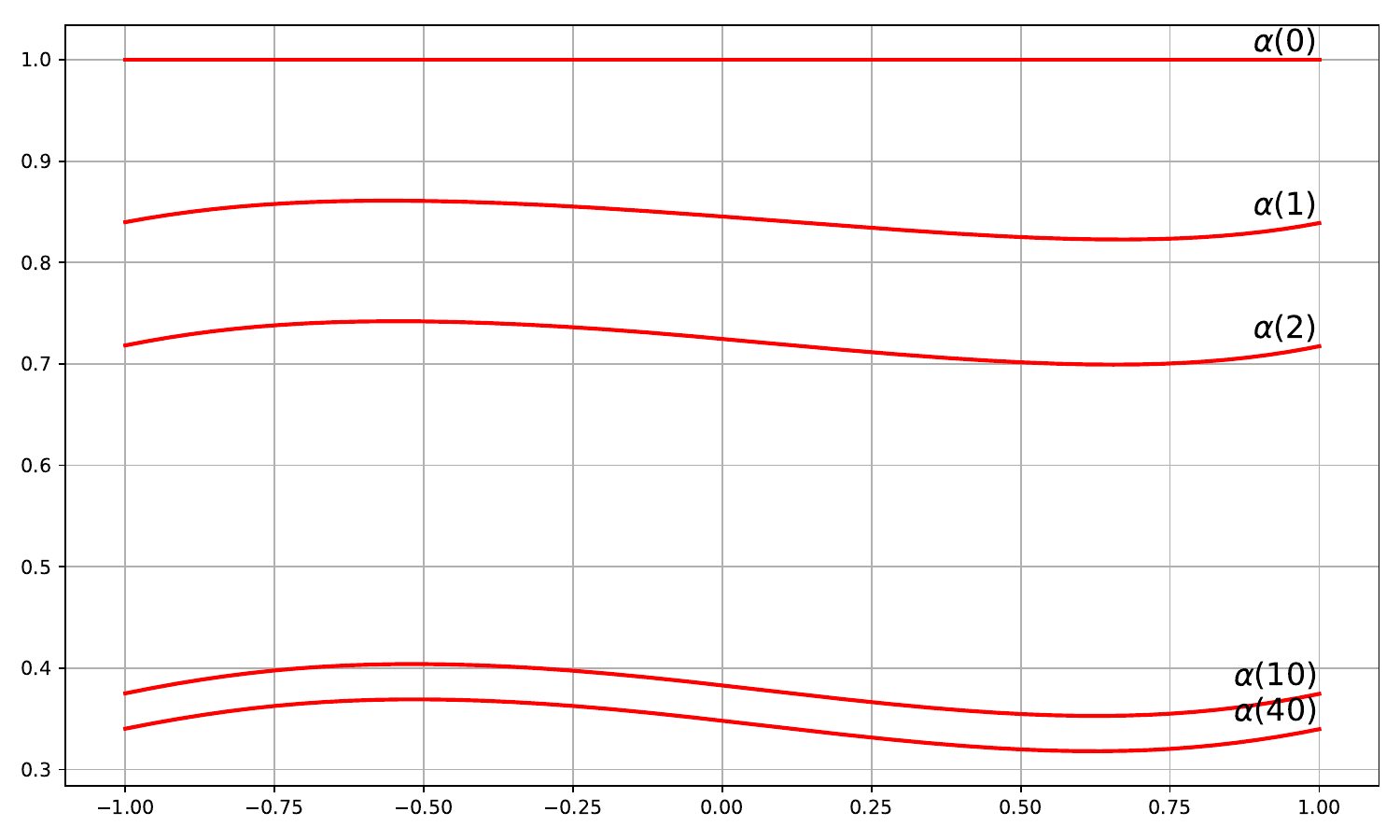}
    }
    \hspace{0.03\textwidth}
    \subfigure[Sequence $(\beta_{n})_{n \in \mathbb{N}}$ associated with Problem \eqref{ejemplomono2par}. We observe that the sequence is indeed non decreasing.]{
        \includegraphics[width=0.45\textwidth]{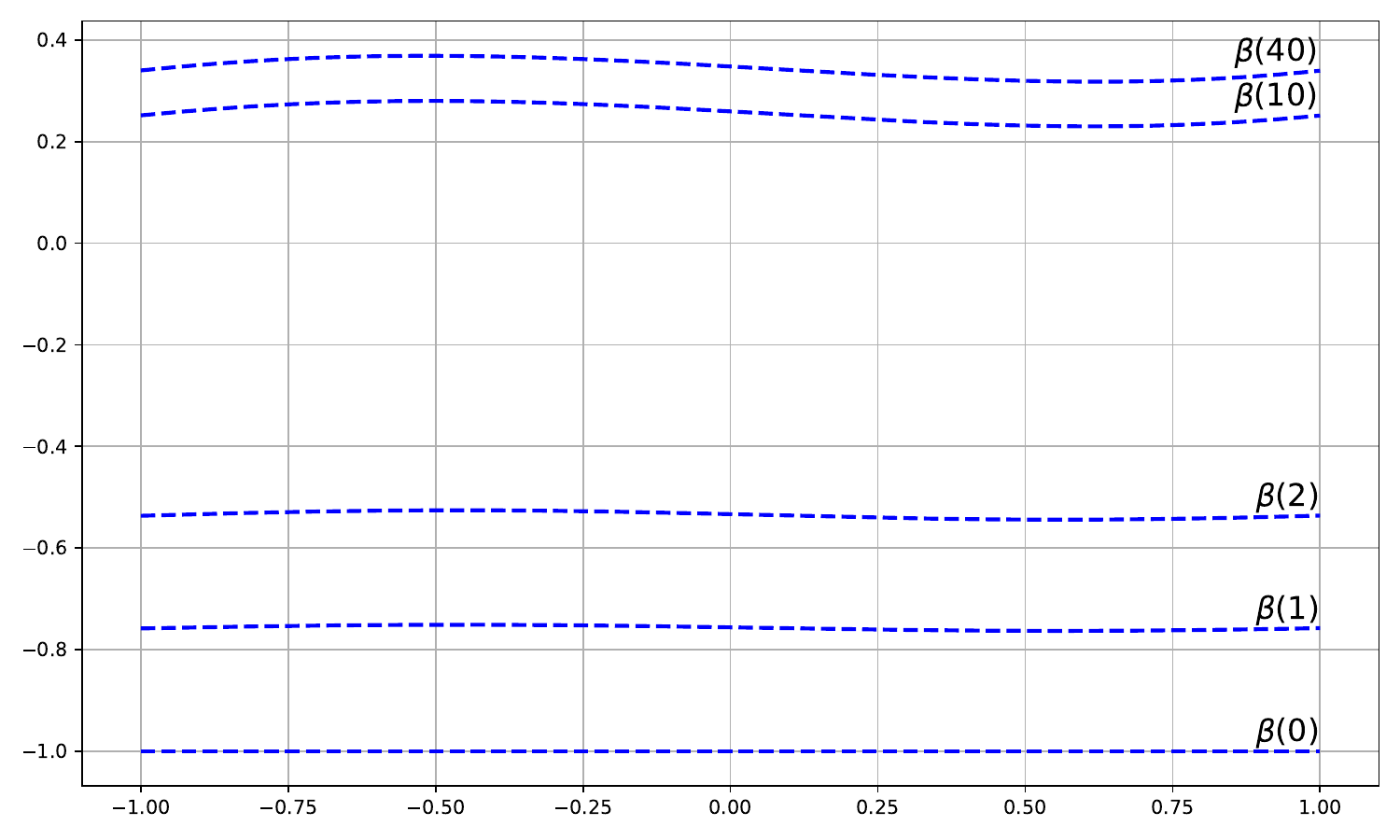}
    }
\caption{Practical application of the monotone method to the particular Problem \eqref{ejemplomono2par}. }
\label{ejemono22}
\end{figure}
\begin{figure}[H]
    \subfigure[Iterates $\alpha_{20}$ and $\beta_{20}$ of the sequences associated with Problem \eqref{ejemplomono2par}.]{
        \includegraphics[width=0.45\textwidth]{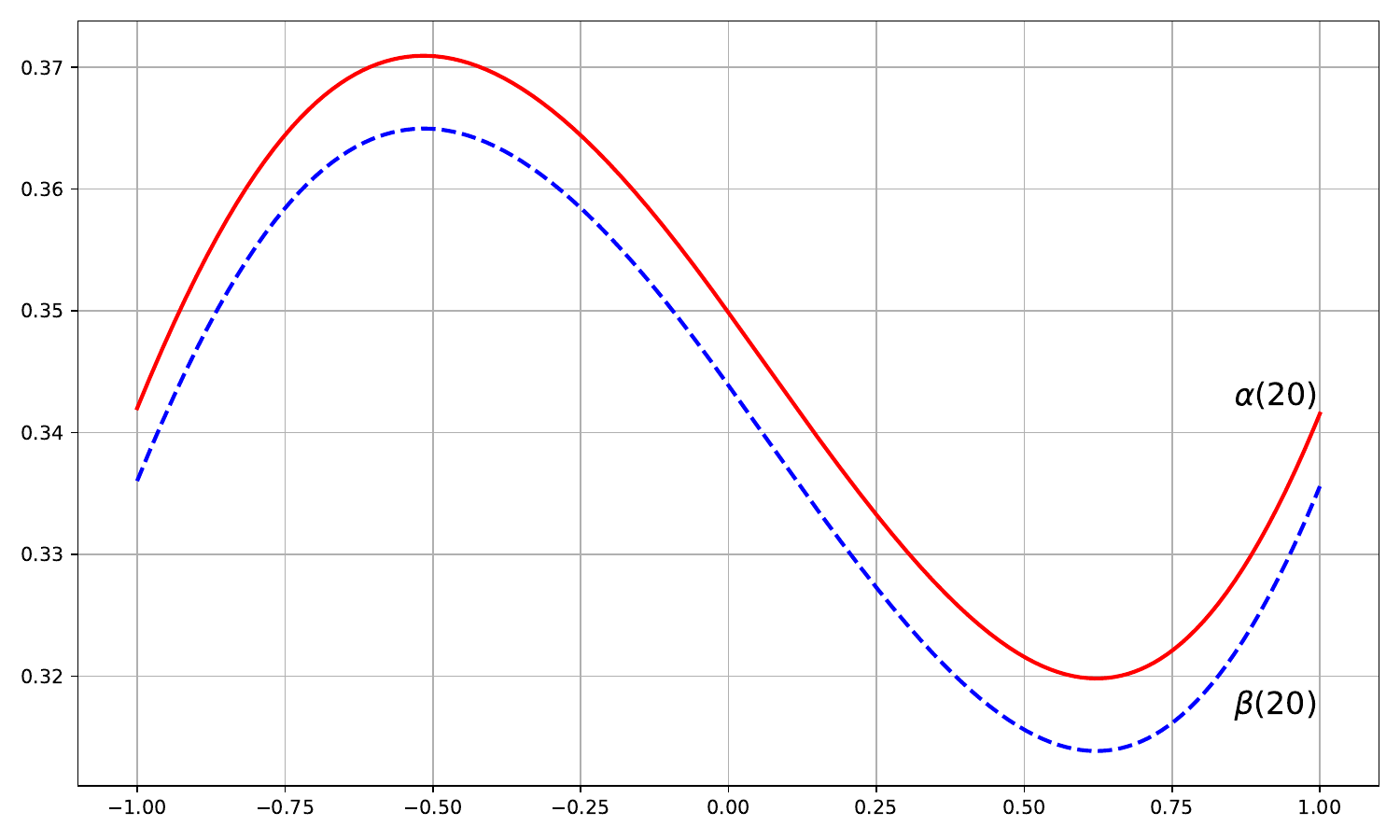}
    }
    \hspace{0.03\textwidth}
    \subfigure[Iterates $\alpha_{40}$ and $\beta_{40}$ of the sequences associated with Problem \eqref{ejemplomono2par}.]{
        \includegraphics[width=0.45\textwidth]{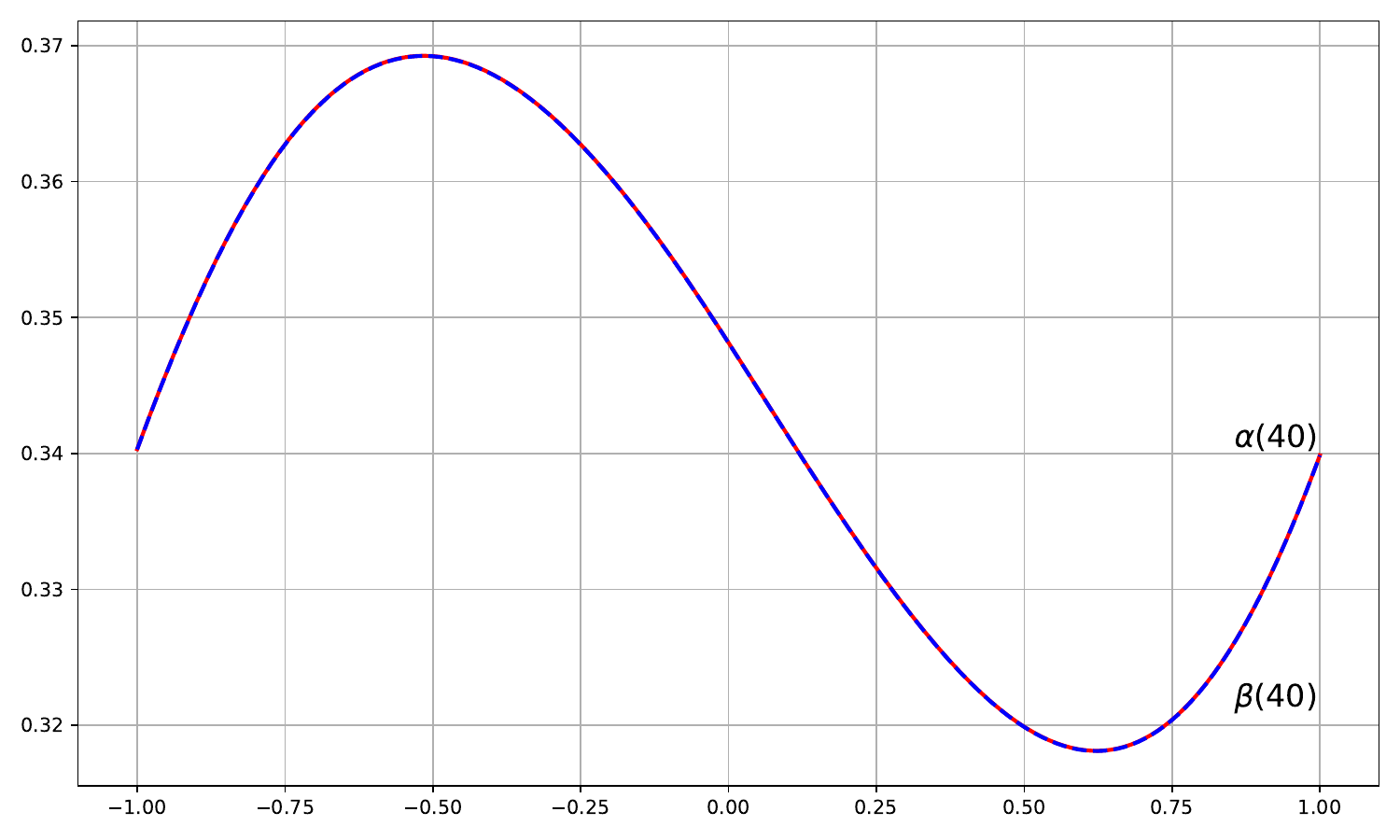}
    }
    \caption{Approximation of the extremal solutions of Problem \eqref{ejemplomono2par} and their similarity.}
    \label{ejemono2}
\end{figure}

Observing the previous Figure \ref{ejemono2}, we see that the extremal solutions appear to be identical, which allows us to conjecture that Problem \eqref{ejemplomono2par} has a unique solution on the sector $[-1,1]$.

\end{example}

\begin{example}
We now consider the problem
\begin{equation}
v'(t)=\lambda \tanh{(t-v(-t)-v([t]))}, \quad t \in \hat{J}, \quad v(-1.6)=v(1.6).
\label{ejemplomono}
\end{equation}

It is easy to verify that $\alpha \equiv \frac{T}{2}$ and $\beta \equiv -\frac{T}{2}$ are lower and upper solutions, respectively, for the problem \eqref{ejemplomono} for all $\lambda \geq 0$. Moreover, $f(t,x,y)=\lambda \tanh{(t-x-y)}$ satisfies that $\left|\frac{\partial f}{\partial{x}}(t,x,y)\right| \leq \lambda$ and $\left|\frac{\partial f}{\partial{y}}(t,x,y)\right| \leq \lambda$ for all $t \in \hat{J}$, $x$, $y \in \mathbb{R}$.

Numerically, we calculate the matrix $A$ related to Problem \eqref{im1} with $T=1.6$, $m=0.21$ and $M=0.2$, obtaining 
\begin{equation*}
A=
\begin{pmatrix}
    1.23 & 0.52 & 0.20 \\
   0.19 & 1.59 & 0.17\\
    0.15 & 0.67 & 1.13  
\end{pmatrix}
\end{equation*}
and we verify that the Green's function with these parameters is positive.

With all of the above, and applying Theorem \ref{teomono}, we can assert that Problem \eqref{ejemplomono} has extremal solutions on $[-T/2,T/2]$ for all
\begin{equation*}
0 \leq \lambda \leq \min\{m,M\}=\frac{1}{5}.
\end{equation*}

Now, we consider the particular problem
\begin{equation}
v'(t)=\frac{\tanh{(t-v(-t)-v([t]))}}{5}, \quad  t \in [-1.6,1.6], \quad v(-1.6)=v(1.6),
\label{monopar}
\end{equation}
and we create a program that allows us to obtain approximately two sequences $(\alpha_{n})_{n \in \mathbb{N}}$ and $(\beta_{n})_{n \in \mathbb{N}}$ defined as shown in Theorem \ref{teomono}, which should approximate the extremal solutions of the problem. Doing this, we obtain the following graphs.

\begin{figure}[H] 
    \centering
    \subfigure[Sequence of $(\alpha_{n})_{n \in \mathbb{N}}$ related to Problem \eqref{monopar}. We can see that the sequence is indeed decreasing.]{
        \includegraphics[width=0.45\textwidth]{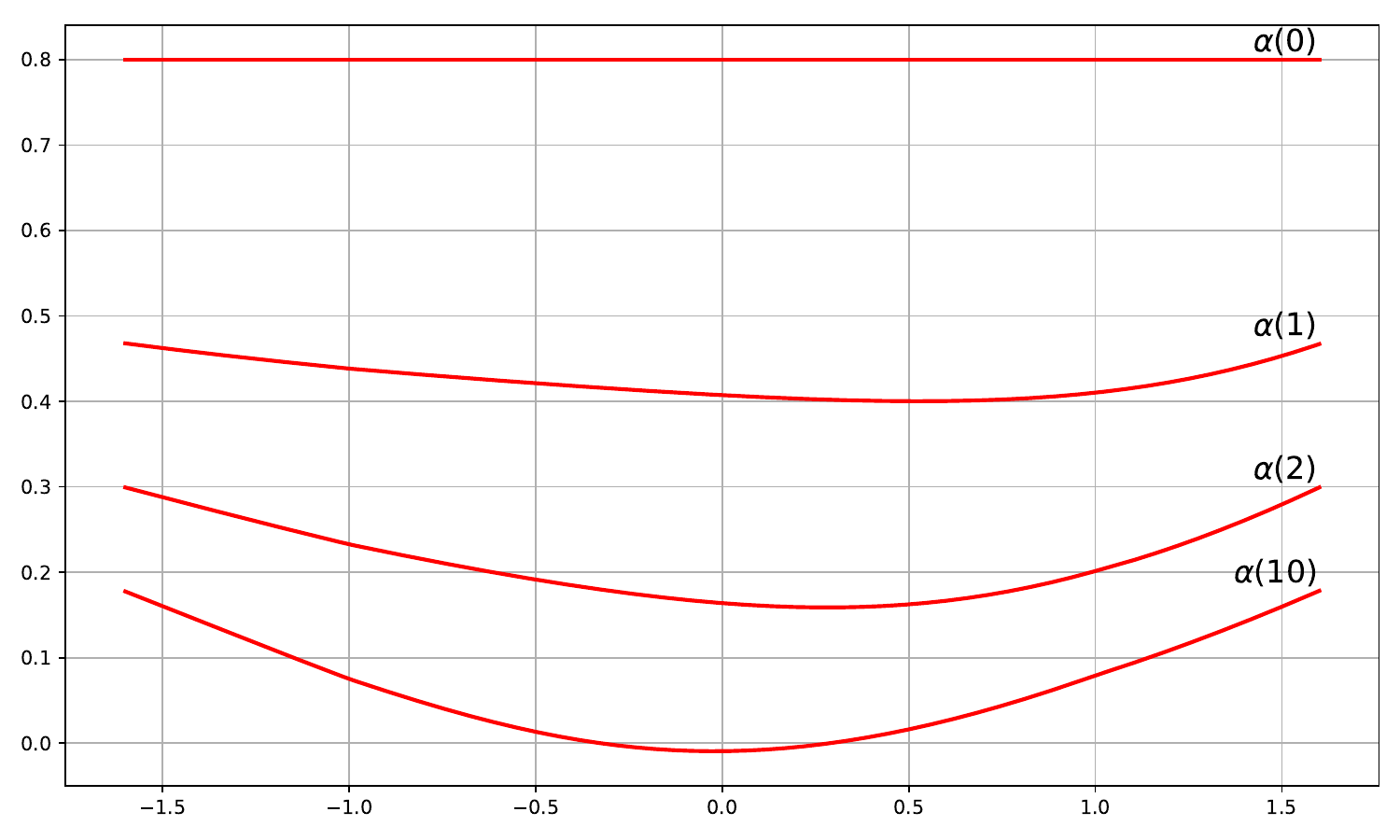}
    }
    \hspace{0.03\textwidth}
    \subfigure[Sequence of $(\beta_{n})_{n \in \mathbb{N}}$ related to Problem \eqref{monopar}. We can see that the sequence is indeed increasing.]{
        \includegraphics[width=0.45\textwidth]{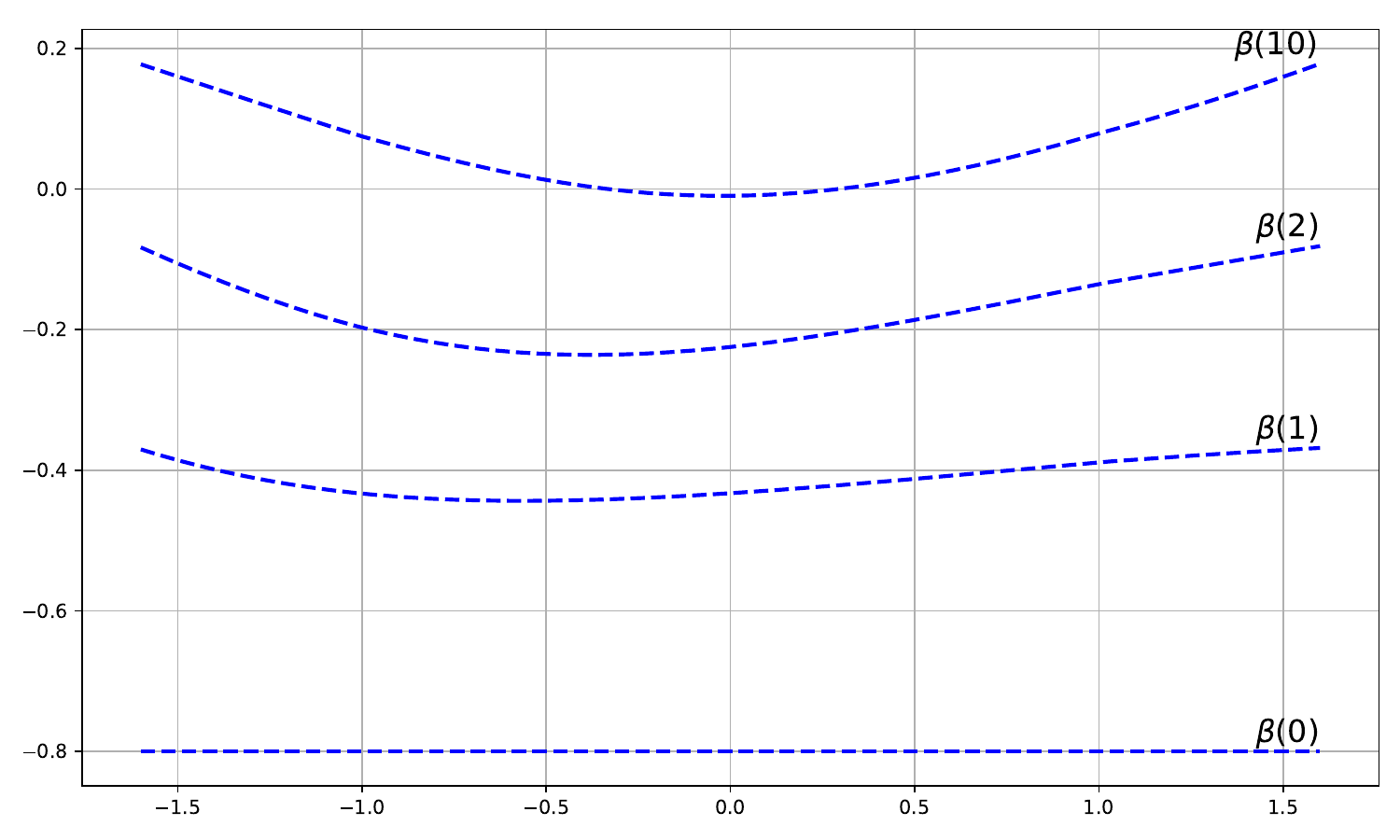}
    }
\caption{Practical application of the monotone method to the particular Problem \eqref{monopar}. }
\label{fig21}
\vspace{-0.6cm}
\end{figure}
\begin{figure}[H]
    \subfigure[Iterates $\alpha_{4}$ and $\beta_{4}$ of the sequences related to Problem \eqref{monopar}.]{
        \includegraphics[width=0.45\textwidth]{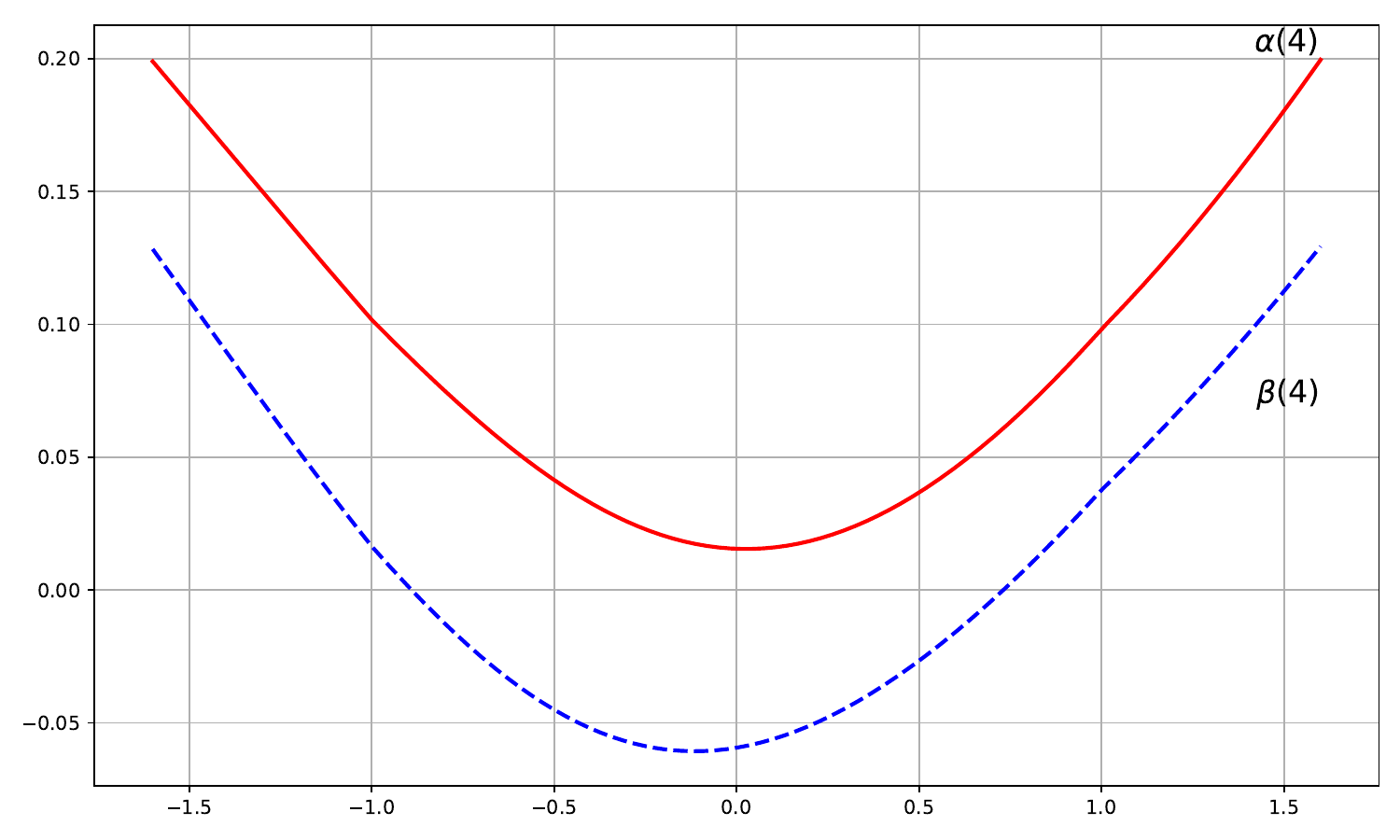}
    }
    \hspace{0.03\textwidth}
    \subfigure[Iterates $\alpha_{10}$ and $\beta_{10}$ of the sequences related to Problem \eqref{monopar}.]{
        \includegraphics[width=0.45\textwidth]{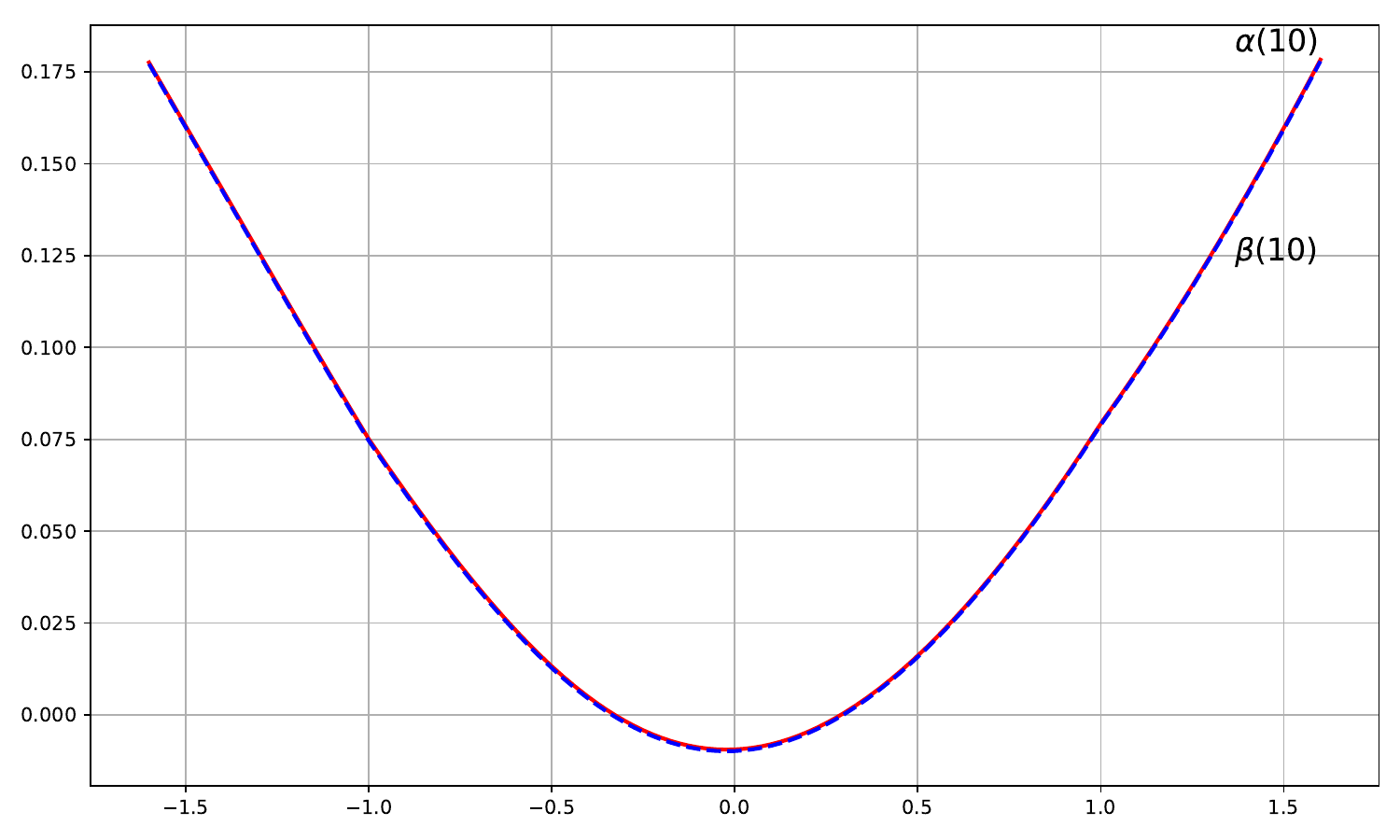}
    }
    \caption{Approximation of the extremal solutions of Problem \eqref{monopar} and their similarity.}
    \label{fig22}
\end{figure}


From the previous figures \ref{fig21} and \ref{fig22}, we can observe that the extremal solutions seem to be the same, from which we can conjecture that Problem \eqref{monopar} has a unique solution lying between the constant functions $-T/2$ and $T/2$.

\end{example}

\section*{Appendix}
\addcontentsline{toc}{section}{Appendix}

In this section, we present part of the numerical code used throughout the work.

\subsection*{Obtaining the function $\overline{H}_{m,M}$}
First, we show how to numerically define the function $\overline{H}_{m,M}$ that satisfies equation \eqref{final1}.

\begin{lstlisting}[language=Python, caption={Numerical calculation of the function $H_{m,M}$.}, label={lst funcion}]
import numpy as np
from scipy.integrate import quad

def H(t, s, m, TT, MM):
    def custom_floor(x):
        return np.floor(x) if x >= 0 else np.ceil(x)
    
    def integrand(s, i, m, TT, MM):
        return MM * G(i - int(custom_floor(TT)) - 1, s, m, TT)
    
    def funcion(m, MM, TT):
        n = 2 * int(custom_floor(TT)) + 1
        matriz = np.zeros((n, n))
        for i in range(1, n+1):
            for j in range(1, n+1):
                if j == 1:
                    matriz[i-1, j-1], _ = quad(integrand, -TT, -custom_floor(TT), args=(i, m, TT, MM))
                elif j == (n + 1) // 2:
                    matriz[i-1, j-1], _ = quad(integrand, -1, 1, args=(i, m, TT, MM))
                elif j == 2 * int(custom_floor(TT)) + 1:
                    matriz[i-1, j-1], _ = quad(integrand, custom_floor(TT), TT, args=(i, m, TT, MM))
                elif j < (n + 1) // 2:
                    s1 = int(custom_floor(j - TT - 2))
                    s2 = int(custom_floor(j - TT - 1))
                    matriz[i-1, j-1], _ = quad(integrand, s1, s2, args=(i, m, TT, MM))
                elif j > (n + 1) // 2:
                    s1 = int(custom_floor(j - TT))
                    s2 = int(custom_floor(j - TT + 1))
                    matriz[i-1, j-1], _ = quad(integrand, s1, s2, args=(i, m, TT, MM))
        inversa = np.linalg.inv(matriz + np.eye(n))
        
        def alpha(r):
            alpha_matrix = np.zeros((n, n))
            for i in range(1, n+1):
                for j in range(1, n+1):
                    if i == 1:
                        alpha_matrix[i-1, j-1] = inversa[i-1, j-1] * (1 if -TT <= r <= custom_floor(-TT) else 0)
                    elif i == n:
                        alpha_matrix[i-1, j-1] = inversa[i-1, j-1] * (1 if custom_floor(TT) <= r <= TT else 0)
                    elif i == (n + 1) // 2:
                        alpha_matrix[i-1, j-1] = inversa[i-1, j-1] * (1 if -1 <= r <= 1 else 0)
                    elif i < (n + 1) // 2:
                        alpha_matrix[i-1, j-1] = inversa[i-1, j-1] * (1 if custom_floor(i - TT - 2) <= r <= custom_floor(i - TT - 1) else 0)
                    elif i > (n + 1) // 2:
                        alpha_matrix[i-1, j-1] = inversa[i-1, j-1] * (1 if custom_floor(i - TT) <= r <= custom_floor(i - TT + 1) else 0)
            return alpha_matrix
        
        def filalpha(r):
            return np.sum(alpha(r), axis=0)
        
        def sumatorio(s, r):
            return sum(G(j - int(custom_floor(TT)) - 1, s, m, TT) * filalpha(r)[j-1] for j in range(1, n+1))
        
        return lambda s, r: sumatorio(s, r)
    
    sumatorio_func = funcion(m, MM, TT)
    integrand = lambda r: G(t, r, m, TT) * sumatorio_func(s, r)
    integral, _ = quad(integrand, -TT, TT)
    
    return G(t, s, m, TT) - MM * integral
\end{lstlisting}

In this code, we follow the steps outlined in Section \ref{sec3}. We compute the elements of the matrix $A$ given in \eqref{matriz}, its inverse, and the elements $\alpha_{ij}(r)$ given in \eqref{g2}. With all these components, and following equation \eqref{final1}, we obtain the function $\overline{H}_{m,M}$.

\subsection*{Numerical approximation of the constant sign of $\overline{H}_{m,M}$}

In this case, we will present the code necessary to approximate the region where $\overline{H}_{m,M}$ is positive when $T=1.6$. We will deal with the case when $m>0$ and $M>0$, or $m<0$ and $M>0$. The case $m<0$ and $M>0$ could be addressed as a combination of both.

\begin{lstlisting}[language=Python, caption={Code for the numerical approximation of the region where $H_{m,M}> 0$ with $m>0$ and $M>0$.}, label={lst:approx-region}]
import numpy as np
import matplotlib.pyplot as plt
from scipy.interpolate import RegularGridInterpolator
from joblib import Parallel, delayed
mmRange = np.arange(0.001, 0.52, 0.01)
MMRange = np.arange(0, 0.79, 0.01)
TT = 1.6 
epsilon = 10**-6
sValues = np.array([-TT, -1 - epsilon, -1 + epsilon, 0 - epsilon, 
                    0 + epsilon, 1 - epsilon, 1 + epsilon, TT])
MMGrid, mmGrid = np.meshgrid(MMRange, mmRange)

def minh(mm, MM, TT, sValues):
    H_values = [H(s, s+epsilon, mm, TT, MM) for s in sValues]
    min_H = min(H_values)
    return min_H

def calculate(mm, MM):
    return mm, MM, minh(mm, MM, TT, sValues)

resultsmM = Parallel(n_jobs=-1)(delayed(calculate)(mm, MM) for mm, MM in zip(mmGrid.ravel(), MMGrid.ravel()))
resultsmM1 = np.array(resultsmM)
minM0L_values = resultsmM1[:, 2].reshape(len(mmRange), len(MMRange))

f_interp = RegularGridInterpolator(
    (mmRange, MMRange),
    minM0L_values,
    fill_value=0  
)

mmRange1 = np.arange(0.001, 0.50, 0.001)
MMRange1 = np.arange(0, 0.78, 0.001)
M, m = np.meshgrid(MMRange1, mmRange1)
f_values = f_interp((m, M))  
condition = f_values > 0
condition_new = M > -m
combined_condition = np.logical_and(condition, condition_new)
plt.figure(figsize=(8, 6))
plt.contourf(m, M, combined_condition, levels=1, cmap='viridis') 
plt.xlabel('M')
plt.ylabel('m')
plt.show()
\end{lstlisting}

In this case, we follow the steps outlined in Section \ref{secnume} for the case where $m>0$ and $M>0$. For different pairs of values $(m,M)$, we evaluate the function $\overline{H}_{m,M}(s,s+\varepsilon)$ with integer values of $s=n$ (taking $n^-$ and $n^+$) and select the minimum. To optimize the process, we parallelize this calculation and store the triplets of values for $m$, $M$, and the corresponding minimum. Next, we interpolate a function that provides the minimum for each value of $m$ and $M$ within the range. Finally, we plot the region where this function is positive and $M>-m$, which will coincide with the area where the function $\overline{H}_{m,M}$ is positive.

\begin{lstlisting}[language=Python, caption={Code for the numerical approximation of the region where $H_{m,M}> 0$ with $m<0$ and $M>0$.}, label={lst:approx-region2}]
import numpy as np
import matplotlib.pyplot as plt
from scipy.interpolate import interp2d
import numpy as np
from scipy.interpolate import RegularGridInterpolator
from joblib import Parallel, delayed

TT = 1.6
epsilon = 10**-6 
tRange1 = np.arange(-1.6, 1.7, 0.1)
tValues = np.array([-1 - 2*epsilon, -1 - epsilon/2, -2*epsilon, epsilon/2, 
                    1 - 2*epsilon, 1 + epsilon/2])
tRange = np.sort(np.concatenate([tRange1, tValues]))
sValues = np.array([-TT, -1 - epsilon, -1 + epsilon, 0 - epsilon, 
                    0 + epsilon, 1 - epsilon, 1 + epsilon, TT])
                    
def precompute_integrals(mm, TT, tRange):
    G = lambda t, r: Gbar(t, r, mm, TT)
    integrals = []
    for t in tRange:
        integral_1, _ = quad(lambda r: G(t, r), -TT, -1)
        integral_2, _ = quad(lambda r: G(t, r), -1, 1)
        integral_3, _ = quad(lambda r: G(t, r), 1, TT)      
        integrals.append([integral_1, integral_2, integral_3])    
    return np.array(integrals)
    
def precompute_H(mm, TT, MM, tRange, sValues):
    Ha = lambda t, s: H(t, s, mm, TT,MM)
    Hvals = []
    for s in sValues:
        Hvals.append([Ha(-1, s), Ha(0, s), Ha(1, s)])    
    return np.array(Hvals)
    
def calc_M0L2(mm, MM, TT, tRange, sValues):
    integrals = precompute_integrals(mm, TT, tRange)
    Hvals = precompute_H(mm, TT, MM, tRange, sValues)
    G = lambda t, r: Gbar(t, r, mm, TT)
    M0L = np.zeros((len(sValues), len(tRange)))
    for i, s in enumerate(sValues):
        for j, t in enumerate(tRange):
            tIndex = np.where(tRange == t)[0][0]
            sIndex = np.where(sValues == s)[0][0]
            M0L[i, j] = G(t, s) / (
                Hvals[sIndex, 0] * integrals[tIndex, 0] +
                Hvals[sIndex, 1] * integrals[tIndex, 1] +
                Hvals[sIndex, 2] * integrals[tIndex, 2]
            )
    max_M0L = np.max(M0L)   
    return max_M0L

mmRange2 = np.arange(0.45, 0.79, 0.01)
MMRange2 = np.arange(-0.79, 0.02, 0.01)
MMGrid2, mmGrid2 = np.meshgrid(MMRange2, mmRange2)

def calculate_minM0L(mm, MM):
    return mm, MM, calc_M0L(mm, MM, TT, tRange, sValues)

results2 = Parallel(n_jobs=-1)(delayed(calculate_minM0L)(mm, MM) for mm, MM in zip(mmGrid2.ravel(), MMGrid2.ravel()))

results20 = np.array(results2)

minM0L_values2 = results20[:, 2].reshape(len(mmRange2), len(MMRange2))

f_interp2 = RegularGridInterpolator(
    (mmRange2, MMRange2),
    minM0L_values2,
    fill_value=np.nan  
)

mmRange22 = np.arange(0.48, 0.780, 0.0001)
MMRange22 = np.arange(-0.78, 0.01, 0.0001)
M2, m2 = np.meshgrid(MMRange22, mmRange22)

f_values2 = f_interp2((m2, M2)) 

condition2 = f_values2 > M2
condition22=M2>-m2
plt.figure(figsize=(8, 6))
plt.contourf(m2, M2, condition2,condition22, levels=1, cmap='viridis')
plt.xlabel('M')
plt.ylabel('m')
plt.show()
\end{lstlisting}

In this case, we again follow the steps outlined in Section \ref{secnume}, but now for $m<0$ and $M>0$. We create an array with integer values of $s=n$ (taking $n^-$ or $n^+$) and the extreme values of $s$ ($s=T$ and $s=-T$), as well as another array with values of $t$ in the interval $(-T,T)$, being careful at the points where $t=s$, where $\overline{H}_{m,M}$ is not well-defined. Next, for different values of $m$ and $M$ in the range of interest, we calculate the maximum of equation \eqref{maxt}. To do this, we will first create various functions that allow us to pre-calculate different elements, avoiding unnecessary repetitions. To optimize the process, we parallelize the calculation of the maximum and store the values of the triplet $m$, $M$, and the maximum. Then, we interpolate a function that, for each value of $m$ and $M$ in the considered range, gives us the maximum. Finally, we represent the region where this function is greater than $M$ and $M>-m$. We have already proved that this region  coincides with the area where the function $\overline{H}_{m,M}$ is positive.

\subsection*{Approximation of the extremal solutions with the monotone method}

Finally, we will present the code necessary to approximate the sequences $(\alpha_{n})_{n \in \mathbb{N}}$ and $(\beta_{n})_{n \in \mathbb{N}}$ associated with Problem \eqref{monopar}.

\begin{lstlisting}[language=Python, caption={Code for the extremal solutions of the monotone method.}, label={lst:approx-region2}]
from joblib import Parallel, delayed
import numpy as np
import matplotlib.pyplot as plt

n1 = 256 
deltaX = 3.2 / (2 * n1) 
n_steps = 10  
def f(t, v, w):
    return 0.2 * np.tanh(t  - v - w) + 0.21 * v + 0.2 * w

k_values = np.arange(-n1, n1 + 1)
x_coords = k_values * deltaX

def compute_H_element(k, j, deltaX):
    return H(k * deltaX, j * deltaX, 0.21, 1.6, 0.20)

H_matrix = Parallel(n_jobs=-1)(delayed(compute_H_element)(k, j, deltaX) 
                               for k in k_values for j in k_values)
H_matrix = np.array(H_matrix).reshape((2 * n1 + 1, 2 * n1 + 1))

x = np.zeros((n_steps + 1, 2 * n1 + 1))
x[0, :] = 1.6 / 2  
for n in range(1, n_steps + 1):
    for k_idx, k in enumerate(k_values):
        sum_val = 0
        for j_idx, j in enumerate(k_values):
            t = j * deltaX
            v = x[n - 1, -j_idx]
            w = x[n - 1, int(custom_floor(j_idx))]
            weight = 0.5 if (j_idx == 0 or j_idx == len(k_values) - 1) else 1.0
			sum_val += weight * H_matrix[k_idx, j_idx] * f(t, v, w)
        x[n, k_idx] = 1.6*sum_val / n1

y = np.zeros((n_steps + 1, 2 * n1 + 1))
y[0, :] = -1.6 / 2 
for n in range(1, n_steps + 1):
    for k_idx, k in enumerate(k_values):
        sum_val = 0
        for j_idx, j in enumerate(k_values):
            t = j * deltaX
            v = y[n - 1, -j_idx]
            w = y[n - 1, int(custom_floor(j_idx))]
            weight = 0.5 if (j_idx == 0 or j_idx == len(k_values) - 1) else 1.0
			sum_val += weight * H_matrix[k_idx, j_idx] * f(t, v, w)
        y[n, k_idx] = 1.6*sum_val / n1
\end{lstlisting}

To approximate the sequences $(\alpha_{n})_{n \in \mathbb{N}}$ and $(\beta_{n})_{n \in \mathbb{N}}$, we start with $\alpha_{0}=T/2$ and $\beta_{0}=-T/2$, and calculate the successive elements iteratively such that $\alpha_{n+1}=\mathcal{T}(\alpha_{n})$ and $\beta_{n+1}=\mathcal{T}(\beta_{n})$, where $\mathcal{T}$ is given by equation \eqref{aproxt}. We will approximate the integral using the trapezoidal rule, with $2n_{1}$ being the number of intermediate points and deltaX the step between each point. In this case, to obtain an accurate result, we only need to calculate $n=10$ iterations for each sequence.

To optimize the process, we will precompute the evaluations of the function $\overline{H}_{m,M}$ that we need and store them in a matrix. 

We find that by approximating the integral with the trapezoidal rule, we achieve satisfactory results. Otherwise, we could try using other methods, such as Simpson's or Monte Carlo methods.

	\section*{Acknowledgements}
	 The authors were partially supported by Grant PID2020-113275GB-I00, funded by\\ MCIN/AEI/10.13039/501100011033 and by “ERDF A way of making Europe” of the “European Union”, and by Xunta de Galicia (Spain), project ED431C 2023/12.


\bibliography{bib5} 
\bibliographystyle{spmpsciper}
\markboth{BIBLIOGRAFÍA}{}

\end{document}